\documentclass[12pt,reqno]{amsart}

\usepackage[colorlinks=true,citecolor=blue]{hyperref}
\usepackage{amsmath, amssymb, amsfonts, amsthm, mathptmx, enumerate, color, mathrsfs, soul, tabularx, xcolor}

\usepackage[utf8]{inputenc}
\usepackage[T1]{fontenc}
\usepackage{lmodern}   
\usepackage{microtype}
\usepackage{graphicx}
\usepackage{epstopdf}

\newtheorem{theorem}{Theorem}[section]

\newtheorem{corollary}[theorem]{Corollary}
\newtheorem{lemma}[theorem]{Lemma}
\newtheorem{proposition}[theorem]{Proposition}

\theoremstyle{definition}
\newtheorem{definition}[theorem]{Definition}

\newtheorem{remark}[theorem]{Remark}

\newtheorem{assump}[theorem]{Assumption}

\numberwithin{equation}{section}

\definecolor{scarletred}{RGB}{207,16,32}
\definecolor{navyblue}{HTML}{000080}
\usepackage{amsmath}
\usepackage{appendix}
\usepackage{xcolor}
\usepackage{amsfonts}
\usepackage{amssymb}
\usepackage[normalem]{ulem}
\usepackage{subfigure}    
\usepackage{caption, subcaption}
\usepackage{wrapfig}
\usepackage{graphicx}
\DeclareUnicodeCharacter{221E}{ }

\usepackage{mathtools}

\usepackage{booktabs}
\usepackage[disable]{todonotes}

\usepackage{etoolbox}
\makeatletter
\patchcmd{\@addmarginpar}{\ifodd\c@page}{\ifodd\c@page\@tempcnta\m@ne}{}{}
\makeatother
\reversemarginpar

\def\grad{\nabla}

\def\cF{\mathcal{F}}

\def\smskip{\smallskip}

\def\texitem#1{\par\smskip\noindent\hangindent 25pt
               \hbox to 25pt {\hss #1 ~}\ignorespaces}

\def\norm#1{\|#1\|}

\newcommand{\BEAS}{\begin{eqnarray*}}
\newcommand{\EEAS}{\end{eqnarray*}}
\newcommand{\BEA}{\begin{eqnarray}}
\newcommand{\EEA}{\end{eqnarray}}
\newcommand{\BEQ}{\begin{eqnarray}}
\newcommand{\EEQ}{\end{eqnarray}}
\newcommand{\BIT}{\begin{itemize}}
\newcommand{\EIT}{\end{itemize}}
\newcommand{\BNUM}{\begin{enumerate}}
\newcommand{\ENUM}{\end{enumerate}}

\newcommand{\BA}{\begin{array}}
\newcommand{\EA}{\end{array}}

\newcommand{\reals}{\mathbb{R}}
\newcommand{\integers}{\mathbb{Z}}

\newcommand{\diag}{\mathop{\bf diag}}

\def\mgrev#1{\textcolor{black}{#1}}
\definecolor{mygreen}{rgb}{0.4, 0.69, 0.2}

\newif\ifpagenumbering
\pagenumberingtrue

\pagenumberingfalse

\newcommand{\Diag}{{\mbox{Diag}}}

\newcommand{\Cml}{\mathcal{C}_\mu^L(\mathbb{R}^d)}
\newcommand{\beq}{\begin{eqnarray}}
	\newcommand{\eeq}{\end{eqnarray}}
\newcommand{\beqs}{\begin{eqnarray*}}
	\newcommand{\eeqs}{\end{eqnarray*}}
\newcommand{\E}{\mathbb{E}}
\newcommand{\R}{\mathbb{R}}

\newcommand{\mg}[1]{{\color{black} #1}}
\newcommand{\mgb}[1]{{\color{black} #1}}
\newcommand{\mgbis}[1]{{\color{black} #1}}
\newcommand{\mgc}[1]{{\color{black} #1}}
\newcommand{\ys}[1]{{\color{black} #1}}

\usepackage{hyperref}
\usepackage{marvosym}

\usepackage{cleveref}

\crefname{theorem}{theorem}{theorems}
\Crefname{theorem}{Theorem}{Theorems}

\crefname{definition}{definition}{definitions}
\Crefname{definition}{Definition}{Definitions}

\crefname{assump}{assumption}{assumptions}
\Crefname{assump}{Assumption}{Assumptions}

\def\sa#1{\textcolor{black}{#1}} 
\def\mg#1{\textcolor{black}{#1}}

\allowdisplaybreaks

\usepackage[letterpaper,margin=1in]{geometry}

\begin{document}
	\setcounter{page}{1}
	\vspace*{1.0cm}
	\title[ACCELERATED GRADIENT METHODS WITH BIASED GRADIENT
	ESTIMATES]{Accelerated Gradient Methods with Biased Gradient Estimates: Risk Sensitivity, High-Probability Guarantees,
		\\and Large Deviation Bounds}
	\author[M. G\"urb\"uzbalaban, Y. Syed, N. S. Aybat]{Mert G\"urb\"uzbalaban$^{1,3}$, Yasa Syed$^{1}$, Necdet Serhat Aybat$^{2}$}
	\maketitle
	
	\begingroup
	\renewcommand{\thefootnote}{\fnsymbol{footnote}}
	\footnotetext[1]{The authors can be contacted at:
		\href{mailto:mg1366@rutgers.edu}{mg1366@rutgers.edu} (M.~Gürbüzbalaban),
		\href{mailto:yasa.syed@rutgers.edu}{yasa.syed@rutgers.edu} (Y.~Syed),
		\href{mailto:nsa10@psu.edu}{nsa10@psu.edu} (N.~S.~Aybat).}
	\endgroup
	\setcounter{footnote}{0}
	\vspace*{-0.6cm}

	\begin{center}
		{\footnotesize
			$^1$Department of Statistics, Rutgers University, Piscataway, NJ, USA 
			\\ $^2$Department of Industrial Engineering, Pennsylvania State University, University Park, PA, USA \\
			$^3$Department of Management Science \& Information Systems, Rutgers University, Piscataway, NJ, USA 
	}\end{center}

\vskip 4mm {\small\noindent {\bf Abstract.}
	{We study trade-offs between convergence rate and robustness to gradient errors in the context 
		of first-order methods. Our focus is on generalized momentum methods (GMMs)—a broad class that includes Nesterov’s accelerated gradient, heavy-ball, and gradient descent \sa{methods}— for minimizing \sa{smooth} strongly convex objectives. We allow stochastic gradient errors that may be adversarial and biased, and quantify robustness \sa{of these methods to gradient errors} via the risk-sensitive index (RSI) from robust control theory. For quadratic objectives with i.i.d.\ Gaussian noise, we give closed form \sa{expressions} for RSI in terms of solutions to $2\times$ 2 matrix Riccati equations, 
		revealing a Pareto frontier between RSI and convergence rate over the choice of stepsize and momentum \sa{parameters}. We then prove a large-deviation principle for time-averaged suboptimality in the large iteration limit and show that the rate function is, up to \sa{a} scaling, \sa{the convex 
			conjugate of the RSI function}. We further show that the rate function and RSI are linked to the $H_\infty$-norm—a measure of robustness to \sa{the} worst-case deterministic gradient errors—so that stronger worst-case robustness (smaller $H_\infty$-norm) leads to 
		\sa{sharper decay of the tail probabilities for the average suboptimality.} 
		Beyond quadratics, under potentially biased sub-Gaussian gradient errors, we derive 
		\sa{non-asymptotic} bounds on a finite-time analogue of the RSI, yielding finite-time high-probability guarantees and non-asymptotic large-deviation bounds for \sa{the} averaged iterates. \sa{In this more general scenario of smooth strongly convex functions,} we also observe an analogous trade-off between RSI and convergence-rate bounds. To our knowledge, these are the first non-asymptotic guarantees for GMMs with biased gradients and the first risk-sensitive analysis of GMMs. Finally, we provide numerical experiments on a robust regression problem to illustrate our results.
		\looseness=-1}
	

	\section{Introduction}
	First-order optimization algorithms form the backbone of modern large-scale optimization, underpinning recent advances across a range of disciplines, including machine learning, signal processing and statistics \cite{bottou2010large,bottou2018optimization,beck2009fast,aybat2011first,aybat2012first,gurbuzbalaban2019iag}. These methods are often the preferred computational approach for obtaining low- to medium-accuracy solutions to large-scale problems, owing to their inexpensive iterations and relatively mild dependence on problem dimension and data size. 
	
	Traditional analyses of first-order methods assume exact gradient information and focus on characterizing the convergence rate of the \sa{iterate sequence using a properly defined metric}. 
	Within the family of first-order methods, gradient descent (GD) is a fundamental algorithm, widely regarded as the canonical baseline \cite{nesterov2018lectures}. Momentum-based accelerations—such as Polyak’s heavy-ball (HB) method \cite{polyakintroduction}, Nesterov’s accelerated gradient (NAG) \cite{nesterov-original}, and more recent triple-momentum methods (TMM) \cite{scoy-gmm-ieee,HuRobust,hu2017dissipativity}—offer faster convergence under convexity and smoothness assumptions when exact gradients are available. In practice, however, gradient information is rarely exact and typically contains stochastic or deterministic errors.
	
	
	Deterministic errors commonly arise from finite-precision arithmetic in numerical computations. They can also result from approximating gradients by solving auxiliary problems~\cite{daspremont,devolder2014first}, or by using \sa{approximation methods based on finite-differences} 
	\cite{scheinberg2022Finite}. Moreover, such errors frequently appear in algorithms that cycle through data points, 
	e.g., incremental gradient methods~\cite{bertsekas2002nonlinearprog,bertsekas2011incremental}. \sa{On the other hand, stochastic errors also} naturally occur in large-scale optimization, data science and machine learning settings. These arise \mgbis{commonly, including} (i) when stochastic noise is deliberately added to gradient computations to preserve data privacy, as in differentially private optimization~\cite{kuru2022diffprivaccopt}—for example, \cite{pmlr-v22-rajkumar12} explores such mechanisms in the context of empirical risk minimization; and (ii) when gradients are estimated from mini-batches, \sa{i.e., 
		using} sample averages over randomly selected data subsets, a standard practice in large-scale or streaming environments, as in stochastic gradient and stochastic approximation methods~\cite{PolyakJuditskyAcceleration,robbins1951stochastic}. Such errors often exhibit Gaussian or sub-Gaussian behavior, particularly when gradient estimates are bounded or when the batch size is sufficiently large to invoke Central Limit Theorem-type arguments \cite{bernstein2018signsgd,bernstein2019signsgd}. 
	
	Gradient errors, when they are stochastic, may be unbiased—having zero mean when conditioned on the iterates \sa{computed so far} \cite{robbins1951stochastic, nemirovski2009robust}—as in the case of uniform sampling of data points with replacement within stochastic gradient descent methods or in differentially private optimization settings where centered noise is added to the gradients for data privacy reasons \cite{bottou2018optimization,kuru2022diffprivaccopt}. However, in many practical scenarios, the errors may be biased containing a drift\ys{:} (i) when the data points are dependent or when the data points are sampled with 
	\sa{techniques} that can invoke bias such as without-replacement sampling, \ys{random reshuffling}, or biased-variance reduction mechanisms \cite{even2023stochasticgradientdescentmarkovian, shamir2016withoutreplacement, nguyen2017sarahnovelmethodmachine, driggsbiasedSG2022, tran2022nesterovacceleratedshufflinggradient, yu2025highprobabilityguaranteesrandom, Mert-RR}\ys{;} (ii) when the gradients are compressed or quantized \cite{beznosikov2024biasedcompressiondistributedlearning, stich2021errorfeedbackframeworkbetterrates, aji2017sparsecommunication, Beikmohammadi_2025, karimireddy2019errorfeedbackfixessignsgd}\ys{;} (iii) when the stochastic gradients are clipped or normalized \cite{koloskova2023revisitinggradientclippingstochastic, chen2021understandinggradientclippingprivate, mai2021stabilityconvergencestochasticgradient, xiao2023dpclippingbias}\ys{;} (iv) in the context of inexact inner solves arising in machine learning applications such as \mgbis{robust learning \cite{zhu2025mean,grz-risk}}, reinforcement learning, and meta-learning \cite{wu2020actorcriticmethods, bhandari2018finitetimeanalysistemporal, shaban2019bilevel}\ys{;} (v) finite-precision arithmetic errors or errors due to smoothing of the objective when estimating the gradients \cite{xia2023influencestochasticroundofferrors, scheinberg2022Finite}\ys{;} (vi) in adversarial learning settings when deterministic perturbations are added to the data points when computing stochastic gradients that depend on the data points \cite{madry2019deeplearningmodelsresistant, gu2023adversarialtraininggradientdescent}. Because such \emph{biased errors} accumulate over time, classical analyses that assume \emph{unbiased} stochastic noise can mispredict performance and fail to quantify robustness of the underlying first-order method to gradient errors.\looseness=-1 
	
	When gradient errors are persistent or decay slowly, the iterates may fail to converge—os\-cil\-lating around the optimal solution, deviating significantly before eventually converging, or even diverging entirely \cite{bertsekas2002nonlinearprog, bertsekas2011incremental, devolder2014first, flammarion2015averagingtoacceleration}. 
	Robustness—understood as sensitivity to gradient errors and quantified by an appropriate solution-accuracy metric—is therefore a critical performance criterion
	\cite{devolder2014first, Hardt-blog}. In particular, while momentum-based methods such as NAG, HB, and TMM are known to converge faster than \sa{the} standard GD in noise-free settings for convex or strongly convex problems, they are generally less robust to gradient errors \cite{devolder2014first, daspremont, flammarion2015averagingtoacceleration, schmidt2011convergence}. As a result, designing effective 
	\sa{algorithm} parameters requires a careful understanding of the trade-offs between convergence speed and robustness.
	
	{Existing robustness analyses of momentum methods have primarily concentrated on two extremes. \sa{On one 
			end,} a growing line of work investigates stochastic, unbiased noise, establishing mean-square or high-probability convergence guarantees under light-tailed assumptions (e.g.,~\cite{aybat2019universally, siopt-ragm, fallah2022robust, sapd, laguel2024high}). \sa{On the other 
			end,} control-theoretic approaches employ the \sa{$H_\infty$-norm} to characterize worst-case deterministic gradient errors \cite{gurbuzbalaban2023robustly}, thereby modeling the gradient as deterministically biased. However, $H_\infty$-based analyses usually ignore stochastic variability and \mgc{typically do not provide} high-probability guarantees on the iteration complexity for obtaining approximate solutions under stochastically biased gradients.}
	
	Only a handful of recent studies have addressed \emph{biased stochastic gradient errors} in the context of \sa{momentum-accelerated} methods; however, this line of research remains limited---even for convex and strongly convex optimization problems that are typically easier to analyze compared to non-convex problems. In particular, to the best of our knowledge, no existing work provides \emph{non-asymptotic high-probability} or \emph{large-deviation guarantees} for the iterates of momentum methods in the presence of both stochastic and deterministic bias, in the strongly convex setting. To bridge this gap, we focus on \sa{the} \emph{strongly convex and smooth minimization} \sa{setting,} and investigate the \emph{robustness of momentum methods} to biased gradient errors within a general three-parameter family—step size~$\alpha$ and momentum parameters~$\beta$ and~$\nu$—\sa{which is} often referred to in the literature as \emph{generalized momentum methods} (GMM), see e.g.,~\cite{can2022entropic,gurbuzbalaban2023robustly} and the references therein. This framework is broad and expressive: with appropriate parameter choices, it recovers gradient descent (GD) and popular momentum schemes, including heavy-ball (HB), Nesterov’s accelerated gradient (NAG), and the triple momentum method (TMM), as special cases.
	
	{We consider a general noise model that permits sub-Gaussian gradient errors, which may be biased.
		This unified formulation is flexible enough to capture bounded deterministic errors, stochastic unbiased errors with sub\ys{-}Gaussian tails, or combinations of both.} We propose the use of the \emph{risk-sensitive index}---a concept from robust control theory \cite{jacobson1973exponential, whittle1981risk, glover1988state}---as a robustness measure to assess the impact of bias and stochasticity on algorithmic performance. Given an iteration budget~\sa{$K\in\mathbb{Z}_+$} and a strongly convex, smooth objective function~\sa{$f:\reals^d\to\reals$}, we define the risk-sensitive cost \sa{$R_K:\reals\to\reals$} {as
		\beq
		R_K(\theta) = \frac{2\sigma^{2}}{\theta\,(K+1)}
		\log\left(\mathbb{E}\exp\left[\frac{\theta}{2\sigma^{2}}\,S_K\right]\right), \quad 
		S_K = \sum_{k=0}^{K} \left(f(x_k) - f_*\right),
		\label{eq-risk-sensitive-cost-intro}
		\eeq
		where} \sa{$f_*=\min_x f(x)$, $\sigma^2$ is the variance proxy of the sub-Gaussian noise}, 
	and $S_K$ denotes the \emph{cumulative suboptimality} of the iterates $\{x_k\}_{k=0}^K$, which depends on the choice of parameters~$(\alpha,\beta,\nu)$. We also consider its stationary analogue, known as the \emph{risk-sensitive index}, defined as $
	R(\theta) = \limsup_{K \to \infty} R_K(\theta)
	$. This metric appears in the literature for the control of dynamical systems driven by Gaussian noise \cite{jacobson1973exponential, whittle1981risk, glover1988state}. Ours is a generalization to sub-Gaussian noise settings.
	
	The parameter $\theta > 0$ \mgbis{in the cost $R_K(\theta)$} provides a continuous interpolation between the 
	\emph{risk-neutral} regime---corresponding to the limit $\theta \to 0$, 
	\sa{capturing the} average performance $\mathbb{E}\!\left[\tfrac{S_K}{K+1}\right]$---and 
	progressively more \emph{risk-averse} regimes that place increasing emphasis on tail events {for larger positive values of $\theta$, \sa{and approximating the worst-case behavior for sufficiently large $\theta$~\cite{whittle2002risk}}}. \mg{Conversely, 
		$\theta<0$ induces a risk-seeking regime that down-weights large suboptimality events and favors variability. Thus, by varying $\theta$, the risk-sensitive cost spans average-case, risk-seeking, and worst-case perspectives with a unified lens.} 
	In addition, the risk-sensitive cost $R_K(\theta)$ is a scaled \sa{cumulant-generating function, i.e., natural logarithm of the moment generating function,} and Chernoff bounds turn $R_K(\theta)$ into explicit exponential tails for suboptimality, yielding \emph{high-probability} bounds and \emph{rate functions} that characterize the decay rate for the average suboptimality $\mathbb{P}\!\left[\tfrac{S_K}{K+1} > t\right]$.
	These properties of the risk-sensitive index make it interesting to study in the context of optimization algorithms.\looseness=-1 
	
	To the best of our knowledge, this work provides the first risk-sensitive analysis of generalized momentum methods. Unlike prior research, which has been predominantly risk-neutral and has typically relied on the assumption of unbiased stochastic gradients, our framework accommodates gradient noise that may be biased, thereby delivering high-probability guarantees that remain valid under biased gradient settings. 
	
	\noindent \textbf{Contributions.} Our main contributions are built on several pillars.
		{First, we analyze the strongly convex \emph{quadratic} setting under i.i.d.~isotropic Gaussian gradient noise. This case provides a clean and tractable foundation, allowing us to develop explicit characterizations before addressing the more general class of \sa{smooth}, strongly convex functions subject to potentially biased noise. In the quadratic setting}, we investigate how the risk-sensitive index $R(\theta)$ depends on the given \sa{algorithm} parameters $(\alpha,\beta,\nu)$ and risk-aversion parameter $\theta$.
		By modeling \sa{the algorithmic class} GMM as a linear time-invariant dynamical system and leveraging some existing results from robust control theory, we argue that $R(\theta)$ admits an explicit characterization 
		\sa{through solving} a discrete-time algebraic Riccati equation (DARE) that depends on the \sa{algorithm} parameters. However, this is a matrix equation 
		\sa{over} $2d\times 2d$ matrices, and in general, it entails solving a $4d\times 4d$ generalized eigenvalue problem, which is computationally prohibitive for large $d$ requiring at least $\mathcal{O}(d^3)$ operations. To this end, we establish a dimension-reduction result: for quadratic objectives subject to i.i.d.\ \sa{isotropic} Gaussian noise with variance $\sigma^2$, we show that the DARE decomposes into $d$ independent \emph{$2\times 2$ Riccati equations}. Each such equation can be solved in $\mathcal{O}(1)$ time, allowing the overall solution to be computed in $\mathcal{O}(d)$ operations \sa{in parallel}. This yields a closed-form expression for the risk-sensitive index in terms of easily solvable $2\times 2$ DAREs. Furthermore, for gradient descent, we can solve the DAREs explicitly and provide an explicit expression for the risk-sensitive index.  
		Building on our characterization of the risk-sensitive index, we then establish a fundamental trade-off between risk (as measured by this index) and convergence rate (quantifying solution speed in the absence of noise) in momentum methods. In particular, parameter choices that slow \sa{down convergence speed also reduce the associated risk and enlarge the range of $\theta$ values} for which the risk remains finite, \sa{i.e.,} 
		improved risk sensitivity and greater flexibility in $\theta$ \sa{comes} at the expense of \sa{a reduction in the} convergence rate. Specifically, the risk is finite if and only if $\theta>0$ satisfies
		$
		\sqrt{\theta}\, H_\infty < \sqrt{d},
		$
		where $H_\infty$ \sa{denotes} the \sa{$H_\infty$-norm} of the dynamical system associated with the GMM iterations—explicitly characterized for quadratic objectives in \cite{gurbuzbalaban2023robustly}. Because the $H_\infty$-norm quantifies robustness to \sa{the worst-case deterministic} gradient errors \cite{gurbuzbalaban2023robustly}, this relation implies that greater robustness, \sa{i.e., smaller $H_\infty$-norm,}
		permits a wider range of $\theta$ for which the risk is finite, \mgbis{thereby linking risk sensitivity (a probabilistic performance measure) to deterministic worst-case robustness in momentum-based methods}. 
		\mgrev{While connections between risk sensitivity and the $H_\infty$-norm exist in the context of control systems and controlled Markov processes \cite[Sec. VI]{fleming2006controlled}, \cite{glover1988state,iglesias2000entropy}; what these connections mean in the context of momentum-based optimization algorithms subject to gradient errors is first explored in our work to our knowledge.} 
		We then compare momentum methods—HB, NAG, GD, and TMM—across different parameter settings, focusing on their risk profiles, \mgbis{i.e., how their risk sensitivity varies with stepsize and momentum.} We find that HB with standard parameters (proposed by Polyak \cite{polyakintroduction}) converges faster than NAG (with standard parameters as in \cite{nesterov2018lectures}) but incurs higher risk, whereas GD (with the common stepsize $\alpha=1/L$ where $L$ is the Lipschitz constant of the gradient) is slower yet exhibits \sa{a lower risk than both methods.} These trade-offs trace a Pareto-optimal frontier over parameter choices, delivering the best attainable risk for a target convergence rate --we illustrate this frontier numerically \mgbis{by leveraging our dimension-reduction framework to compute the risk index}.
		
		Second, we establish a large-deviation principle for the average suboptimality \(\Delta_K \coloneqq S_K/(K+1)\), quantifying the exponential decay rate of rare-event probabilities for \(\Delta_K\) \mgbis{in the quadratic setting.} \mgbis{This also allows us to control the probability of rare events associated with suboptimality of the ergodic averages of the iterates.} We show that the rate function is the Fenchel–Legendre transform of a scaled version of the risk-sensitive index \(R(\theta)\). A key challenge to establish such a result is that the summands \(f(x_k)-
		\sa{f_*}\) that appear in the sum $S_K$ are neither independent nor identically distributed—even when the \sa{gradient} noise \(\{w_{k+1}\}_{k\ge0}\) is i.i.d.)—so classical results such as Cramér’s theorem do not apply. Instead, we verify the hypotheses of the Gärtner–Ellis theorem \cite{dembo2009large}, which accommodates such non-i.i.d.\ sequences, to identify the rate function. The analysis takes advantage of our characterizations of the risk-sensitive index for momentum methods. \mgbis{These results demonstrate that the rate function and the risk-sensitive index—which govern the tail behavior of suboptimality in the large-iteration limit—are intrinsically linked to the $H_\infty$-norm. Moreover, they show that greater robustness to \sa{the} worst-case deterministic errors (as quantified by the $H_\infty$-norm) is associated with a sharper decay of tail probabilities, reflected in a larger rate function.}\looseness=-1
		
		\mgbis{In our third set of results, we extend the analysis to the minimization of strongly convex and smooth functions. In this setting, explicit characterizations of the risk-sensitive index and the $H_\infty$-norm are generally unavailable. While upper bounds $\overline{H}_\infty$ exist for the $H_\infty$-norm \cite{gurbuzbalaban2023robustly}, no analogous bounds are known for the risk-sensitive cost or the risk-sensitive index. To bridge this gap, we adopt a more general noise model in which gradient errors are sub-Gaussian and may be biased. This model subsumes the Gaussian noise studied in the quadratic case and accommodates stochastic errors with sub-Gaussian tails—including bounded stochastic errors, deterministic errors, and their mixtures}. 
		We derive upper bounds on $R_K(\theta)$ at every iteration $K$, and translate these into \emph{finite-time high-probability guarantees} and \emph{non-asymptotic large-deviation upper bounds} for the averaged iterates. \mgbis{Our approach relies on introducing a novel Lyapunov function together with a backward-induction argument, inspired by dynamic programming methods, 
			to iteratively estimate the risk-sensitive cost.} Taking the limit as \(K \to \infty\), we derive bounds on the risk-sensitive index \(R(\theta)\) and associated large-deviation upper bounds. These results can be interpreted as non-asymptotic counterparts of the asymptotic characterizations established in the quadratic case. To the best of our knowledge, such non-asymptotic high-probability and large-deviation bounds for \textit{biased} gradient estimates in the general class of momentum methods considered here have not previously appeared in the literature. Analogous to the quadratic case, we establish intrinsic trade-offs between risk and convergence rate in the design of first-order accelerated methods. Moreover, we find that the bounds on the \(H_\infty\) norm control the range of \(\theta\) values for which the risk remains finite. 
		\mgbis{Our bounds connect $R(\theta)$, high-probability guarantees, and large-deviation rates to the $H_\infty$-norm, thereby sharpening the link between probabilistic performance and worst-case robustness beyond the quadratic setting as well. They show that greater robustness to deterministic errors (as quantified by the $H_\infty$-norm) is associated with sharper tail decay in average suboptimality and reduced bias in the averaged iterates.}\looseness=-1
			\looseness=-1
			
			\ys{Fourth, we provide numerical experiments that illustrate our results. In the strongly convex quadratic setting, our  
				experiments show that a faster convergence rate corresponds to a higher risk sensitivity 
				\sa{for each momentum method considered in this paper}. We provide \sa{Pareto-optimal frontiers} for GD, NAG, HB to compare the 
				\sa{convergence rate vs. risk} trade-offs across these three methods. We also compare \sa{large-deviation} rate functions across GMMs. 
				Lastly, in the strongly convex setting, 
				\sa{we illustrate the behavior of risk-sensitive cost across iterations}  for HB, \mgbis{RS-HB (a heavy-ball variant more robust to deterministic errors~\cite{gurbuzbalaban2023robustly})}, TMM, NAG, and GD.}   

				\section{Related Work}
				\ys{
					Under a light-tailed (sub-Gaussian) noise assumption, high-probability convergence guarantees for stochastic gradient descent (SGD) and its momentum variants have been established for convex objectives \cite{nemirovski2009robust, hazan2014beyondregret, harvey2019nonsmoothSGD, armacki2024highprobabilityconvergenceboundsnonlinear, kakade2008onlinestronglyconvex, rakhlin2012makinggradientdescentoptimal, davis2019lowprobabilityhighconfidence, jain2019makingiteratesgdinformation} and for non-convex objectives \cite{ghadimi2013stochasticfirstzerothordermethods, li2020highprobabilityanalysisadaptive, armacki2024highprobabilityconvergenceboundsnonlinear}. Beyond this setting, several works relax the light-tailed assumption and instead provide guarantees under bounded variance noise 
					\cite{gorbunov2020heavytailednoiseclipping, gorbunov2024highprobabilitycomplexitybounds} or, for sub-Weibull noise \cite{li2022highprobnonconvexsgd}. 
					Unlike our setting, however, these works assume \textit{unbiased} gradient noise and do not extend to the general class of momentum methods (GMM) studied in this paper. 
					
					{\color{black}\sa{On the other hand, for particular biased gradient noise settings, there is a body of existing work that provides guarantees in expectation}, i.e., bounds on the expected suboptimality over the randomness in the noise model, as opposed to high-probability guarantees. These results are obtained under noise models different from ours, typically in more specialized settings and for particular applications. For instance, without-replace\-ment sampling (WRS) in finite-sum problems leads to biased noise, with guarantees established for SGD and Nesterov’s accelerated gradient under various WRS schemes \cite{tran2022nesterovacceleratedshufflinggradient, shamir2016withoutreplacement, Mert-RR}. Variance-reduction methods such as SAG and SARAH also yield biased gradient estimates, for which in-expectation guarantees are available \cite{driggsbiasedSG2022, nguyen2017sarahnovelmethodmachine, even2023stochasticgradientdescentmarkovian}. In distributed learning, compression and quantization introduce bias, with guarantees provided in \cite{Beikmohammadi_2025, karimireddy2019errorfeedbackfixessignsgd, beznosikov2024biasedcompressiondistributedlearning, stich2021errorfeedbackframeworkbetterrates, even2023stochasticgradientdescentmarkovian}. Gradient clipping and normalization, often used for heavy-tailed noise or differential privacy, also induce bias \cite{chen2021understandinggradientclippingprivate, xiao2023dpclippingbias, koloskova2023revisitinggradientclippingstochastic, mai2021stabilityconvergencestochasticgradient}. Other sources include low-precision arithmetic \cite{xia2023influencestochasticroundofferrors}, non-i.i.d. data in statistical learning \cite{wu2020actorcriticmethods, bhandari2018finitetimeanalysistemporal}, and truncated backpropagation strategies for bi-level optimization \cite{shaban2019bilevel}. \sa{Furthermore, there are also few papers that} 
						attempt to unify such biased-noise settings, \sa{e.g.,}  \cite{demidovich2023guidezoobiasedsgd} \sa{provides} in-expectation guarantees, \sa{and} \cite{tadipatri2025convergencestochasticheavyball} \sa{establishes} almost-sure convergence guarantees for stochastic HB with decaying stepsizes under a biased noise model for general (possibly nonconvex) objectives, where the bias term may grow proportionally with the gradient norm. 
						\mgbis{However, none of these works provides large-deviation bounds; moreover, their analyses rely on decaying step sizes converging to zero and do not apply to the constant step size considered in our work.}\looseness=-1
					}
					
					{\color{black}Compared to in-expectation guarantees, high-probability guarantees under biased noise are less developed. While in-expectation guarantees can be converted into high-probability guarantees using strategies such as longer runs of the algorithm combined with Chebyshev’s inequality, or parallel runs with selection of the best outcome \cite[p. 243]{nemirovsky1983problem}, \cite{hsu2016loss}, these strategies often produce loose complexity bounds with respect to the target probability and desired precision and can be impractical in streaming data settings \cite{pmlr-v235-gorbunov24a}, \cite[Remark 13]{laguel2024high}}. Existing high-probability results include random reshuffling for SGD \cite{Mert-RR, yu2025highprobabilityguaranteesrandom} and non-linear SGD variants with heavy-tailed noise where clipping, quantization, or normalization induce bias \cite{cutkosky2021highprobabilityboundsnonconvexstochastic,armacki2024highprobabilityconvergenceboundsnonlinear}. 
				}

			{The literature on establishing \sa{large-deviation principles (LDP)} for stochastic GD and its momentum variants is sparse. \sa{Previous 
					work on LDP has been 
					mainly on stochastic approximation algorithms,} including first- and zeroth-order stochastic gradient methods \cite{dupuis1985largedeviations, dupuis1987largedeviations, dupuis1988largedeviationsrecursivealgos,hult2025weakconvergenceapproachlarge}. However, these 
				\sa{papers} have a different focus than ours: \sa{rather than analyzing suboptimality across iterations under a constant step size, they instead} study the vanishing–step-size regime by letting the stepsize go to zero, quantifying deviations of the stochastic dynamics from their mean trajectory (which is determined by the continuous-time gradient flow) and deriving exponential estimates for escape times from neighborhoods of a point when the gradients are subject to unbiased noise or to stationary Markovian noise, \sa{i.e., they do not consider the large iteration limit for constant stepsize.}
				More recently, Bajovi\'c et al.~\cite{bajovic2022largedeviationsratesstochastic} derive 
				\sa{LDP results} with matching upper and lower bounds for stochastic GD with a decreasing step size 
				\sa{assuming} twice continuously differentiable smooth, strongly convex objectives under unbiased gradient noise. For quadratic objectives, they obtain the rate function in closed form in the large iteration limit, showing that higher-order noise moments---not just the variance---govern rare, large errors of the last iterate. \sa{On the other hand, in 
					\cite{armacki2025largedeviationupperbounds} the authors consider \textit{unbiased} heavy-tailed noise, and for SGD 
					with decaying stepsize sequence,} they establish an LDP upper bound for the minimum norm-squared of gradients \mgbis{for nonconvex, lower-bounded objectives with Lipschitz-continuous gradients}. 
				\mgbis{In the unbiased gradient setting for strongly convex minimization, \cite{ghadimi2013optimal} establishes large-deviation results for the convergence rate of the AC-SA algorithm, strengthening the earlier results of \cite{ghadimi2012optimal}. However, neither of these works considers constant step sizes or the broader class of momentum methods studied here, nor do they address \textit{biased} gradient noise or the trade-offs between convergence rate and risk sensitivity.} 
			}
			
			{\sa{Next, we are going to discuss a related body of work from the risk perspective.} Risk sensitivity, and more broadly, entropic risk, has been covered in the literature largely in the control theory context \cite{jacobson1973exponential, whittle1994risk, whittle1981risk, glover1988state}. These works consider centered, independent and identically distributed (i.i.d.) Gaussian noise. In particular, 
				\sa{
					in \cite{glover1988state}, an exact analysis of $R(\theta)$ is provided in the linear–quadratic control setting. \mgbis{These results are not given in the context of optimization algorithms but rather for general linear dynamical systems. In our work, however, we interpret GMM as a linear dynamical system when the objective is a quadratic and show that their result translates directly into a formula for the risk-sensitive cost of GMM algorithms in this setting.}
					\mgbis{However, computing the risk-sensitivity index via the approach of \cite{glover1988state} requires integrating the system’s transfer function in the frequency domain, which---using the results of \cite{iglesias2000entropy}, \cite[Remark~4.11]{peters2012minimum}, and \cite[Lemma~3.4]{stoorvogel1993discrete}---can be reduced to solving $2d \times 2d$ Riccati equations, entailing at least $\mathcal{O}(d^3)$ computational complexity.}
					By exploiting the special structure of the GMM, we show that it suffices to solve $2\times 2$ Riccati equations, which involve three unknowns—the entries of a symmetric 
					$2\times 2$ matrix—and three corresponding equations; \sa{moreover,} the approach in \cite{glover1988state} applies only to quadratics and does not 
					\sa{extend} to strongly convex objectives.}  
			}  
			More recently, \sa{in \cite{siopt-ragm} the authors} introduced robustness measures related to asymptotic optimality under unbiased gradient errors, analyzing quadratic functions and, more generally, \sa{smooth} strongly convex functions. In particular, \sa{the trade-offs between convergence rate and robustness to gradient errors for GD and NAG have been analyzed in \cite{siopt-ragm};} however, in contrast to ours,
			\sa{the approach in \cite{siopt-ragm}} does not extend to \textit{biased} gradient estimates or deterministic worst-case errors \sa{--the rate vs risk trade-off scheme introduced in \cite{siopt-ragm} for first-order optimization methods is later extended in \cite{sapd} to analyze first-order primal-dual methods for solving strongly convex-strongly concave saddle point problems subject to unbiased gradient errors with finite variance.} 
			\sa{In a different direction from \cite{siopt-ragm}, a more recent} work \cite{can2022entropic} proposes an alternative robustness measure motivated by the entropic risk of the iterates and provides high-probability bounds for GMM in the unbiased \sa{sub-Gaussian} gradient noise setting. However, their analysis does not accommodate \textit{biased} gradient estimates, focuses solely on the suboptimality of the last iterate $x_K$ rather than on ergodic averages or cumulative suboptimality $S_K$, and does not establish large-deviation bounds. Compared to prior work, we treat (i) constant step sizes (not decaying), (ii) biased sub-Gaussian gradient errors (not just unbiased), and (iii) risk-sensitive metrics (RSI/LDP) and their link to $H_\infty$-norm in the context of momentum-based optimization algorithms.

			\noindent \textbf{\sa{Notation.}} 
			\sa{Let $\mathbb{Z}_+$ and $\mathbb{R}_{+}$ denote positive integers and reals including $0$, respectively; and $\mathbb{R}_{++}=\reals_+\setminus\{0\}$.}
			A function $f:\mathbb{R}^d\to\mathbb{R}$ is called $\mu$\emph{-strongly convex} if the function $x \mapsto f(x)-\frac{1}{2}\mu \|x\|^2$ is convex on $\mathbb{R}^d$ for some constant $\mu>0$. A continuously differentiable function $f:\mathbb{R}^d\to\mathbb{R}$ is \emph{$L$-smooth} if its gradient satisfies $\|\nabla f(x) - \nabla f(y) \|\leq L \|x-y\|$ for all $x,y \in \mathbb{R}^d$. Let $\Cml$ denote the set of all functions \sa{$f:\R^d\to\R$} that are both $\mu$-strongly convex and $L$-smooth. 
			\sa{It can be shown that any $f\in\Cml$} satisfies
			\beq
			\frac{\mu}{2}\|x-y\|^2 \leq f(x) - f(y) - \nabla f(y)^\top  (y-x) \leq \frac{L}{2}\|x-y\|^2,\quad\sa{\forall~x,y\in\R^d.}
			\label{ineq-strcvx-smooth}
			\eeq
			Due to strong convexity, $f$ admits a unique global minimum on $\mathbb{R}^d$ which we will denote by $x_*$. \sa{From \eqref{ineq-strcvx-smooth}, one has $\mu\leq L$, and} we assume $\mu < L$ throughout the paper \sa{--otherwise, for $\mu=L$, the class $\Cml$ is trivial, i.e., $f(x)=\frac{\mu}{2}\norm{x}^2$.} Let $\mathsf{I}_d$ denote the $d\times d$ identity and let $0_d$ \ys{be the $d \times d$ zero matrix}. We drop the subscript $d$ in some cases 
			\sa{when} it is clear from the context. \sa{Given matrix $A$, let $A_{ij}$ denote the entry on the $i$-th row and $j$-th column of $A$. The spectral radius $\rho(A)$ of a square matrix $A$ is defined as the largest modulus of the eigenvalues of $A$, and $\|A\|$ denotes the spectral norm of $A$.} \mgc{We use $\det(A)$ and $\operatorname{trace}(A)$ to denote the determinant and the trace of a matrix $A$}. \mg{
				Throughout the paper, $A \otimes B$ denotes the Kronecker product of the matrices $A$ and $B$.} 
			\section{Preliminaries}
			Let $f\in\Cml$ be given. We consider the unconstrained 
			problem \sa{$f_*:=\min_{x\in \mathbb{R}^d} f(x)$}. The Generalized Momentum Method (GMM) \cite{can2022entropic,gurbuzbalaban2023robustly} consists of the following iterations
			\begin{subequations} \label{RBMM}
				\begin{eqnarray}
					x_{k+1}&=&x_{k}-\alpha \nabla f( y_k ) +\beta(x_{k}- x_{k-1}), \label{RBMM: Iter1}\\
					y_{k}&=&x_k+\nu (x_k- x_{k-1}), \label{RBMM: Iter2}
				\end{eqnarray}
			\end{subequations} 
			\sa{starting from an arbitrary initial point $x_0 \in \mathbb{R}^d$ with the convention that $x_{-1}=x_0$. GMM} admits three non-negative parameters $\alpha, \beta$ and $\nu$, 
			\sa{where} $\alpha>0$ is the step size, and $\beta$ and $\nu$ are called ``momentum parameters". 
			This method generalizes a number of momentum-averaging-based first-order algorithms. If $\nu=0$, GMM is equivalent to the heavy-ball method (HB) of Polyak \cite{polyakintroduction}. If we choose $\beta=\nu$, GMM recovers Nesterov's accelerated gradient (NAG) method \cite{nesterov-original}. On the other hand, when $\beta=\nu=0$, GMM reduces to the gradient descent (GD) method. {\sa{Finally}, the triple momentum method (TMM) corresponds to the parameter choice 
				$
				\alpha=\tfrac{1+\bar \rho}{L}$,
				$\beta=\tfrac{\bar\rho^2}{2-\bar\rho}$, 
				$\nu=\tfrac{\bar\rho^2}{(1+\bar\rho)(2-\bar\rho)}$,
				where $\bar\rho = 1 - 1/\sqrt{\kappa}$~\cite{scoy-gmm-ieee}}. 
			
			\ys{A growing body of work has employed control-theoretic tools to analyze the robustness of momentum-based optimization algorithms to noise \cite{siopt-ragm, gurbuzbalaban2023robustly, hu2017dissipativity, lessard2016analysis, can2019accelerated, sapd}. To the best of our knowledge, 
				our 
				\sa{work} is the first to investigate the risk sensitivity of the dynamical systems associated with GMMs, with the aim of characterizing the trade-offs between their sensitivity to gradient noise and their convergence rates.}\footnote{\mgbis{Throughout this paper, we use the terms \emph{gradient noise} and \emph{gradient errors} interchangeably.}} We 
			\sa{study} the performance of GMM algorithms subject to additive gradient noise, \sa{i.e.,} at each iteration $k\geq 0$, instead of \sa{employing} the actual gradient $\nabla f(y_k)$, 
			\sa{suppose we only} have access to its noisy \sa{estimate:} 
			\beq \tilde \nabla f(y_k, w_{\ys{k+1}}) \sa{:=} \nabla f(y_k) + w_{\ys{k+1}},
			\label{eq-gradient-noise}
			\eeq where $w_{k+1} \in \mathbb{R}^d$ is the additive gradient noise \sa{at $y_k$}. Assumptions on the exact 
			\sa{nature of this noise will be provided later.} We start with presenting the dynamical system reformulation of \sa{\eqref{RBMM}} as a discrete-time dynamical system:
			\begin{subequations}
				\label{Sys: RBMM}
				\begin{align} 
					&\xi_{k+1}=A \xi_{k}+Bu_{\ys{k+1}},\\
					&y_k = C \xi_k, \quad u_{\ys{k+1}} = \tilde \nabla f(y_k, w_{\ys{k+1}})  = \nabla f(y_k) + w_{\ys{k+1}},
				\end{align}
			\end{subequations}
			where $A, B$ and $C$ are system matrices of the \sa{form:} \beq A = \tilde{A}\otimes \mathsf{I}_d,  \quad B = \tilde{B}\otimes \mathsf{I}_d, \quad C = \tilde{C}\otimes \mathsf{I}_d, 
			\label{def-ABC}
			\eeq
			with
			\begin{align}\label{def: system mat for TMM}
				\tilde{A}:=\begin{bmatrix} 
					(1+\beta) & -\beta  \\ 
					1 & 0
				\end{bmatrix} , \;\; 
				\tilde{B}:=\begin{bmatrix} 
					-\alpha \\ 
					0
				\end{bmatrix},\;\; 
				\tilde{C}:=\begin{bmatrix}
					(1+\nu) & -\nu 
				\end{bmatrix},
			\end{align} 
			and
			\begin{equation}\xi_k := \begin{bmatrix}  x_k  \\
					x_{k-1} 
				\end{bmatrix}
				\label{def-ksi-k},\
			\end{equation}
			is the \emph{state vector} which 
			\sa{is defined by the iterates} at step $k$ and the previous iterate \sa{at $k-1$}, \sa{i.e., $x_k$ and $x_{k-1}$, respectively.}  We also define 
			\begin{equation}\xi_* := \begin{bmatrix}  x_*  \\
					x_* 
				\end{bmatrix}
				\label{def-ksi-star},\
			\end{equation}
			which consists of two copies of the global optimum $x_*$ \sa{concatenated vertically}. The vector \sa{$u_{k+1}$} is the noisy gradient evaluated at the point $y_k$.
			Notice that we can rewrite the noisy GMM iterations \eqref{Sys: RBMM} as\begin{subequations}
				\label{Sys:stoc-RBMM}
				\begin{eqnarray} 
					\xi_{k+1}&=&A \xi_{k}+B\nabla f(C\xi_k) +  B w_{\ys{k+1}},\\
					z_k &=& F(\xi_k),
				\end{eqnarray}
			\end{subequations}
			where $z_k$ is the output of the system that is of interest (defined through an arbitrary map $F:\mathbb{R}^{2d} \to \mathbb{R}^{ \ys{p}}$ for some $\ys{p}\geq 1$ and the matrices $A,B$ and $C$ are as in \eqref{def: system mat for TMM}. Various choices of $F$ are admissible, \sa{e.g.,} {
				setting
				\beq
				F(\xi_k) := \overline{w}\!\big(f(x_k) - f(x_*)\big),
				\label{def-F-output-seq-generator}
				\eeq
				for an increasing scalar function $\overline{w}: \mathbb{R}_+ \to \mathbb{R}_+$ yields a functional that quantifies a weighted measure of suboptimality.}
			
			\mg{Previous research on the robustness of momentum methods against gradient noise predominantly considered \sa{either} \sa{\textit{stochastic}} unbiased noise that is 
				\sa{assumed to be} Gaussian or light-tailed 
				\cite{aybat2019universally, siopt-ragm,  can2022entropic, mohammadirobustness2021, van2021speed, sapd, zhang2022sapd+},
				or deterministic worst-case noise \cite{gurbuzbalaban2023robustly, devolder2013exactness, daspremont}. \ys{The existing research on \textit{biased} stochastic gradient methods with momentum to our knowledge is quite limited}.
				In this study, we expand this scope \mgb{by considering \ys{the case when} \mgbis{the gradient noise} is light-tailed but is not necessarily \ys{un}biased}. We 
				\sa{first begin} our discussion by introducing the necessary filtrations.}
			
			\mg{GMM generates a sequence of iterates \sa{$\{x_k\}_{k\geq 0}$}, each situated within a defined \sa{filtered} probability space $(\Omega, \mathcal{F}, \mathbb{P})$ \sa{such that $\mathcal{F}_k\subset \cF$ is defined as the $\sigma$-algebra generated by the iterates $x_0, x_1, \dots, x_k$ for all $k\geq 0$.}}
			\ys{The following assumption says that the noise $w_{k+1}$ at step 
				\sa{$k\in\integers_+$} has light tails; \sa{but, it} is not necessarily unbiased.}
			\begin{assump}\label[assump]{assump-exp-moment-det-additive} 
					\sa{Let $\{w_{k+1}\}_{k\geq0}$ be an adapted process, i.e., $w_{k+1}$ is $\cF_{k+1}$-measurable for $k\geq 0$.} 
					The norm $\|w_{k+1}\|$ is sub-Gaussian conditional on $\mathcal{F}_k$ with variance proxy $\sigma^2$, i.e., \sa{there exists some $\sigma>0$ such that for all $k\geq 0$ it holds that}
					$$\mathbb{P}(\ys{\|w_{k+1}\|} \geq t | \mathcal{F}_k) \leq 2 e^{-\frac{t^2}{2\sigma^2}}, \qquad 
					\forall~t\geq 0.$$
				\end{assump}
				\begin{remark}\label[remark]{remark-doobs}
					\mgb{Note that in \cref{assump-exp-moment-det-additive}, the noise sequence can be \emph{biased}; \sa{i.e.,} $\mathbb{E}[\ys{w_{k+1}} | \mathcal{F}_k]$\ $ = 0$ does \sa{\textit{not}} necessarily hold. For example, when \ys{$w_{k+1}= \tilde \zeta_{k+1} + \tilde{\delta}_k$} where \ys{$\tilde{\zeta}_{k+1}$} is i.i.d. sub-Gaussian, independent from $\mathcal{F}_k$ with variance proxy ${\tilde\sigma}^2$ and mean $\mathbb{E}[\tilde{\zeta}_{k+1} | \mathcal{F}_{k}]= 0$ and $\tilde{\delta}_k$ is a bounded deterministic sequence with the bound $\|\tilde{\delta}_k\|^2\leq M^2$ for some \sa{$M\geq 0$}; then, it is easy to see that \cref{assump-exp-moment-det-additive} holds for \ys{$\sigma^2 = \bar c_0(\tilde{\sigma}^2 + M^2)$, for some large enough constant $\bar c_0 > 0$,}
						where we did not attempt to optimize the choice of $\sigma$. \mgbis{As a special case, if $w := \{w_{k+1}\}_{k\geq 0} \in \ell_2(\mathbb{R}^d)$—that is, if $\{w_{k+1}\}_{k\geq 0}$ is a deterministic square-summable sequence satisfying $\sum_{k\geq 0} \|w_k\|^2 < \infty$—then such a finite $M$ exists and \cref{assump-exp-moment-det-additive} is satisfied.}
					}
				\end{remark}
			Assumption~\ref{assump-exp-moment-det-additive} encompasses a variety of noise models relevant to practical scenarios. 
			The \emph{centered (unbiased) stochastic noise} is commonly encountered in differential privacy, where i.i.d. centered noise is added to gradients to protect user data~\cite{kuru2022diffprivaccopt}. A similar situation arises in stochastic optimization and statistical learning contexts when the objective takes the form $f(x) \coloneqq \mathbb{E}_\xi[\mathcal{L}(x,\xi)]$ for a loss function $\mathcal{L}$, and the expectation is approximated via empirical averaging over mini-batches \cite{bottou2018optimization,bottou2010large}. Under moderate batch sizes, the Central Limit Theorem justifies modeling the resulting gradient noise as Gaussian, a behavior also observed empirically~\cite{bernstein2018signsgd,bernstein2019signsgd}. Consequently, Assumption~\ref{assump-exp-moment-det-additive} is often satisfied in practice. Light-tail conditions such as this are standard in the analysis of stochastic optimization algorithms, particularly for deriving tail bounds and high-probability guarantees. On the other hand, 
			\ys{it is possible 
				\sa{that the gradient noise is \emph{deterministic}, i.e., it contains no stochastic element.}} This situation is typical in incremental gradient methods or in settings where gradients are approximated by solving subproblems inexactly~\cite{daspremont,devolder2013exactness,zhang2022sapd+,gurbuz:2014,gurbuzbalaban2017convergence,gurbuzbalaban2019iag}. Our noise model also captures \emph{deterministic corruption} of gradients at each iteration, which may arise due to low-precision floating-point arithmetic~\cite{gd-with-precision-error}.

			\subsection{Risk-sensitive cost and the risk-sensitive index}

				Let $f\in \Cml$ and consider the iterates $\{\xi_k\}_{k\geq 0}$ of the GMM method \sa{in \eqref{Sys:stoc-RBMM} subject to gradient noise $\{w_{k+1}\}_{k\geq 0}$.} Under Assumption \ref{assump-exp-moment-det-additive}, for a fixed iteration budget $K > 0$, \sa{we define our \textit{risk-sensitive cost} as in \eqref{eq-risk-sensitive-cost-intro}}, i.e.,   
				\begin{equation}\label{def-risk-cost}
					R_{K}(\theta) \coloneqq  \frac{\mgb{\ys{2}\sigma^2}}{ \mgb{(K+1)} \theta}
					\log \mathbb{E} \left[e^{\frac{\theta}{\mgb{2\sigma^2}} S_K} \right], \quad  S_K:= \sum_{k=0}^{K} \left[ f(x_k) - \sa{f_*}\right],
				\end{equation}
				for a real scalar parameter $\theta> 0$, where the expectation is taken with respect to the entire sequence of noise encountered until $K$, assuming that the expectation is finite.  
				{This type of 
					\sa{risk measure} appears frequently in the robust control literature \cite{jacobson1973exponential, whittle1981risk, glover1988state, shaiju2008FormulasFD} \sa{for different choices of $S_K$ depending on the dynamics of interest --here,}} we take the cost function $S_K$ to be cumulative suboptimality over the iterations. \mg{By considering the Taylor series expansion of $R_{K}(\theta)$ with respect to $\theta$, it follows that as $\theta \to \sa{0^+}$, the following limit exists:
					\beq\label{def-risk-param-zero} 
					R_{K}(0):= \lim_{\theta\to \sa{0^+}}R_{K}(\theta)= \E \left(\frac{S_K}{\mgb{K+1}}\right),
					\eeq
					provided that \sa{there exists some $\bar\theta>0$ such that $R_{K}(\theta)$ is finite for $\theta\in(0, \bar\theta)$.} 
					The parameter $\theta$ is called the \emph{risk-sensitivity parameter} \cite{whittle1994risk,whittle2002risk}. The term $S_K$ is a running sum of suboptimality over $K$ iterations; therefore, we see from \eqref{def-risk-param-zero} that when $\theta=0$, the risk-sensitive cost is a measure of average performance $S_K/(K+1)$ \sa{in expectation}. For $\theta>0$, as $\theta$ gets larger, the risk-sensitive cost $R_{K}(\theta)$ penalizes large values of $S_K$ more and more due to the exponential term in \eqref{def-risk-cost} and $R_{K}$ will be more sensitive to large values of 
					suboptimality. Therefore, in analogy with a risk-averse financial investor that tries to avoid large losses, the case $\theta>0$ is called the ``risk-averse case" (see e.g., \cite{glover1988state,whittle1994risk}).
					This can also be readily seen from the Taylor expansion of \eqref{def-risk-cost} around $\theta=0$, \sa{i.e.,} 
					\begin{equation}\label{test}
						R_{K}(\theta) =  R_{K}(0) + \frac{\theta \mgb{(K+1)} }{\ys{4}\sigma^2} \mgb{\mbox{Var} \left(\frac{S_K}{\mgb{K+1}} \right)} + \mathcal{O}(\theta^2). 
					\end{equation}
					We observe that as $\theta>0$ gets larger, \sa{the second term on the right-hand side of~\eqref{test} penalizes 
						\mgb{the variance} of the average suboptimality $S_K/\mgb{(K+1)}$ more and more --note that the variance does not play a role within the $\theta=0$ case.}
					The \textit{risk-sensitive index} is then defined as
					\[
					R(\theta) \coloneqq \mg{\limsup}_{K \to \infty} R_{K}(\theta),
					\]
					\mg{provided \sa{that} the limit exists}. Risk-sensitive cost is well-studied in the robust control literature and is used to quantify the robustness of dynamical systems to stochastic noise \cite{whittle2002risk, DupuisRobustProp, fleming1995risk, pham-risk-sensitive-control}. \ys{In some of our results, we will also consider the $\theta < 0$ case, \sa{i.e., the risk-seeking scenario}.}
				}
					\subsection{$H_\infty$-norm.} {
						By \sa{appropriately} choosing the output map $F(\cdot)$ in~\eqref{Sys:stoc-RBMM}, we \sa{can} obtain different output \sa{sequence $\{z_k\}$ that is relevant to the quantity we would like to analyze for GMM.} In particular, if $F$ is chosen such that  
						\begin{equation}
							\|F(\xi_k)\|^2 = f(x_k) - \sa{f_*},
							\label{eq-def-zk}
						\end{equation}
						for all $k \geq 0$, then the cumulative cost $
						S_K = \sum_{k=0}^K \big(f(x_k) - \sa{f_*}\big),
						$
						can be expressed in terms of the $\ell_2$ norm of the sequence $\{z_k\}$, \sa{i.e., since $z_k=F(\xi_k)$ for $k\geq 0$, we have}
						$
						S_K = \sum_{k=0}^K \|z_k\|^2$ \sa{--for instance,} the condition \eqref{eq-def-zk} is satisfied when the weighting function in \eqref{def-F-output-seq-generator} is chosen as $\overline{w}(z) = \sqrt{z}$. In this case,  when the noise $w_{k+1}$ is deterministic for every \sa{$k\geq 0$} and is square-summable; from the limit \eqref{def-risk-param-zero}, we see that $R_{K}(0)=\frac{S_K}{\mgb{K+1}}$ is also deterministic and corresponds to the average suboptimality. 
						\mgbis{Given the decaying nature of square-summable noise}, a natural relevant question would be whether the \mg{deterministic} limit $S_\infty:=\lim_{K\to\infty}S_K=\sum_{k=0}^\infty\|z_k\|^2$, \mgbis{which captures the cumulative suboptimality over the iterations},
						\sa{exists} \sa{--with the convention that the limit is $+\infty$ if $S_K$ is not bounded.}} 
					
					{The $H_\infty$-norm of a dynamical system, \mgbis{defined below}, provides tight bounds on the quantity \sa{$S_\infty=\lim_{K\to\infty}\sum_{k=0}^K \big(f(x_k) - f_*\big)$} (see \cite{fleming1995risk,glover1988state}), \mgbis{and consequently on the average suboptimality $R_K(0)$ 
							\sa{for $K\in\integers_+$ sufficiently large;} thereby yielding tight bounds for the suboptimality of the averaged iterates \cite{gurbuzbalaban2023robustly}}. Roughly speaking, $H_\infty$-norm is a measure of how much a dynamical system amplifies deterministic noise 
						\sa{in terms of the $\ell_2$-norm} (from input noise sequence to output sequence $z_k$), and is formally defined as follows.\looseness=-1
						\begin{definition}[\cite{tran2017qualitative, lin1996h, van2016l2,fleming1995risk,gurbuzbalaban2023robustly}]\label[definition]{def-Hinfty}
							Consider the purely deterministic setting where the gradient error sequence $w:=\{w_{k+1}\}_{k\geq 0}$ is \textit{deterministic} and \textit{square-summable}.  
							The \sa{GMM} system~\eqref{Sys:stoc-RBMM}, with input sequence (gradient noise) $\{w_{k+1}\}_{k\geq 0} \in \ell_2(\mathbb{R}^d)$ and output sequence $\{z_k\}_{k\geq 0}$ satisfying~\eqref{eq-def-zk}, has $\ell_2$-gain at most $\gamma$ if there exists a function 
							$
							H_\gamma : \mathbb{R}^{2d} \to \mathbb{R}_+,
							$
							with $H_\gamma(\xi_*) = 0$ \sa{and} such that for \textit{every initialization} $\xi_0 \in \mathbb{R}^{2d}$ and every input $w \in \ell_2(\mathbb{R}^d)$, 
							\sa{it holds that}
							\begin{equation}
								\sum_{k\geq 0} \big(f(x_k) - 
								\sa{f_*} \big) 
								= \sum_{k\geq 0} \|z_k\|^2 
								\;\leq\; \gamma^2 \sum_{k\geq 0} \|w_{k+1}\|^2 + H_\gamma(\xi_0).
								\label{def-l2-gain-general}
							\end{equation}
							\mg{In this context, the function $H_\gamma$ is independent of both the disturbance sequence $\{w_{k+1}\}_{k\geq 0}$ and the initialization $\xi_0$, while it may still exhibit dependence on $\gamma$.
							}  
							The $H_\infty$-norm of the \sa{GMM} system is then defined as the minimal such bound, i.e.,
							\begin{equation}
								H_\infty \coloneqq 
								\inf \Bigl\{ \gamma \in \mathbb{R}_{++} \;\Big|\;
								\exists\, H_\gamma \;\text{such that}\;
								H_\gamma(\xi_*) = 0 
								\;\text{and}\;
								\eqref{def-l2-gain-general} \;\text{holds for all}\; 
								w \in \ell_2(\mathbb{R}^d) 
								\Bigr\}.
							\end{equation}
							This quantity is also referred to as the \emph{minimal $\ell_2$-gain} of the system~\eqref{Sys:stoc-RBMM} with output $\{z_k\}_{k\geq0}$.
						\end{definition}
					}

						\mgbis{\sa{The requirement that $H_\gamma(\xi_*)=0$ ensures \eqref{def-l2-gain-general} providing tight bounds} in the absence of noise. Specifically, if \sa{$w_{k+1}=0$ for all $k\geq 0$} and $\xi_0=\xi_*$, then $x_0=x_*$ and \sa{since $x_*$ is a fixed point of the GMM iterations, this initialization} implies $\sum_{k\geq 0} \big(f(x_k)-
							\sa{f_*}\big)=0$. In this case, the inequality \eqref{def-l2-gain-general} reduces to an equality with both sides equal to zero. 
							\sa{For} linear systems, there exist alternative yet equivalent definitions of the $H_\infty$-norm, either based on initialization at a fixed point $\xi_0=\xi_*$ or 
							\sa{stated} in the frequency domain \cite{zhou1996robust}. In this work, we adopt a more general definition that applies to both linear and nonlinear systems, since the GMM system can be nonlinear for $f \in \Cml$ \sa{that is not a quadratic function.}\looseness=-1}

						The $H_\infty$-norm quantifies the worst-case \sa{deterministic noise} amplification of the system in~\eqref{Sys:stoc-RBMM}, \sa{i.e., it quantifies how the output behaves in response to} the noise input $\{w_{k+1}\}_{k\geq 0}$. A smaller $H_\infty$ value indicates a more robust system, as the output is less sensitive to worst-case deterministic square-summable input noise; consequently, the cumulative suboptimality over the iterations, i.e., $\sum_{k\geq 0} \big(f(x_k)-
						\sa{f_*}\big)$, will also be smaller for the same noise level.
						For GMM, $H_\infty$-norm depends on the 
						\sa{algorithm} parameters $(\alpha,\beta,\nu)$ and the \sa{problem parameters $L$ and $\mu$, i.e., smoothness and convexity constants, respectively}. For strongly convex quadratics, it is known that $H_\infty$-norm can be computed exactly. The following result is a direct consequence of \cite[Thm.~4.1]{gurbuzbalaban2023robustly} together with \cite[eqns.~(60)--(61)]{gurbuzbalaban2023robustly}. 
					}
					
					
					\begin{theorem}[\cite{gurbuzbalaban2023robustly}]\label[theorem]{thm-h-inf}
						\mg{Consider a quadratic function $f \in \mathcal{C}_\mu^L(\mathbb{R}^d)$ \mg{with a Hessian matrix $Q\in\mathbb{R}^{d\times d}$}. 
							\sa{Suppose that with the given parameters $(\alpha, \beta, \nu)$, 
								the GMM dynamics exhibit} global linear convergence \sa{to the unique fixed point of the system $\xi_*$} in the absence of gradient noise, 
							specifically when $w_{k+1} = 0$  in \eqref{Sys:stoc-RBMM} \sa{for all $k\geq 0$}. Under these conditions, the $H_\infty$-norm (
							\sa{as given} in \cref{def-Hinfty}) of the GMM is} 
						\begin{equation}\label{hinfty-quad}
							H_\infty
							=  \mg{\frac{\alpha}{\sqrt{2}} \max_{1\leq i \leq d}  \frac{ \sqrt{\lambda_i}}{\tilde{s}_{\lambda_i}}}	= \frac{\alpha}{\sqrt{2}} \cdot
							\max_{\lambda \in \{ \mu, L\} }  \frac{ \sqrt{\lambda}}{{\tilde{s}_\lambda}}\,,
						\end{equation}
						where \mg{$\mu= \lambda_1 \leq \lambda_2 \leq \dots \lambda_d = L$ are the eigenvalues of $Q$ in a non-decreasing order}; \sa{moreover,} 
						\beq {\tilde{s}_\lambda}: = \begin{cases} | 1-\tilde{c}_\lambda | ~ \sqrt{1 - \frac{{\tilde{b}_\lambda}^2}{4{\tilde{c}_\lambda}}} & \mbox{if   } {\tilde{c}_\lambda}>0 \mbox{ and } \frac{|{\tilde{b}_\lambda}| (1+{\tilde{c}_\lambda})}{4{\tilde{c}_\lambda}} < 1,\nonumber \\
							\big| |1+{\tilde{c}_\lambda}| - |{\tilde{b}_\lambda}| \big| & \mbox{otherwise},
						\end{cases} 
						\label{def-r-lambda}	
						\eeq	  
						with 	
						\beq \tilde{b}_\lambda :=\alpha\lambda (1+\nu)-(1+\beta), \quad {\tilde{c}_\lambda} := \beta - \alpha \lambda \nu.
						\label{def-b-c-lambda}
						\eeq
						\sa{Finally, it always holds that} $\mg{H_\infty}\geq \frac{1}{\sqrt{2\mu}}$.
					\end{theorem}
					For general strongly convex functions that are not quadratic, exact analytical expressions are typically unavailable. Nevertheless, tight upper bounds have been derived in~\cite{gurbuzbalaban2023robustly} \mgbis{for GD, NAG, and more generally for GMM,} by verifying the feasibility of certain $4 \times 4$ matrix inequalities. For completeness, we present these bounds in \cref{thm-hinfty-bound} \ys{in the appendix}. In the context of stochastic dynamical systems driven by Gaussian noise, it is known that risk sensitivity is closely related to the $H_\infty$-norm: specifically, the largest value of $\theta > 0$ for which the risk-sensitive index $R(\theta)$ remains finite depends on the system's $H_\infty$-norm~\cite{fleming1995risk,glover1988state}. To the best of our knowledge, however, such connections have not been explored in the setting of optimization algorithms driven by noise models such as those described in \cref{assump-exp-moment-det-additive}. To investigate this connection further and gain insight via precise analytical formulas, the next section focuses on the special case of \sa{strongly convex} quadratic objectives subject to \ys{unbiased (centered) i.i.d. Gaussian noise}. This serves as a foundation before extending the analysis to general \sa{smooth} strongly convex functions under the broader noise model of \cref{assump-exp-moment-det-additive}.
					
					
					\section{Exact analysis of $R(\theta)$ for strongly convex quadratics}\label{sec-quad}
					We now consider the special case 
					\sa{with $f$ being} a strongly convex quadratic 
					of the form
					\beq\label{eq-quad}
					f(x):= \frac{1}{2}x^\top Q x + \mg{g}^\top  x + \mg{h},
					\eeq    
					where $Q\in \mathbb{R}^{d\times d}$ is a \sa{symmetric} positive-definite matrix, $\mg{g}\in \mathbb{R}^d$ 
					and $\mg{h}\in\mathbb{R}$. In this case, the system defined in \eqref{Sys:stoc-RBMM} reduces to a linear dynamical system \mg{which can be studied further \sa{by providing some closed form expressions about its dynamics}.} 
					First, we will rewrite these iterations as a linear dynamical system. For this purpose, we consider the eigenvalue decomposition \begin{equation} Q = U \Lambda U^\top, 
						\label{def-U-matrix}
					\end{equation} 
					where \sa{$\Lambda=\diag(\lambda)$} is a diagonal matrix containing eigenvalues of $Q$ in increasing order, i.e., $\sa{\Lambda_{ii}} = \lambda_i$ \sa{for $i=1,\ldots,d$ such that}
					$\sa{0<}\mu = \lambda_1 \leq \lambda_2 \leq \dots \leq \lambda_d = L$ are the eigenvalues of $Q$. 
					\sa{Note that}
					\begin{equation}  f(x_k) - f(x_*) = \frac{1}{2} (x_k - x_*)^\top  Q(x_k - x_*) =  \frac{1}{2} (x_k - x_*)^\top  (U\Lambda U^\top ) (x_k - x_*);
						\label{eq-f-with-eigdeg}
					\end{equation}
					\sa{hence, we have $\grad f(x)=Q(x-x_*)$, implying that}
					\beq 
					\nabla f(C\xi_k) = Q(C\xi_k - x_*) = QC \xi^c_k, \quad \mbox{where} \quad \xi_k^c:= \xi_k - \xi_* \quad \mbox{with} \quad \xi_* = \begin{bmatrix} x_* \\ x_*\end{bmatrix}.
					\label{def-centered-iter}
					\eeq
					Here, the superscript $``c"$ is to highlight that these iterates are shifted to be ``centered" around the \sa{optimal point $x_*$}, i.e., if GMM \sa{iterate sequence} converges to \sa{$x_*$}, 
					\sa{then} $\xi_k^c\to 0$. 
					
					Based on \eqref{def-centered-iter}, we can rewrite \eqref{Sys:stoc-RBMM} as	
					\begin{eqnarray} \label{Sys:quad-deterministic-noise}
						{\xi}^c_{k+1} &=&A_Q {\xi}^c_{k} +  Bw_{\ys{k+1}},
					\end{eqnarray}
					\sa{where $A_Q:= A+BQC$; therefore, using the definitions of $A,B$ and $C$ given in \eqref{def-ABC}, we have}
					\beq 
					A_Q:=\begin{bmatrix} 
						(1+\beta)\mathsf{I}_d -\alpha(1+\nu)Q & -\beta \mathsf{I}_d +\alpha\nu Q \\
						\mathsf{I}_d & 0_d 
					\end{bmatrix}.
					\label{def-AQ}
					\eeq
					\sa{For this setting, we set the map $F$ in \eqref{Sys:stoc-RBMM} 
						such that $F(\xi_k) \coloneqq T \xi_k^c$, this leads to the output}
					\beq
					z_k = T \xi_k^c  \quad \mbox{with} \quad T \sa{:=}\begin{bmatrix} \frac{1}{\sqrt{2}} \Lambda^{1/2} U^\top  \quad  0_d \end{bmatrix};
					\label{def-T}
					\eeq
					\sa{and, from \eqref{eq-f-with-eigdeg}, it can be seen that $z_k$ satisfies} the identity \eqref{eq-def-zk}. {Therefore, based on \cref{def-l2-gain-general}, we can talk about the $H_\infty$-norm of the GMM system and the $H_\infty$ formula from \cref{thm-h-inf} is applicable. {The $H_\infty$-norm characterizes robustness in the hypothetical setting where the noise is fully deterministic. In \cref{assump-exp-moment-det-additive}, we consider \sa{a more general noise model}; however, if the noise specified therein were further restricted to be deterministic and square-summable, the $H_\infty$-norm would quantify robustness with respect to the worst-case impact of such noise.}}
					
					The GMM dynamics, as defined by \eqref{Sys:quad-deterministic-noise}--\eqref{def-T}, constitute a linear dynamical system identified by the system matrices $(A_Q, B, T)$. In this context, in the special case when the noise sequence \ys{$\{w_{k+1}\}_{k\geq 0}$ \sa{consists of i.i.d. 
							multivariate Gaussian vectors such that $w_{k+1} \sim \mathcal N(0,\ys{\sa{\frac{\sigma^2}{d}{\mathsf{I}_d}}})$ for some $\sigma>0$}}, then the noise variance is $\sigma^2$ and \cref{assump-exp-moment-det-additive} is satisfied. In this case, it is known that the risk-sensitive index $R(\theta)$ is finite if and only if $\sqrt{\theta}H_\infty <\mgb{\sqrt{d}}$ \cite{glover1988state}; \sa{indeed, a particular form of a discrete-time algebraic Riccati equation (DARE), i.e., 
						\begin{equation}
							X = A_Q^\top  X A_Q + A_Q^\top  X B (\gamma^2 \mathsf{I}_d - B^\top  X B)^{-1} B^\top  X A_Q + T^\top T, \label{dare}
						\end{equation}
						admits a positive semi-definite solution $X \succeq 0 $  if and only if $ \gamma> H_\infty $ \cite[Corollary 2.1, Prop. 3.7]{hinrichsen1991stability} --thus, $\sqrt{\theta}H_\infty <\sqrt{d}$ holds if and only if \eqref{dare} has a positive semi-definite solution for the parameter choice $\gamma=\frac{1}{\sqrt{\theta/d}}$} (see \cite[Lemma 3.4]{stoorvogel1993discrete} and \cite{glover1988state}).
					\mg{
						To compute a solution $X \in \mathbb{R}^{2d \times 2d}$ for \sa{the DARE above}, a standard method involves tackling a $4d \times 4d$ generalized eigenvalue problem. More specifically, it necessitates the calculation of the generalized eigenvalues and eigenvectors of the following $4d \times 4d$ matrix pencil:}
					\beq M(\tilde\lambda) = \tilde\lambda \begin{bmatrix}
						\mathsf{I}_\mgbis{2d}  & \sa{- BB^\top/\gamma^2}  \\
						0_\mgbis{2d}  &  A_Q^\top  
					\end{bmatrix} -
					\begin{bmatrix}
						A_Q & 0_\mgbis{2d}\\
						-T^\top T & \mathsf{I}_\mgbis{2d}
					\end{bmatrix},
					\label{def-matrix-pencil}
					\eeq
					\mg{i.e., \sa{one needs to compute 
							a pair $(\tilde{\lambda},\tilde v)$ such that} $M(\tilde\lambda) \tilde{v} = 0$, and then a solution $X$ to the DARE \sa{in~\eqref{dare}} can be expressed in terms of $\tilde{\lambda}$ and $\tilde{v}$ (see, \sa{e.g.,} \cite[eqn. (6)-(8)]{arnold1984generalized}). This can be computationally expensive when $d$ is large, requiring at least $\mathcal{O}(d^3)$ operations. However, in the case of GMM methods, the matrices $A_Q$, $B$, and $T$ have some special block structure. In the next result, we exploit this structure and provide an explicit formula for the risk-sensitive index $R(\theta)$ in terms of solutions to $2\times 2$ DARE 
						\sa{systems} when the input noise is i.i.d. Gaussian with covariance $\ys{\frac{\sigma^2}{d}} \mathsf{I}_d$ for some $\sigma>0$. With this scaling of the noise covariance matrix, the noise variance is $\sigma^2$ and \cref{assump-exp-moment-det-additive} is satisfied with variance (proxy) $\sigma^2$.}  
					\mgbis{The proof is deferred to the appendix. To establish this result, we first apply a scaling argument by modifying the $B$ matrix in \eqref{dare} to $B\sigma$. We then exploit the block-diagonalizability of the matrix $A_Q$ into $2 \times 2$ blocks, together with the uniform block structure of $B$. The main difficulty lies in proving the existence of solutions to the reduced DARE system and in quantifying how these solutions can be mapped to those of the full $2d \times 2d$ system via suitable permutation matrices and linear transformations.}
					\begin{theorem}\label[theorem]{thm-gmm-risk-formula} Let $f$ be a quadratic of the form \eqref{eq-quad}. \ys{Assume that the gradient noise follows an isotropic Gaussian distribution, i.e., Assumption \ref{assump-exp-moment-det-additive} holds for an i.i.d. sequence $\{w_{k+1}\}_{k\geq 0}$ with $w_{k+1} \sim \mathcal N(0,\ys{\sa{\frac{\sigma^2}{d}\mgb{\mathsf{I}_d}}})$ for some $\sigma>0$}. Let the parameters $(\alpha, \beta, \nu)$ of GMM be such that $\rho(A_Q)<1$ \sa{for $A_Q$ defined in \eqref{def-AQ}}. \sa{For 
							$\lambda>0$ given,} consider the following $2\times 2$ DARE,
						\begin{equation} \tilde{X} 
							\sa{=} (\tilde{A}^{(\lambda)})^\top  \tilde{X} \tilde{A}^{(\lambda)} +   (\tilde{A}^{(\lambda)})^\top  \tilde{X}
							\tilde{B}(\gamma^2  - \tilde{B}^\top  \tilde{X}\tilde{B})^{-1}\tilde{B}^\top 
							\tilde{X}
							\tilde{A}^{(\lambda)}
							+\begin{bmatrix} \frac{\lambda}{2}& 0 \\
								0 & 0
							\end{bmatrix},
							\label{eq-small-riccati-bis}
						\end{equation}
						where $\gamma = \frac{1}{\sqrt{\theta/d}}$, $\tilde B\in\mathbb{R}^{2\times 1}$ is \sa{given} in $\eqref{def: system mat for TMM}$ and $$ \tilde{A}^{(\lambda)}:= \begin{bmatrix}
							1+\beta - \alpha(1+\nu)\lambda & -\beta + \alpha\nu \lambda \\
							1 & 0
						\end{bmatrix} \in \mathbb{R}^{2\times 2}.$$
						The risk-sensitive index of the GMM method with risk parameter $\theta>0$ is given by 
						\beq R(\theta) = 
						\begin{cases} 
							-\frac{\mgb{\sigma^2}}{\theta} \sum_{i=1}^d \log (1 - \frac{\theta}{d} \alpha^2 \tilde{X}^{(\lambda_i)}_{11}) & \mbox{if} \quad \sqrt{\theta} \mg{H_\infty} < \mgb{\sqrt{d}}, \\
							\infty  & \mbox{if} \quad \sqrt{\theta }\mg{H_\infty} \geq \mgb{\sqrt{d}},
						\end{cases} 
						\label{risk-index-quadratics}
						\eeq
						where \mg{$\{\lambda_i\}_{i=1}^d$ are the eigenvalues of $Q$}, $\mg{H_\infty}$ is given by \eqref{hinfty-quad}, and $\tilde{X}^{(\lambda_i)} \in \mathbb{R}^{2\times 2}$ is the stabilizing solution to \eqref{eq-small-riccati-bis} for $\lambda=\lambda_i$,
						i.e., $\tilde{X}^{(\lambda_i)}$ is positive semi-definite solving \eqref{eq-small-riccati-bis} for $\lambda=\lambda_i$ and satisfies $ \rho(\tilde{A}^{(\lambda_i)} + \gamma^{-2} \tilde B {\tilde B}^\top  \tilde{X}^{(\lambda_i)} ) < 1$.
					\end{theorem}
					\begin{proof} \mg{The proof is given in \cref{sec-proof-theorem-risk-quadratic}}.
					\end{proof}
					\mgc{\begin{remark}Theorem \ref{thm-gmm-risk-formula} relies on Riccati equations for providing a formula for the risk-sensitive index and applies only when $\theta>0$. Later on, in \cref{proposition-rate-function} and \cref{coro-risk-index-as-onedim-integral}, we will provide an alternative approach based on the evaluation of an integral that applies to the $\theta\leq 0$ case as well. \end{remark}}
					We note that the solutions to the $2\times 2$ Riccati equations \sa{in~\eqref{eq-small-riccati-bis} 
						of} Theorem \ref{thm-gmm-risk-formula} can be computed by a similar approach to \eqref{def-matrix-pencil}. \sa{For a given parameter $\lambda>0$,} it would suffice to compute the generalized eigenvalues and eigenvectors of the following $4\times 4$ pencil:
					$$ 
					\sa{M^{(\lambda)}(\mgbis{\tilde\mu}):=}\mgbis{\tilde\mu} \begin{bmatrix}
						\mathsf{I}_2  & -\ys{\tilde{B}\tilde{B}^\top/\gamma^2}\\
						0_2  & ({\tilde{A}^{(\lambda)}})^\top 
					\end{bmatrix} -
					\begin{bmatrix}
						{\tilde{A}^{(\lambda)}} & 0_2\\
						-\tilde{Q}^{(\lambda)} & \mathsf{I}_2
					\end{bmatrix}, \quad
					\mbox{where}
					\quad \tilde{Q}^{(\lambda)}  = \begin{bmatrix} \frac{\lambda}{2}& 0 \\
						0 & 0
					\end{bmatrix}.$$
					{That is, given the eigenvalues $\{\lambda_i\}_{i=1}^d$ of $Q$, one must compute pairs $(\tilde\mu_i,\tilde v_i)$ satisfying $M^{(\lambda_i)}(\tilde\mu_i)\tilde v_i$ $ = 0$ for all $i = 1,\ldots,d$.\looseness=-1}
					{Another approach is to solve \eqref{eq-small-riccati-bis} directly, which involves three unknowns—the entries of the symmetric $2 \times 2$ matrix—and three polynomial equations in these variables of degree at most four.} For instance, the Symbolic Toolbox of {\sc Matlab} can be used to express the solutions with an explicit formula. 
					\sa{That being said, 
						for gradient descent (GD)}, we can also compute the solutions explicitly by hand and obtain a formula for the risk-sensitive index which is provided in the next result.
					
					\begin{corollary}\label[corollary]{risk-sensitivity-index-GD-quad-f} 
						\sa{Under the premise} of \cref{thm-gmm-risk-formula},  
							\sa{given a step size $\alpha \in (0,\frac{2}{L})$, consider the gradient descent method, i.e., $\beta=\nu=0$, the risk-sensitive index corresponding to the risk  parameter} $\theta>0$ is given by 
							\begin{equation}
								R(\theta) =  \begin{cases}
									-\frac{\sigma^2}{\theta} \sum_{i=1}^d \log\left(
									1 -  \mgb{\frac{\theta}{d}} \alpha^2 \tilde{a}_{i} 
									\right)	& \mbox{if} \quad \sqrt{\theta}  H_\infty  < \mgb{\sqrt{d}}; \\
									\infty 	&\mbox{if} \quad \sqrt{\theta} H_\infty  \geq \mgb{\sqrt{d}},
								\end{cases}
								\label{eq-rtheta-to-prove}
							\end{equation}
							with 
							\mg{\beq H_{\infty} = \begin{cases} 
									\frac{1}{\sqrt{2\mu}} & \text{if}\quad 0 < \alpha \leq \frac{2}{L + \sqrt{L\mu}}; \\
									\frac{\alpha \sqrt{L}}{\sqrt{2(2 - \alpha L)}} & \text{if}\quad \frac{2}{L + \sqrt{L\mu}} < \alpha < \frac{2}{L},
								\end{cases}\label{eq-hinfty-gd-formula}
								\eeq}
							and
							\beqs
							\tilde{a}_{i} := 
							\frac{1}{2}
							\left(
							\frac{\lambda_i}{2}+c_{\scriptscriptstyle \mathrm{GD}}\bigg( 1-(1-\alpha\lambda_i)^2\bigg) 
							- 
							\sqrt{
								\bigg(\frac{\lambda_i}{2}+c_{\scriptscriptstyle \mathrm{GD}}\bigg( 1-(1-\alpha\lambda_i)^2\bigg)
								\bigg)^2 
								-2c_{\scriptscriptstyle \mathrm{GD}}\lambda_i
							}
							\right),
							\eeqs
							if $1-\alpha\lambda_i \neq 0$ and
							$\tilde{a}_{i}  := \frac{\lambda_i}{2}$ if $1-\alpha\lambda_i=0$, where $ c_{\scriptscriptstyle \mathrm{GD}} := \frac{d}{\alpha^2 \theta}$ and $\{\lambda_i\}_{i=1}^d$ are the eigenvalues of $Q$.\looseness=-1
						\end{corollary} 
						\begin{proof} The proof 
							\sa{follows from} \cref{thm-gmm-risk-formula} 
							\sa{through constructing} an explicit solution $\tilde{X}^{(\lambda_i)}$ to the DARE \sa{in} \eqref{eq-small-riccati-bis} for $\lambda=\lambda_i$ and $i=1,2,\dots,d$. \sa{Let $i\in\{1,\ldots,d\}$ be fixed.} For gradient descent, we have $\beta=\nu=0$; \sa{hence,} 
							$$ \tilde{A}^{(\lambda_i)} = \begin{bmatrix} 1-\alpha \lambda_i & 0 \\
								\sa{1}& 0
							\end{bmatrix}, \quad 
							\tilde{B}=
							\begin{bmatrix} 
								-\alpha \\ 
								0
							\end{bmatrix}.
							$$
							\sa{The sparsity structure of $\tilde{A}^{(\lambda_i)}$ implies that the solution $\tilde{X}^{(\lambda_i)}$ has the following form:}
							\beq \tilde{X}^{(\lambda_i)} = \begin{bmatrix} \tilde{a}_{i}  & 
								\sa{0}\\
								\sa{0} & 
								\sa{0}
							\end{bmatrix},\label{tilde-X}
							\eeq
							for 
							\sa{some} non-negative scalar $\tilde{a}_{i} $ (that may depend on $\lambda_i$). It follows after a straightforward computation that such an $\tilde{X}^{(\lambda_i)} $ is a solution to \eqref{eq-small-riccati-bis} for $\gamma = \frac{1}{\sqrt{\theta/d}}$ and $\lambda=\lambda_i$ if $\tilde{a}_{i} $ satisfies 
							$$\tilde{a}_{i} = (1-\alpha \lambda_i)^2 \tilde{a}_{i}  +  (1-\alpha \lambda_i)^2 {\tilde{a}_{i}}^2 \left(\frac{d}{\theta \alpha^2 }  - {\tilde{a}_{i}}\right)^{-1}
							+ \frac{1}{2}\lambda_i. $$
							If $1-\alpha\lambda_i = 0$, then we have ${\tilde{a}_{i}} = \frac{1}{2}\lambda_i$; \sa{otherwise, if $1-\alpha\lambda_i \neq 0$,} then
							this is a scalar quadratic equation with \sa{solutions of the form:}
							\beqs 
							\tilde{a}_{i,\pm} = \frac{1}{2}
							\left(
							\frac{\lambda_i}{2}+c_{\scriptscriptstyle \mathrm{GD}}\bigg( 1-(1-\alpha\lambda_i)^2\bigg) 
							\pm
							\sqrt{
								\bigg(\frac{\lambda_i}{2}+c_{\scriptscriptstyle \mathrm{GD}}\bigg( 1-(1-\alpha\lambda_i)^2\bigg)
								\bigg)^2 
								-2c_{\scriptscriptstyle \mathrm{GD}}\lambda_i
							}
							\right)\,
							\eeqs
							where $ \sa{c_{\scriptscriptstyle \mathrm{GD}}} = \frac{d}{\alpha^2 \theta}$.
							Since we are interested in the stabilizing solution of the Riccati equation of \eqref{eq-small-riccati-bis}, i.e., the solution $\tilde{X}^{(\lambda_i)}$ that satisfies $ \rho(\tilde{A}^{(\lambda_i)} + \gamma^{-2} \tilde{B} \tilde{B}^\top  \tilde{X}^{(\lambda_i)}) < 1$, we take $\tilde{a}_{i}  = \tilde{a}_{i,-}$. Then, with such choice of $\tilde{a}_{i} $, the matrix $\tilde{X}^{(\lambda_i)}$ is a stabilizing solution of the Riccati equation in \eqref{eq-small-riccati-bis}, and we conclude from \cref{thm-gmm-risk-formula} \mg{that the formula in \eqref{eq-rtheta-to-prove} holds with $H_\infty$, \sa{given in \eqref{hinfty-quad}; moreover, for GD,} since we have $\beta=\nu=0$, the formula in~\eqref{hinfty-quad} simplifies to \eqref{eq-hinfty-gd-formula}, and we conclude.}    
						\end{proof}
						
						\subsection{Numerical experiments comparing risk, convergence rate, and parameter choice}
						
						
						In this section, we investigate the relationship among risk-sensitive cost, parameter selection, and convergence rate in common parameterizations of GMMs from a numerical perspective. See \cref{tab:algo-params} \mgb{for a list of common parameters which covers HB, NAG and GD methods. The last column of this table provides the linear convergence rate \sa{$\rho\in(0,1)$ associated with the stated parameters as a function of the condition number $\kappa:=L/\mu$ when there is no noise}.\footnote{If $w_{k+1} = 0$ for every $k$; then $\rho^2 \in [0,1)$ is such that $f(x_k)-f(x_*)\leq \bar c_k\rho^{2k}[f(x_0) - f(x_*)]$ where $\bar c_k$ has at most polynomial growth, see e.g., \cite{can2022entropic}. \mgbis{Here, we abuse the notation and use $\rho$ to denote both the spectral radius of a matrix and the convergence rate, similar to the literature \cite{lessard2016analysis,gurbuzbalaban2023robustly}}.}
						} 
						
						\begin{table}[htbp]\centering
							\caption{Common parameterizations of GMMs. }
							\label{tab:algo-params}
							\setlength{\tabcolsep}{1pt}      
							\renewcommand{\arraystretch}{1.5} 
							{\small
								\begin{tabular}{|>{\centering\arraybackslash}m{1.5cm}
										|>{\centering\arraybackslash}m{5.75cm}
										|>{\centering\arraybackslash}m{4.2cm}
										|>{\centering\arraybackslash}m{3cm}|}
									\hline
									\textbf{Alg.} & \textbf{Parameters} & \textbf{Comments} & \textbf{Conv. Rate (\(\rho\)) without noise} \\
									\hline
									GD-pop & \(\alpha = \frac{1}{L},\; \beta = \nu = 0\) & Popular default choice \cite{lessard2016analysis}. & \(1 - \frac{1}{\kappa}\) \cite{lessard2016analysis}\\
									\hline
									GD-fastest & \(\alpha = \frac{2}{L+\mu},\; \beta = \nu = 0\) & Fastest rate without noise \cite{lessard2016analysis}. & \(1 - \frac{2}{\kappa+1}\) \cite{lessard2016analysis} \\
									\hline
									RS-GD & \(\alpha = \frac{2}{L+\sqrt{L\mu}},\; \beta = \nu = 0\) & Fastest rate while achieving optimal \(H_\infty\) \cite{gurbuzbalaban2023robustly}. & \(1 - \frac{2}{\kappa+\sqrt{\kappa}}\) \cite{gurbuzbalaban2023robustly} \\
									\hline
									NAG-pop & \(\alpha = \frac{1}{L},\; \beta = \nu = \frac{1-1/\sqrt{\kappa}}{1+1/\sqrt{\kappa}}\) & Popular Nesterov setting \cite{lessard2016analysis}. & \(1 - \frac{1}{\sqrt{\kappa}}\) \cite{lessard2016analysis} \\
									\hline
									NAG-fastest & \(\alpha = \frac{4}{3L+\mu},\; \beta = \nu = \frac{\sqrt{3\kappa+1} - 2}{\sqrt{3\kappa+1} + 2}\) & Fastest known rate for quadratics \cite{lessard2016analysis}. & \(1 - \frac{2}{\sqrt{3\kappa+1}}\) \cite{lessard2016analysis}\\
									\hline
									NAG-beta-opt & \(\alpha \in \left(0, \frac{1}{L}\right],\; \beta = \nu = \frac{1 - \sqrt{\alpha \mu}}{1 + \sqrt{\alpha \mu}}\) & Tunable \(\alpha\) with matching \(\beta\) optimizing \(\rho\) \cite[Lemma 2.1]{aybat2019universally}. & \(1 - \sqrt{\alpha \mu}\) \cite{aybat2019universally}\\
									\hline
									TMM & \(\alpha = \frac{1+\rho}{L},\; \beta = \frac{\rho^2}{2 - \rho},\; \nu = \frac{\rho^2}{(1+\rho)(2 - \rho)}\) & Proposed in \cite{scoy-gmm-ieee}. & \(1 - \frac{1}{\sqrt{\kappa}}\) \cite{scoy-gmm-ieee} \\
									\hline
									HB & ~~\(\alpha = \frac{4}{(\sqrt{L} + \sqrt{\mu})^2},\; \beta = \left(\frac{\sqrt{\kappa} - 1}{\sqrt{\kappa} + 1}\right)^2, \nu = 0\)\; & Fastest rate without noise \cite{lessard2016analysis}. & \(1 - \frac{2}{\sqrt{\kappa}+1}\) \cite{lessard2016analysis} \\
									\hline
									RS-HB & \;\(\alpha = \frac{a^2(\kappa)}{L},\; \beta = \left(1 - \frac{a(\kappa)}{\sqrt{\kappa}}\right)^2, \nu=0 \)\; & Proposed in \cite{gurbuzbalaban2023robustly}. & \(1 - \frac{\sqrt{2}}{\sqrt{\kappa}} + \mathcal{O}\left(\frac{1}{\kappa \sqrt{\kappa}}\right)\) as \(\kappa \to \infty\) \cite{gurbuzbalaban2023robustly}\\
									\hline
								\end{tabular}
							}
						\end{table}
						
						We focus on a simple strongly convex quadratic problem of the form \eqref{eq-quad} with $d = 2, L = 3, \mu = 1, \sa{g = [0,~0]^\top}, \mg{h} = 0, Q = \begin{bmatrix} \mu & 0 \\ 0 & L \end{bmatrix}$. \sa{We 
							set $\sigma^2 = 2$, and we conduct the experiments for various values of $\theta>0$ chosen appropriately so that \mgb{$\sqrt{\theta}H_\infty < \sqrt{d}=\sqrt{2}$} --otherwise \sa{$R(\theta)$} would be infinite by \cref{thm-gmm-risk-formula}.} 
						
						\begin{figure}[h!]
							\centering
							\includegraphics[width=0.75\linewidth]{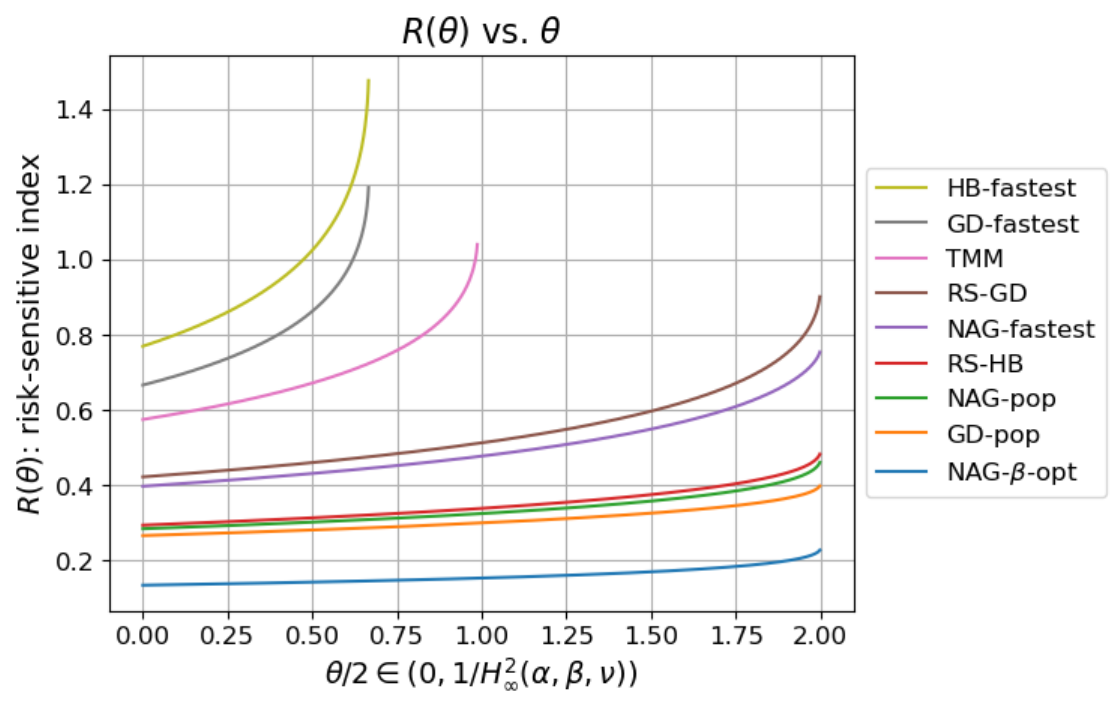}
							\caption{Risk sensitivity for common parameterizations of GMM methods; $d = 2, L = 3, \mu = 1, \sigma^2 = 2$.}
							\label{fig:risk-vs-theta}
						\end{figure}
						\mgc{In \cref{fig:risk-vs-theta}, we plot $R(\theta)$ for several GMM parameterizations, computed analytically with the DARE-based formula given in \cref{thm-gmm-risk-formula}. This formula is valid for $\theta \in \left(0, \tfrac{d}{H_\infty^2}\right)$, and $R(\theta)$ diverges when $\theta \geq \bar{\theta} := \tfrac{d}{H_\infty^2}$; \sa{hence, we truncated} each curve in \cref{fig:risk-vs-theta} around $\bar{\theta}$.}
						
						\sa{Certain 
							choice of parameters} offer more flexibility in $\theta$ selection than others, \mgb{as the $H_\infty$-norm depends on \sa{these} parameters}. 
						\mgb{We observe from \cref{fig:risk-vs-theta} that methods achieving the  fastest convergence rates \sa{(smallest $\rho$)}—such as HB-fastest, GD-fastest, TMM, and NAG-fastest—often exhibit the highest risk and the least flexibility (as they admit a larger $H_\infty$-norm).}
						
						{For NAG, we compare three step sizes:
							$\alpha = \frac{4}{3L + \mu}$ (NAG-fastest), $\alpha = \frac{1}{L}$ (NAG-pop), and $\alpha = \frac{1}{2L}$ (NAG-beta-opt) \sa{--since $L>\mu$, $\frac{4}{3L + \mu}>\frac{1}{L}>\frac{1}{2L}$,}
							\sa{each choice of $\alpha$} is paired with the momentum parameter
							$\sa{\beta=\nu} = \frac{1 - \sqrt{\alpha \mu}}{1 + \sqrt{\alpha \mu}}$ \sa{--this choice is in line with NAG-pop and NAG-fastest parameter choice stated in \cref{tab:algo-params}}.
							All variants achieve the accelerated rate $\rho = 1 - \Theta\left(\frac{1}{\sqrt{\kappa}}\right)$, with step size affecting the constant factor. Notably, smaller step sizes yield slower convergence but significantly lower risk, illustrating the trade-off between rate and risk sensitivity.
							A similar pattern holds for gradient descent (GD), where we compare three step sizes:
							$\alpha = \frac{2}{L+\mu}$ (GD-fastest), $\alpha = \frac{2}{L+\sqrt{L\mu}}$ (RS-GD), and $\alpha = \frac{1}{L}$ (GD-pop).  As the step size decreases, the convergence rate slows, \sa{i.e., corresponding $\rho$ value increases, but the corresponding risk sensitivity measured by the risk-sensitive index $R(\theta)$} improves and the 
							\sa{$R(\theta)$} is finite for a wider range of \sa{$\theta>0$}. 
							This highlights a fundamental trade-off: improved flexibility in choice of $\theta$ and lower risk sensitivity come at the expense of slower convergence.}
						\begin{figure}[h!]
							\centering
							\subfigure[]{\label{GD-hinfty-stepsize}\includegraphics[width=0.49\linewidth]{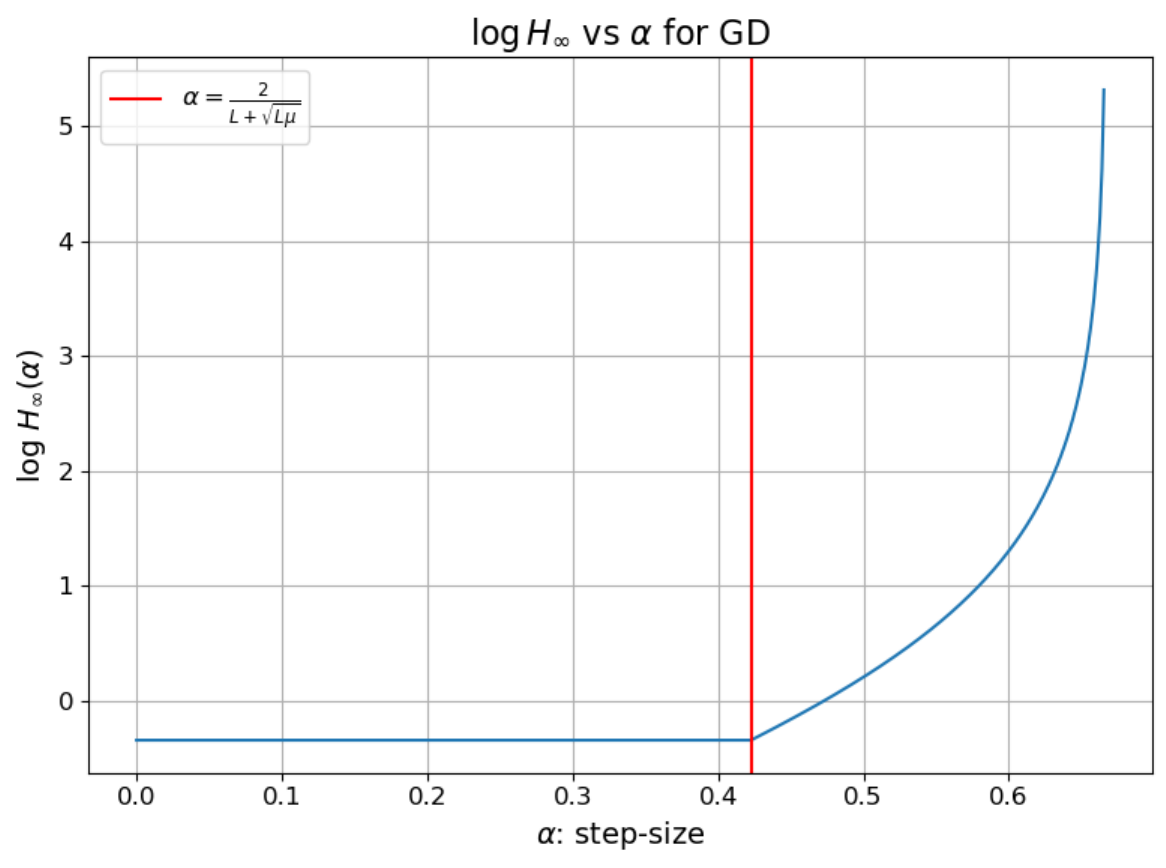}}
							\subfigure[]{\label{GD-risk-stepsize}\includegraphics[width=0.49\linewidth]{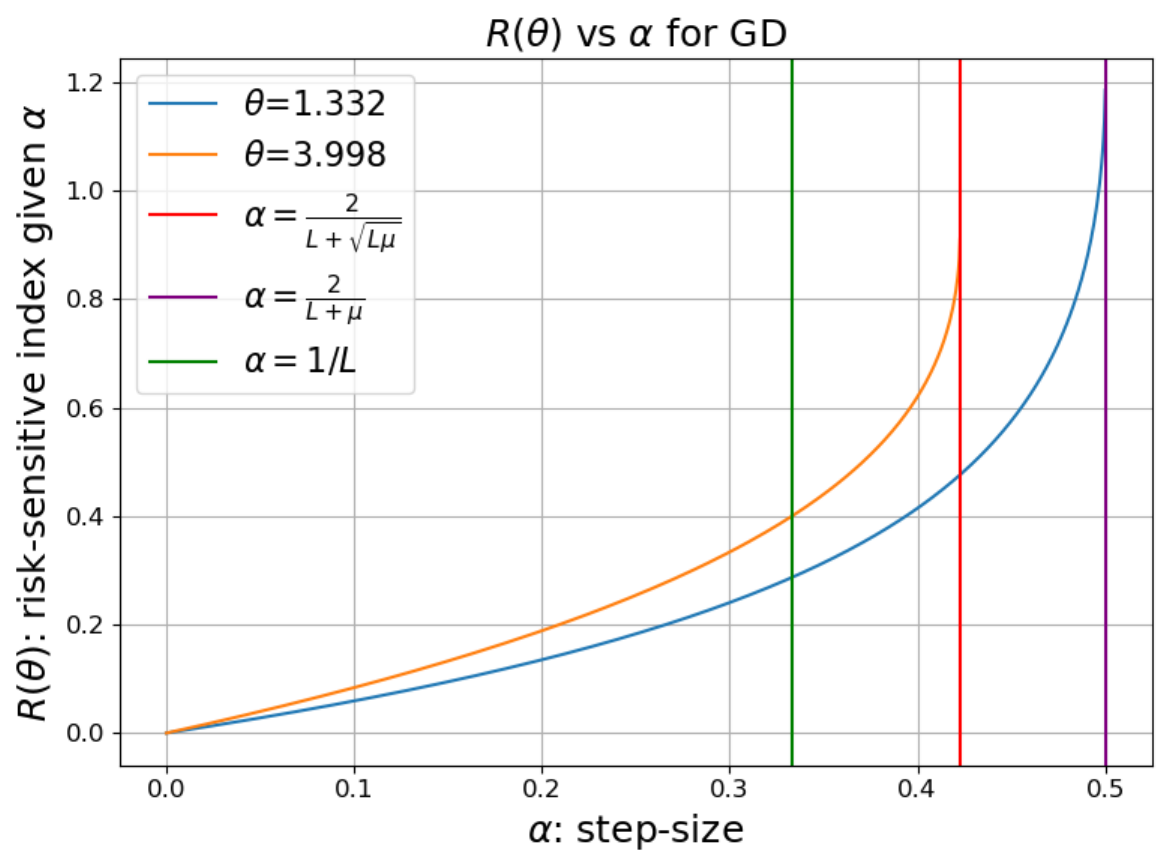}}
							\caption{\sa{(a) $H_\infty$ as a function of $\alpha$ based on \eqref{eq-hinfty-gd-formula} in the same problem setting; (b) The risk-sensitive index as a function of GD step-size $\alpha$ for $\theta \approx \{4/3,~4\}$, where the objective is a quadratic with $d = 2, L = 3, \mu = 1$, and $\sigma^2 = 2$.}}
							\label{fig: GD-risk-hinfty}
						\end{figure}
						
						\mgb{\ys{In} \cref{GD-hinfty-stepsize}, we plot the $H_\infty$-norm for GD as a function of the step size based on the formula \eqref{eq-hinfty-gd-formula}. The best robustness in the $H_\infty$ sense—namely, the minimal achievable $H_\infty$-norm of $\tfrac{1}{\sqrt{2\mu}}$, as implied by \eqref{eq-hinfty-gd-formula} and \cref{thm-h-inf}—is attained only when $\sa{0<}\alpha \leq \alpha_{\text{crit}} := \tfrac{2}{L + \sqrt{L\mu}}$. 
							In particular, RS-GD works with this maximal step size 
							\sa{$\alpha_{\text{crit}}$} that provides the fastest rate subject to the best (smallest) $H_\infty$-norm.\footnote{GD admits the convergence rate $\rho_{\text{GD}}(\alpha) = \max\{|1-\alpha \mu|, |1-\alpha L|\}$ for $\alpha \in (0,\frac2L)$, see \cite{bertsekas2002nonlinearprog}.} \sa{For $\alpha> \alpha_{\text{crit}}$, $H_\infty$-norm increases with $\alpha$; therefore, the range of $\theta$ for which the risk-sensitive index is finite gets narrower.} \mgbis{This is illustrated in \cref{GD-risk-stepsize}, which plots the risk as a function of the step size for GD with two choices of $\theta$, namely around $4$ (orange) and $4/3$ (blue).} 
							Comparing three step size choices, $\alpha = \frac1L$ (GD-pop),  $\alpha = \frac{2}{L+\mu}$ (GD-fastest), and $\alpha=\alpha_{\text{crit}}$ (RS-GD), we observe that 
							the risk becomes infinite for GD-fastest as $\theta$ approaches $4/3$, \mgbis{whereas for the same $\theta$ value, the risk is still finite for the other step size choices (that are relatively smaller)}. Similarly, the risk associated 
							\sa{with} RS-GD 
							\sa{tends to} infinite as $\theta$ approaches $\ys{4}$ but it is finite \mgbis{for smaller choices of the step size.} Basically, faster convergence rates come at the expense of increased risk sensitivity.} 
						
						\begin{figure}[h!]
							\centering     
							\subfigure[]{\label{fig: HB-risk-stepsize}\includegraphics[width=0.49\linewidth]{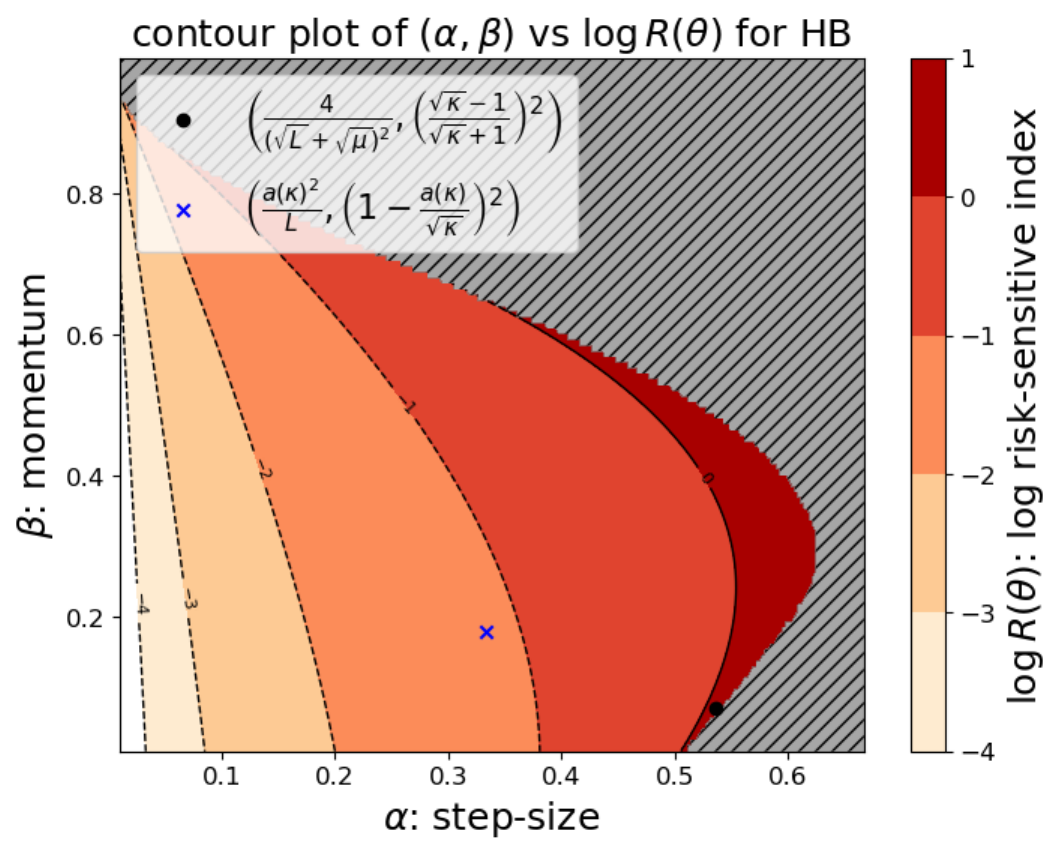}}
							\subfigure[]{\label{fig: HB-rate-stepsize}\includegraphics[width=0.49\linewidth]{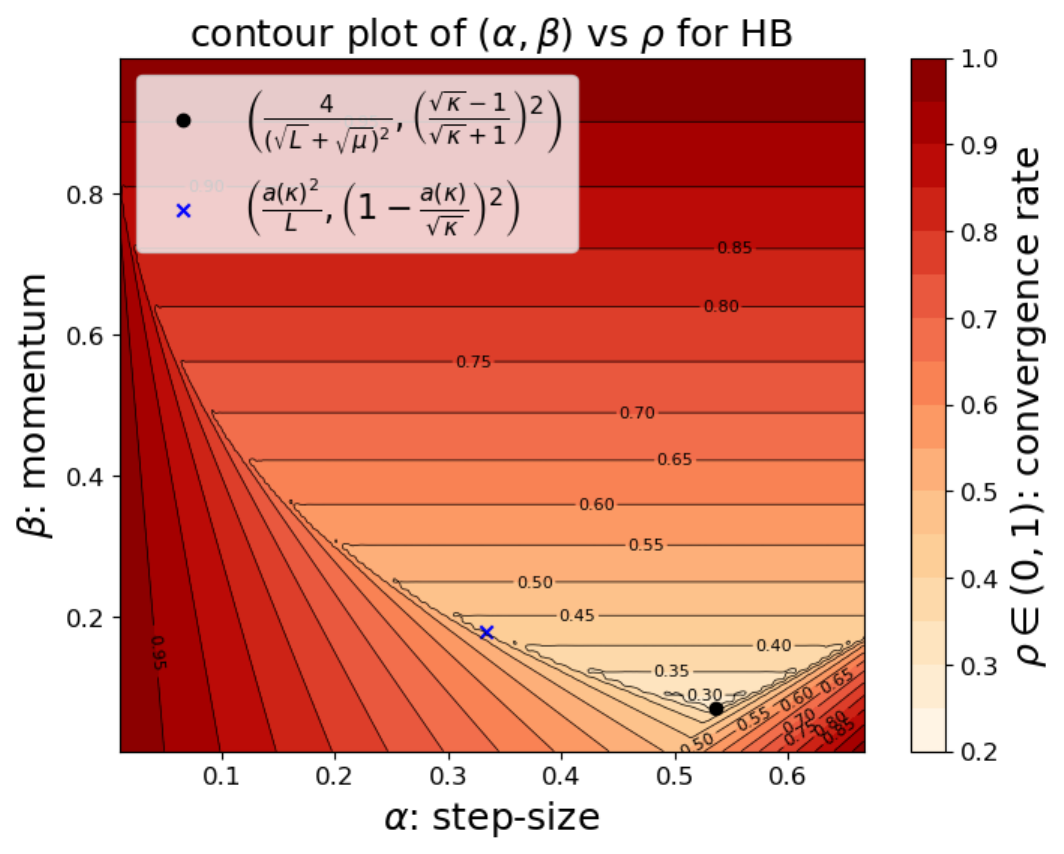}}
							\caption{(a) Logarithm of the risk-sensitive index $R(\theta)$ as a function of the parameters $\alpha,\beta$ of HB with $\theta = 1.2$, $\sigma^2=2$,$\nu=0$. (b) Convergence rate $\rho$ as a function of $\alpha,\beta$. The objective is a quadratic with $d = 2, L = 3,\mu = 1$.}
							\label{fig: HB-risk-rate}
						\end{figure}
						
						In \cref{fig: HB-risk-rate}, we consider the HB algorithm on a quadratic function with $\mu=1, L=3, d=2, \sigma^2=2, \theta=1.2$ and plot the logarithm of risk-sensitive index (\cref{fig: HB-risk-stepsize}) and the convergence rate $\rho$ (\cref{fig: HB-rate-stepsize}) as a function of the step size \sa{$\alpha$} and momentum parameter \sa{$\beta$} --see \cite[Lemma 3.1]{can2022entropic} for a precise convergence rate characterization 
						\sa{of HB for} strongly convex quadratics. The black dot indicates the standard HB parameters that yield the fastest convergence rate, while the blue cross represents the parameters of the RS-HB method, which achieve the best known rate while minimizing the $H_\infty$-norm \cite{gurbuzbalaban2023robustly}. The shaded gray region denotes parameter values for which the risk becomes infinite. As shown in \cref{fig: HB-risk-rate}, the fastest convergence rate in the noiseless case (black dot) comes at the cost of higher risk, whereas the RS-HB configuration (blue cross), offers improved robustness with a lower risk sensitivity and $H_\infty$-norm. \mgbis{Overall, we observe that faster convergence rates (lighter colors in \cref{fig: HB-rate-stepsize}) are typically accompanied by higher risk (darker colors in \cref{fig: HB-risk-stepsize}).}

						In \cref{fig: NAG-risk-rate}, we observe a phenomenon similar to the HB case for the NAG algorithm. The parameter setting labeled NAG-fastest (black dot), which achieves the fastest convergence rate in the absence of noise, incurs a higher risk compared to NAG-pop (blue cross), a commonly used parameterization. While NAG-fastest offers faster convergence, it comes at the cost of increased risk. 
						\sa{As in GD and HB, the higher risk values are generally associated with faster convergence rates also for NAG.}\looseness=-1
						
						\begin{figure}[h!]
							\centering     
							\subfigure[]{\label{fig: NAG-risk-stepsize}\includegraphics[width=0.49\linewidth]{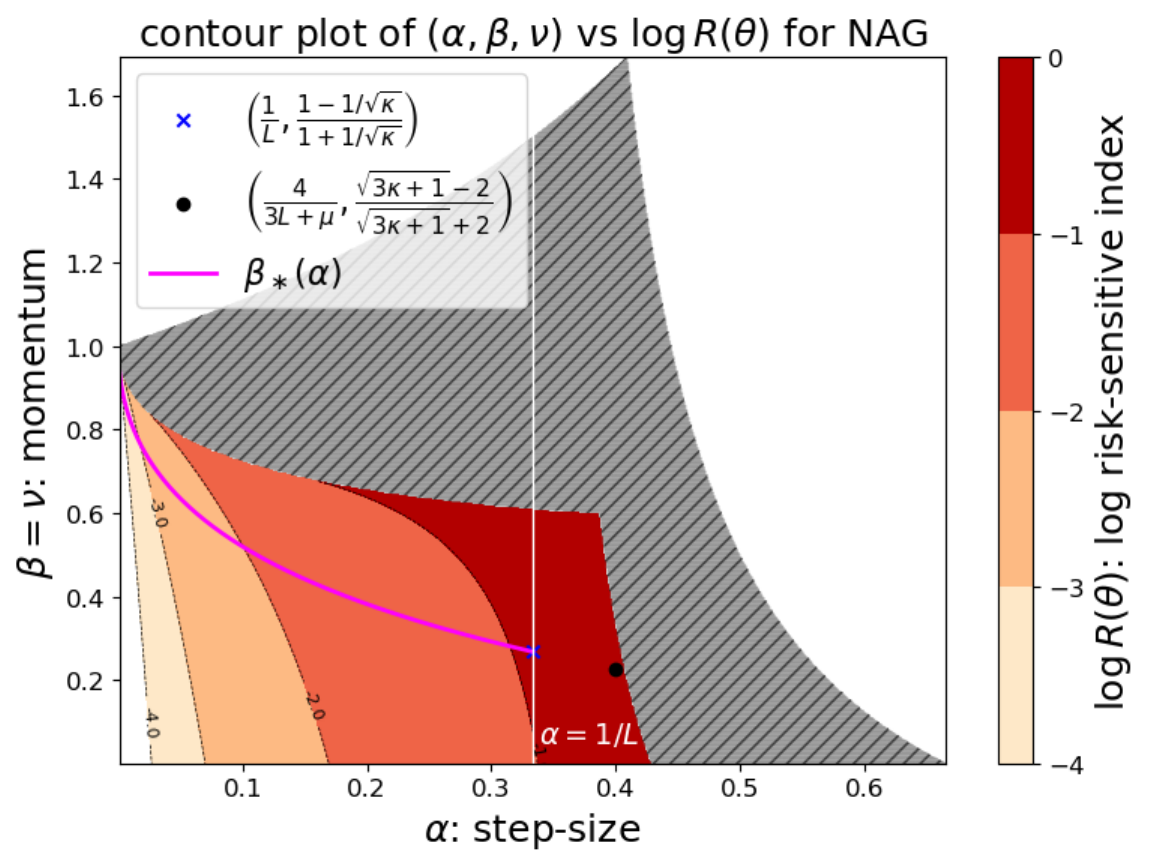}}
							\subfigure[]{\label{fig: NAG-rate-stepsize}\includegraphics[width=0.49\linewidth]{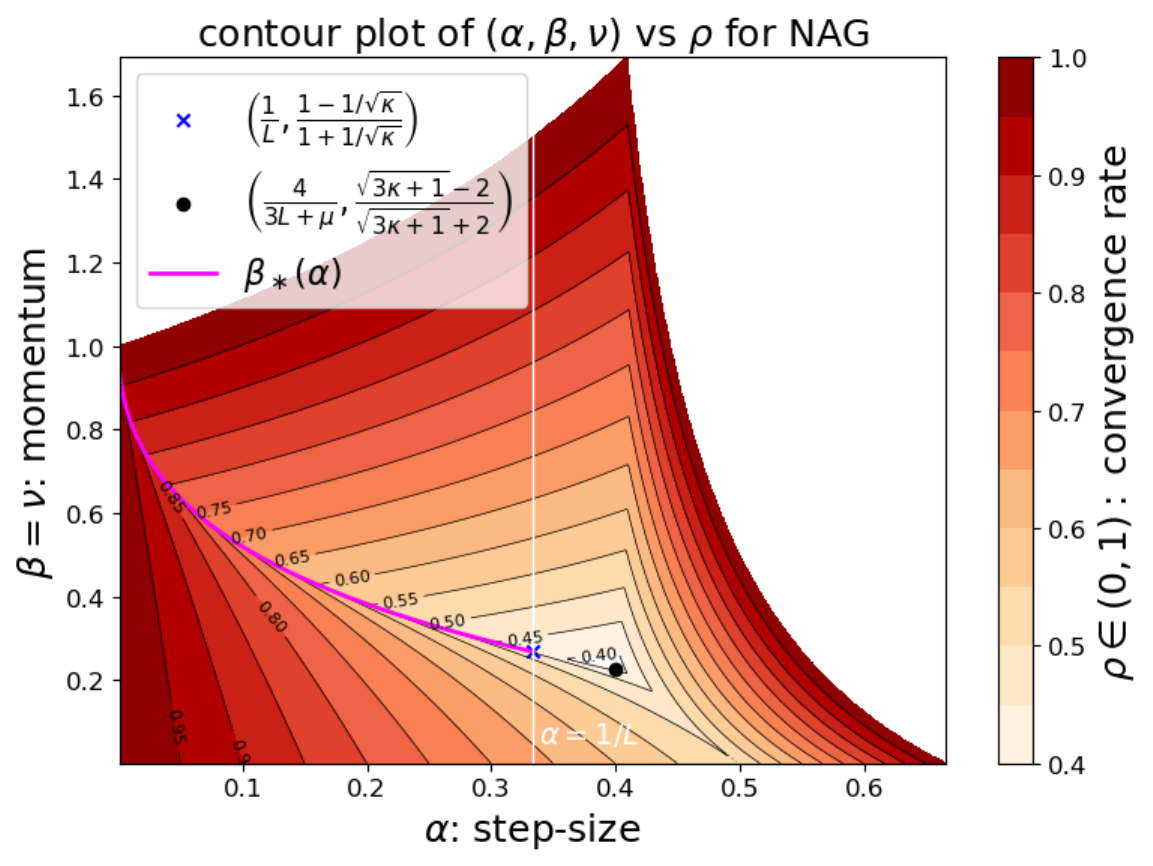}}
							\caption{(a) Logarithm of the risk-sensitive index $R(\theta)$ as a function of the parameters $\alpha,\beta$ of NAG with $\theta = 3.7$, $\sigma^2=2$,$\nu=\beta$. (b) Convergence rate $\rho$ as a function of $\alpha,\beta$, \sa{where} \ys{$\beta_*(\alpha) = \frac{1 - \sqrt{\alpha\mu}}{1+\sqrt{\alpha\mu}}$ for $\alpha \in [0,\frac1L]$. The objective is a quadratic with $d = 2, L = 3,\mu = 1$.} 
							}
							\label{fig: NAG-risk-rate}
						\end{figure}

						\begin{figure}[h!]
							\centering
							\includegraphics[width=0.75\linewidth]{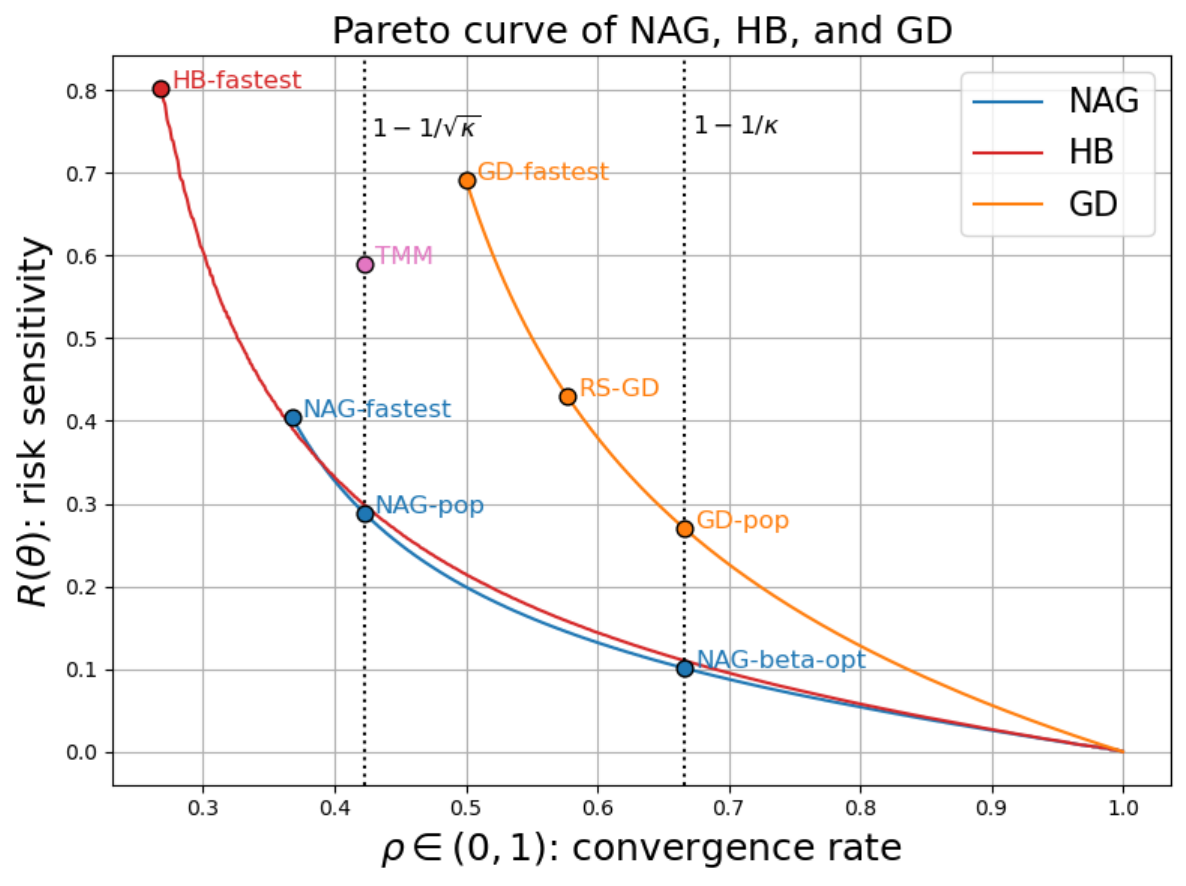}
							\caption{Pareto boundary for GD, NAG, HB illustrating the trade-off between risk and rate for a quadratic objective with $d = 2, L = 3, \mu = 1, \sigma^2 = 2, \theta = 0.2$.}
							\label{fig: pareto-boundary}
						\end{figure}
						As demonstrated so far, a faster rate, \sa{i.e., smaller $\rho\in(0,1)$,} comes often at the expense of an increased risk sensitivity \sa{index}. To better understand this relationship, \sa{fixing the risk parameter $\theta = 0.2$,} we examine the Pareto-optimal boundary in \Cref{fig: pareto-boundary} \sa{for} the same simple quadratic problem 
						\sa{we used for} \cref{fig:risk-vs-theta}, \sa{i.e., $d = 2, L = 3, \mu = 1$, and $\sigma^2 = 2$.} 
							In this setting, a point \((\rho,R)\) in the plane of \emph{convergence rate} \(\rho\in(0,1)\) and \emph{risk sensitivity \sa{index}} \(R\ge0\) is said to lie on the Pareto‐optimal boundary if there is no other achievable pair \((\rho', R')\) with
							$  \rho' < \rho$ and $R' \leq R$,
							or 
							$\rho' \leq \rho$ and $R'<R.$
							Equivalently, one cannot improve both the rate and the risk simultaneously beyond a Pareto‐optimal point. In \cref{fig: pareto-boundary}, each smooth curve corresponds to one algorithm (GD, HB, NAG), parametrized by its step‐size \(\alpha\) (and momentum parameters \(\beta,\nu\) where applicable).  As we slide \(\alpha\) (and \(\beta,\nu\)) within the stability region, we trace out the set of all achievable \(\bigl(\rho(\alpha,\beta,\nu),\,R(\alpha,\beta,\nu)\bigr)\). The Pareto frontier is the upper‐left envelope of these curves: \mg{it marks the parameters that has the lowest risk $R$ for a given achievable rate $\rho$ 
								\sa{for} a method.} 
							\mg{This frontier is numerically estimated via a grid search
								over the parameter space using the convergence rate formula from \cite{can2022entropic} and the explicit risk characterization provided in \cref{thm-gmm-risk-formula}.\footnote{\ys{A $3000 \times 3000$ grid is used for HB and NAG over the interval $(0,\frac2L) \times (0,2)$ for $(\alpha,\beta)$. 
										\sa{For HB we set $\nu =0$, 
											while for NAG we set $\nu = \beta$.} A grid of size 5000 is used for GD over the interval $(0, \alpha^*)$, where $\alpha^* \coloneqq \frac{2}{L+\mu} = \max\{\frac1L, \frac2{L+\mu}, \frac{2}{L + \sqrt{L\mu}}\}$, \mgbis{for which GD achieves the fastest rate.}
							}}}
							
							The dashed vertical lines \sa{on the right of \Cref{fig: pareto-boundary} marks} the standard rate $1-\frac{1}{\kappa}$ achieved by gradient descent with common step size $\alpha=1/L$ \sa{while the other line on the left marks} the accelerated rate  $1-\frac{1}{\sqrt{\kappa}}$ with the square root dependency to $\kappa$; \sa{\mgbis{indeed}, $1-\frac{1}{\kappa}>1-\frac{1}{\sqrt{\kappa}}$ \mgbis{here as} $\kappa>1$}. Dots label classical parameter choices within the parameter space of each algorithm 
							\sa{--see \Cref{tab:algo-params} for the definitions of these parameter choices.} We observe that GD achieves the most \textit{conservative} frontier: \mgbis{its Pareto-optimal curve lies above those of the momentum methods HB and NAG shown here.}
							\mgbis{Although GD-fastest attains the best rate achievable by GD, it remains slower and riskier than NAG-pop and NAG-fastest.} 
						For moderate rates (\(\rho\approx0.5\)–\(1.0\)), \mgbis{the Pareto-optimal curve of HB shows that HB can achieve lower risk than GD.} 
						At very fast rates (\(\rho\approx0.25\)–\(0.275\)) which GD cannot achieve, HB’s accelerated dynamics come with increased \sa{risk sensitivity}. NAG's popular tuning, NAG-pop, lies slightly below HB for low‐to‐moderate rates, achieving a small reduction in risk for the same \(\rho\). \mg{NAG-fastest is faster than NAG-pop as expected, but admits a slightly higher risk than what HB can achieve for the same rate.} 
						\sa{On the other hand,} TMM sits clearly above NAG and HB, indicating that while it is able to achieve a faster rate than GD, 
						\sa{its risk sensitivity index is more} than its counterparts which are able to achieve the same rate (NAG-pop has exactly the same rate) \mgbis{in the quadratic setting}. Finally, if the step size, $\alpha$, is carefully selected, one can achieve the same rate as GD-pop with NAG-beta-opt while simultaneously minimizing risk sensitivity among all competing popular parameterizations
						\sa{--indeed, for \Cref{fig: pareto-boundary} we selected} $\alpha = \frac{\mu}{L^2}$ for NAG-beta-opt to match GD-pop's convergence rate.
						
						\subsection{Large deviations and the risk-sensitive index} 
						It is well-known that averaging iterates can help with optimization performance in inexact gradient settings, see e.g., \cite{gurbuzbalaban2023robustly,zhang2022sapd+,devolder2013exactness,nemirovski2009robust,dieuleveut2020bridging} and the references therein. If we consider the ergodic averages of the iterates $$\bar{x}_K:=\frac{x_0+x_1+\dots + x_K}{K+1},$$ a natural performance metric we want to control is the probability $\mathbb{P}(f(\bar{x}_K) - 
						\sa{f_*} \geq t)$ that the suboptimality is larger than a given threshold \sa{$t>0$}. Clearly, by convexity, we have the bound
						\beq
						\mathbb{P}(f(\bar{x}_K) - 
						\sa{f_*}\geq t) \leq
						\mathbb{P}\left( \frac{  \sum_{j=0}^K \sa{[f(x_j) - 
								f_*]}}{K+1} \geq t\right)
						=
						\mathbb{P}\left(\frac{S_K}{K+1}\geq t\right).
						\label{convexity-averaging-ldp}
						\eeq
						The decay rate of this probability as a function of $t$ is 
						related to the large deviation behavior of 
						the sequence of random variables  $\Delta_K=S_K/(K+1)$ for $K\geq 0$ \cite{pham-risk-sensitive-control,dembo2009large}. Large deviations aim to quantify the exponential decay rate of the probabilities corresponding to rare events. More specifically, given a sequence of random variables $\{\Delta_K\}_{K\geq 0}$, we say that the sequence $\Delta_K$ satisfies a large deviation principle (LDP) if there exists a non-negative, lower semi-continuous function $I:\mathbb{R}\to [0,\infty]$ such that for every (Borel) measurable set $E \subset \mathbb{R}$, it holds that
						\beq
						-\inf_{s \in E^o} I(s)
						\leq 
						\liminf_{K\to\infty} \frac{1}{K}\log \mathbb{P}(\Delta_K \in E)
						\leq 
						\limsup_{K\to\infty} \frac{1}{K}\log \mathbb{P}(\Delta_K \in E)
						\leq -\inf_{s\in \bar{E}} I(s)\,, 
						\label{def-LDP}
						\eeq
						where $E^o$ and $\bar{E}$ denotes the interior and closure of the set $E$. If only the (upper-bound) right-most inequality in \eqref{def-LDP} holds (either for all Borel sets $E$ or for certain classes of sets), then we say that $\{\Delta_K\}$ satisfies an \emph{LDP upper bound}.
						The function \sa{$I(\cdot)$} is called the
						\textit{rate function}. \mgb{A rate function is called a \emph{good rate function} if it admits compact level sets; this allows the infimum over $I(s)$ in \eqref{def-LDP} to be attained 
							\sa{at some} 
							point in the closure of $E$, \sa{i.e., $\bar E$}. One difficulty in characterizing a good rate function arises from the fact that the random variables \sa{\(\{f(x_k) - 
								f_*\}_{k\geq 0}\)} are not i.i.d., even when the noise sequence $\ys{\{w_{k+1}\}_{k\ge0}}$~is i.i.d. Consequently, standard large deviations results such as Cramér’s Theorem do not apply. In the following result, we characterize the rate function by verifying the conditions of the Gärtner--Ellis Theorem~\cite{dembo2009large}, which accommodates non-i.i.d.~sequences 
							\sa{similar to} \( \{f(x_k) - 
							\sa{f_*}\}_{k\geq0} \). \mgbis{This characterization further leverages the risk-sensitive index formula, established in the proof of \cref{thm-gmm-risk-formula}, expressed as an integral over the unit circle.}
							\looseness=-1
						}
						
						\begin{proposition}\label[proposition]{proposition-rate-function} Let $f$ be a quadratic of the form \eqref{eq-quad}. Assume that the gradient noise follows an isotropic Gaussian distribution, i.e., Assumption \ref{assump-exp-moment-det-additive} holds for an i.i.d. sequence $\ys{\{w_{k+1}\}_{k\geq 0}}$ with $\ys{w_{k+1}} \sim \mathcal N(0,\ys{\sa{\frac{\sigma^2}{d}\mgb{\mathsf{I}_d}}})$ for some $\sigma>0$. Let the parameters $(\alpha, \beta, \nu)$ of GMM be such that $\rho(A_Q)<1$ \sa{for $A_Q$ defined in \eqref{def-AQ}}. Then, the sequence \sa{$\{\Delta_K\}_{K\geq 0}$ such that} $\Delta_K = S_K/(K+1)$ satisfies an LDP with \sa{the following rate function:}
							\beq I(s) =  \sup_{-\infty < \theta < \frac{d}{H_\infty^2}}\left[\frac{\theta}{2\sigma^2} \big(s - R(\theta)\big)
							\right],
							\label{eq-rate-function-quad}
							\eeq
							\sa{which is a good rate function,} where $H_\infty$ 
							\sa{satisfies} \eqref{hinfty-quad} and \mgbis{the risk-sensitive index admits the integral representation}
							\mgc{\beq R(\theta)=
								\begin{cases}
									-\dfrac{\sigma^2}{2\pi\,\theta}\displaystyle\int_{-\pi}^{\pi}
									\log\det\!\Bigl(\mathsf{I}_d-\dfrac{\theta}{d}\,G\!\left(e^{i\omega}\right)G\!\left(e^{i\omega}\right)^{*}\Bigr)\,d\omega,
									& \text{if } -\infty < \theta < \frac{d} {H_\infty^2}\sa{;}\\[2.0ex]
									+\infty, & \text{if } \theta \ge \frac{d}{H_\infty^{2}},
								\end{cases}
								\label{eq-entropy-integral}
								\eeq
								with the convention that the value at $\theta=0$ is interpreted by continuity as the limit
								\beq
								R(0)&=&\lim_{\theta\to 0} -\dfrac{\sigma^2}{2\pi\,\theta}\displaystyle\int_{-\pi}^{\pi}
								\log\det\!\Bigl(\mathsf{I}_d-\dfrac{\theta}{d}\,G\!\left(e^{i\omega}\right)G\!\left(e^{i\omega}\right)^{*}\Bigr)\,d\omega\\
								&=&\frac{\sigma^{2}}{2\pi d}\int_{-\pi}^{\pi}
								\operatorname{trace}\!\Bigl(G\!\left(e^{i\omega}\right)G\!\left(e^{i\omega}\right)^{*}\Bigr)\,d\omega.
								\eeq}Here, $ G(z) \coloneqq T(z\mathsf{I}_{2d} - A_Q)^{-1}B \in \mathbb{R}^{d\times d}$ for $z\in \mathbb{C} \sa{\setminus}\mbox{\text{Spec}}(A_Q)$ where $\mbox{\text{Spec}}(A_Q)$
							denotes the set of eigenvalues of $A_Q$, \mgbis{$G\!\left(e^{i\omega}\right)^{*}$} denotes the Hermitian (conjugate) transpose of $G(e^{i\omega})$, 
							\sa{$B$ and $T$ are defined in \eqref{def-ABC} and \eqref{def-T}, respectively.} Furthermore, for any $t>0$, 
							\beq
							\ys{\limsup_{K\to\infty}~ \ys{\frac1K}
								\log \mathbb P(f(\bar x_K) - 
								\sa{f_*}> t)
								\leq}
							\lim_{K \to \infty} \frac1K \log \mathbb P\left(\frac{S_K}{K+1} \geq t\right)
							\mg{=}-\inf_{s\in [t,\infty)} \ys{I(s)}.  
							\label{eq-rate-function-interval}
							\eeq
						\end{proposition}
						\begin{proof} The proof is provided in \cref{sec-prop-rate-function}.
						\end{proof}
						\begin{remark}\label[remark]{remark-convex-conjugate} \mgc{Since $R(\theta)<+\infty$ if and only if $\theta H_\infty^{2}<d$ (cf.~\eqref{eq-entropy-integral}),
								the supremum in \eqref{eq-rate-function-quad} may equivalently be taken over all $\theta\in\mathbb{R}$, i.e., $
								I(s)=\sup_{\theta\in\mathbb{R}}\frac{\theta}{2\sigma^{2}}\bigl(s-R(\theta)\bigr).
								$
								With the change of variable $\lambda=\theta/(2\sigma^{2})$ and $\tilde{\Lambda}(\lambda)=\lambda\,R(2\sigma^{2}\lambda)$, we obtain
								$
								I(s)=\sup_{\lambda\in\mathbb{R}}\{\lambda s-\tilde{\Lambda}(\lambda)\},
								$
								i.e., $I$ is the Fenchel--Legendre (convex) conjugate of $\tilde{\Lambda}$; \sa{hence,}
								$\tilde{\Lambda}$ is simply a scaled version of the risk-sensitive index. See \cref{sec-prop-rate-function} (proof of
								\cref{proposition-rate-function}) for details.}
						\end{remark}
						\mgbis{With some additional effort, involving the computation of the matrix \(G(e^{i\omega}) G^*(e^{i\omega})\) and leveraging its diagonal structure, we can simplify the integral representation \eqref{eq-entropy-integral} in the following corollary.
						}
						
						\begin{corollary}\label[corollary]{coro-risk-index-as-onedim-integral} In the setting of \cref{proposition-rate-function}, the formula \eqref{eq-entropy-integral} is equivalent to
							\mgc{
								\beq R(\theta) = \begin{cases}
									-\frac{\sigma^2}{2\pi \theta} \displaystyle\sum_{i=1}^d\displaystyle \int_{-\pi}^\pi  \log \left(1 - \mgb{\frac{\theta}{d}} h_{\omega}(\lambda_i)
									\right) d\omega\sa{,} & \mbox{if} \quad -\infty< \theta < \frac{d} {H_\infty^2}\sa{;}\\ 
									+\infty, & \theta \geq  \frac{d} {H_\infty^2}, 
								\end{cases} 
								\label{eq-entropy-integral-bis}
								\eeq
								with the convention that the value at $\theta=0$ is interpreted by continuity as the limit $$R(0)=\lim_{\theta\to 0} -\frac{\sigma^2}{2\pi \theta} \displaystyle\sum_{i=1}^d\displaystyle \int_{-\pi}^\pi  \log \left(1 - \mgb{\frac{\theta}{d}} h_{\omega}(\lambda_i)
								\right) d\omega =  \frac{\sigma^2}{2\pi d} \displaystyle\sum_{i=1}^d\displaystyle \int_{-\pi}^\pi   h_{\omega}(\lambda_i) d\omega.$$} Here, $\{\lambda_i\}_{i=1}^d$ are the eigenvalues of $Q$, \[
							h_\omega(\lambda) 
							:= \frac{\alpha^2 \lambda}{2 \Bigl( 
								1 + \tilde{b}_{\lambda}^2 + \tilde{c}_{\lambda}^2 
								+ 2\bigl[ \tilde{b}_{\lambda}(1+\tilde{c}_{\lambda})\cos(\omega) 
								+ \tilde{c}_{\lambda} \cos(2\omega) \bigr] 
								\Bigr)},
							\]
							with $\tilde{b}_\lambda =\alpha\lambda (1+\nu)-(1+\beta)$ and ${\tilde{c}_\lambda} = \beta - \alpha \lambda \nu$ as in \eqref{def-b-c-lambda}.
						\end{corollary}
						\begin{proof}The proof can be found in \cref{appendix-proof-coro-risk-index-as-onedim-integral}.
						\end{proof}
						\begin{figure}[h!]
							\centering
							\subfigure[]{\label{fig:rate-fn-quadratics}\includegraphics[width=0.49\linewidth]{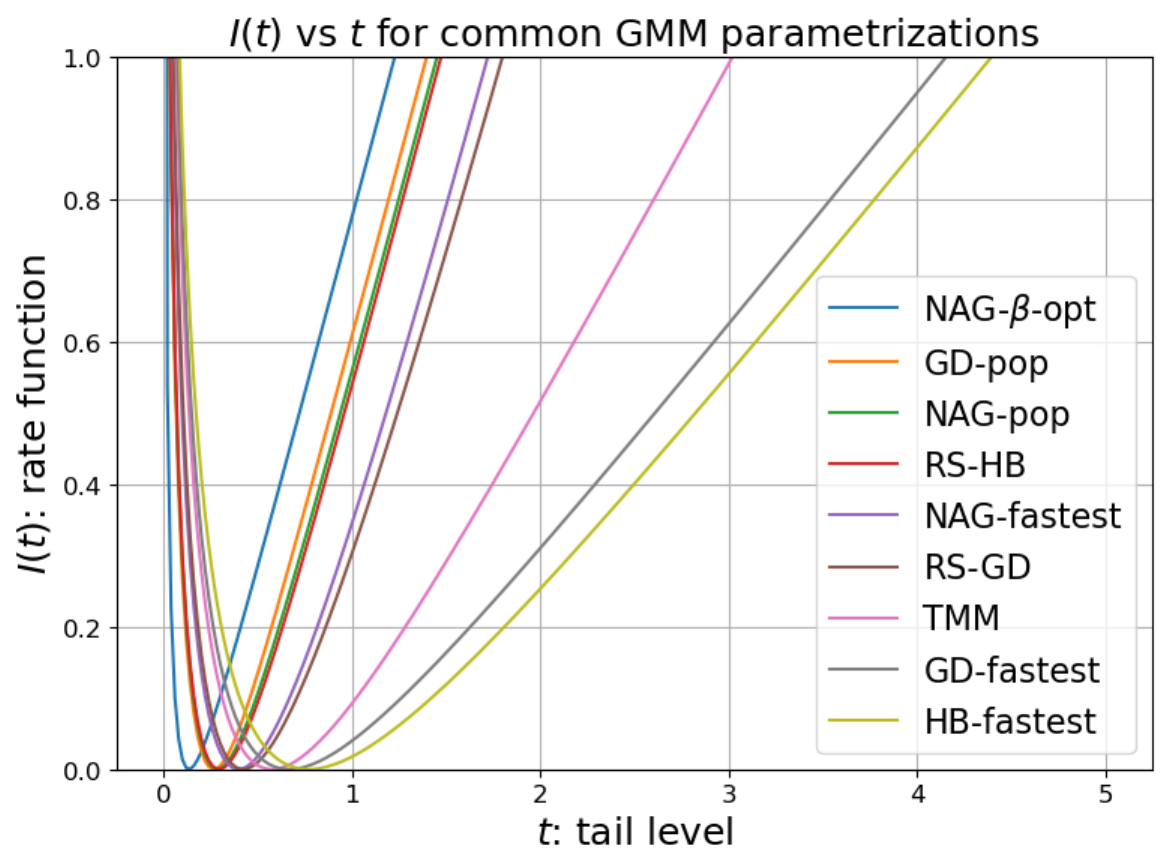}}
							\subfigure[]{\label{fig:rate-fn-nag-beta-opt-vary-stepsize}\includegraphics[width=0.49\linewidth]{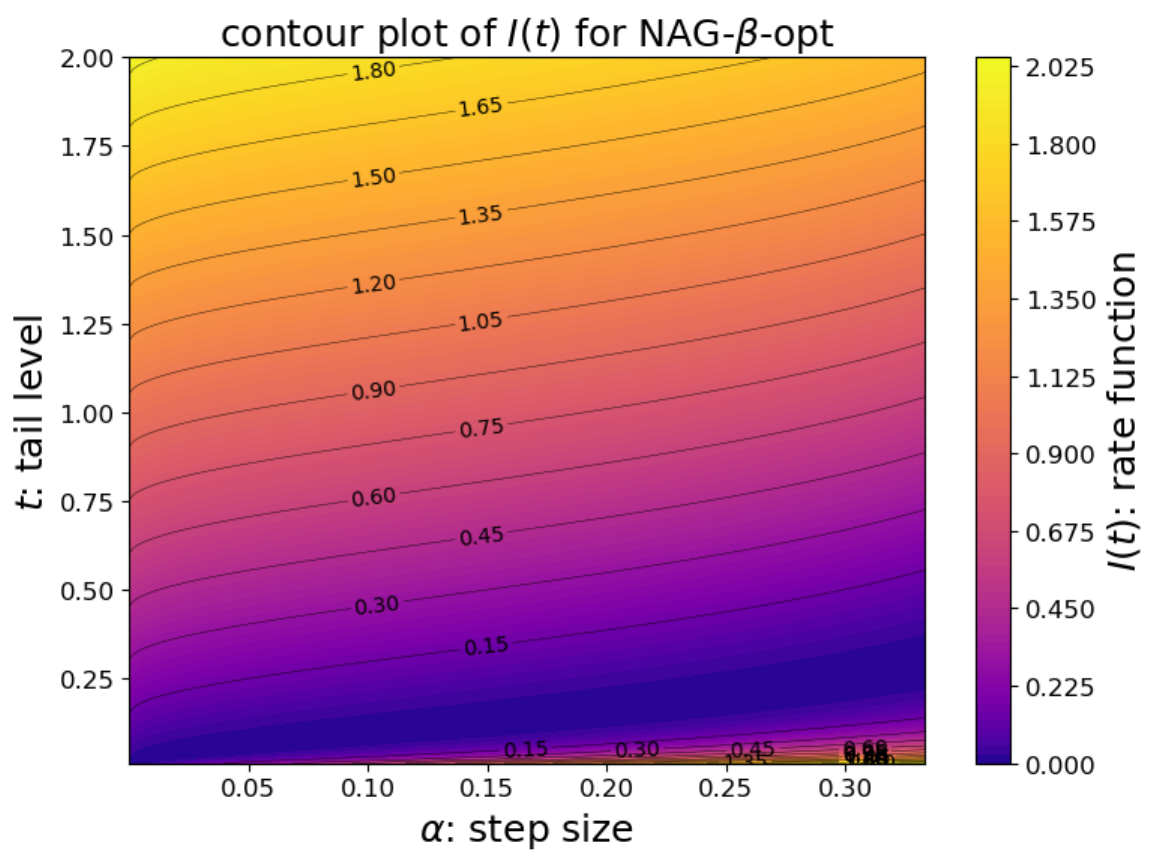}}
							\caption{\ys{(a) The rate function, $I(t)$, for common GMM parameterizations as a function of $t$, a given suboptimality threshold. (b) The rate function, $I(t)$, for NAG-beta-opt varying step-size $\alpha$ and varying suboptimality threshold, $t$. The objective is a quadratic with $d = 2, L = 3, \mu = 1, \sigma^2 = 2$.}}
							\label{fig:rate-fn-qd}
						\end{figure}
						\sa{\ys{In \cref{fig:rate-fn-quadratics}}}, we illustrate \cref{proposition-rate-function} for a quadratic objective \sa{as in \eqref{eq-quad} in dimension \( d = 2 \)} with parameters \sa{\( g = [0, 0]^\top \)}, \( h = 0 \), and
						$
						Q = \begin{bmatrix} \mu & 0 \\ 0 & L \end{bmatrix}$
						where \( \mu = 1 \), \( L = 3 \). The figure displays the rate function \( I(t) \) corresponding to various methods and parameter choices within the GMM framework, as listed in \cref{tab:algo-params}, using the dual representation \eqref{eq-rate-function-quad}. \mgbis{This entails computing $R(\theta)$, for which we evaluate the integral in~\eqref{eq-entropy-integral} numerically. For $\theta>0$, one may alternatively compute $R(\theta)$ based on the alternative representation ~\eqref{risk-index-quadratics}.} 
						Since \( I(t) \) is a good rate function, it has compact level sets, a property that is evident in the plots. Given that the gradient noise is i.i.d.\ Gaussian, the iterates follow a Gaussian distribution, and the average suboptimality \( S_K / (K + 1) \) 
						exhibits exponential tails for all \( K \), including in the limit as \( K \to \infty \). This behavior results in the linear growth of \( I(t) \) in the large-\( t \) regime. Notably, in this regime, 
						\sa{HB-fastest, the fastest in terms of convergence rate (see \cref{tab:algo-params}), has 
							the smallest values of \( I(t) \), which indicates a} slower decay of probabilities for large deviations. 
						\mgbis{As seen from formula \eqref{eq-rate-function-quad}, the rate function is, up to scaling, the convex conjugate (Fenchel–Legendre transform) of the risk-sensitive index (see \cref{remark-convex-conjugate}). Consequently, a similar trade-off between convergence rate and risk arises: faster convergence entails a slower decay of the probabilities of rare events for sufficiently large $t$.} \sa{Due to the parabolic shape of \( I(t) \), 
							its minimum is attained} at the most probable value of \( t \), \sa{i.e.,
							smaller or larger values} being relatively less likely.
							\sa{\ys{In \cref{fig:rate-fn-nag-beta-opt-vary-stepsize}}}, we further investigate the NAG-beta-opt method, which exhibits the fastest tail decay for sufficiently large \( t \) 
							\sa{as seen \ys{in \cref{fig:rate-fn-quadratics}}}. Specifically, we present a contour plot of the rate function \ys{\( I(t) \) as a function of $\alpha$ and $t$} under the NAG-beta-opt parameterization, varying the step size \( \alpha \in (0, 1/L] \) and the tail threshold \( t \), where \( \beta = \frac{1 - \sqrt{\alpha\mu}}{1 + \sqrt{\alpha\mu}} \). The prominent “purple valley” visible in the plot traces, for each value of \( \alpha \), the corresponding value of \( t \) that minimizes \( I(t) \)—representing the most probable values achievable at that step size. \mgbis{For a fixed step size, as $t$ increases, $I(t)$ is non-decreasing, as expected, and it governs the decay rate of the probability that the average suboptimality exceeds the threshold $t$.} 

							\section{Main Results for Strongly Convex Objectives}\label{sec-strongly-convex}
							Our analysis in Section \ref{sec-quad} 
							\sa{focuses on} the asymptotic limit of $R_K(\theta)$ as $K\to \infty$ for strongly convex quadratic objectives. In this section, our purpose is to obtain bounds for finite $K$ for \sa{more general smooth} strongly convex functions $f\in \Cml$ \mgbis{within our biased noise model}.

							\sa{In general for the class of GMM algorithms,} beyond gradient descent, one difficulty is 
							that the squared distance to the optimal solution $\| {x}_k - x_*\|^2$ is not necessarily monotonically decreasing \cite[Lemma 3.1]{can2022entropic} even 
							\sa{in the absence of} noise. \sa{This monotonicity issue can be circumvented by analyzing the method using 
								the following more general Lyapunov function:}
							\beq  V_{P,c_1}({\xi}_k) := c_1 \left( f({x}_k) - f(x_*) \right) + \begin{bmatrix} {x}_k - x_* \\ {x}_{k-1} - x_*  \end{bmatrix}^\top  P\begin{bmatrix} {x}_k - x_* \\ {x}_{k-1} - x_*  \end{bmatrix},
							\label{eq-lyap-gmm-general}
							\eeq
							where $c_1\geq 0$, $P= \tilde{P}\otimes \mathsf{I}_d$ \sa{for some $2\times 2$ symmetric positive semi-definite matrix $\tilde{P}$}. This Lyapunov function has been beneficial in the study of HB, NAG and GMM methods \cite{hu2017dissipativity,can2022entropic,gurbuzbalaban2023robustly,siopt-ragm,aybat2019universally} \mgb{leading to a tight convergence rate analysis without noise (when $w_{k+1} = 0 ~\forall k\geq 0$) for 
								\sa{a wide range of parameter values}~\cite{can2019accelerated}}. 
							Indeed, in the absence of noise, \sa{it can be certified that GMM 
								iterate sequence linearly converges to $x_*$ with a rate $\rho\in (0,1)$ or better,} if $\rho$ and $\tilde{P}$ satisfy a $3\times 3$ matrix inequality (MI) \cite{hu2017dissipativity}. Explicit solutions $(\rho, \tilde{P})$ that satisfy this MI are known for NAG with the common choice of parameters (NAG-pop, NAG-beta-opt in Table \ref{tab:algo-params}), for GD and also for some choices of HB, GMM parameters \cite{can2022entropic}; however for more general parameters, explicit solutions are not known and they can be computed numerically \mgbis{in an efficient way by solving a $3\times 3$ MI.\looseness=-1}  
							
							When the objective function $f$ is strongly convex and smooth but not quadratic, the dynamical system corresponding to the GMM iterations becomes nonlinear, as the gradients are nonlinear functions of the state. In this setting, although the $H_\infty$-norm remains well-defined (as per \mg{\cref{def-Hinfty}}), deriving an explicit formula for it does not seem possible, and the proof technique used in \cref{thm-gmm-risk-formula} does not directly apply. Moreover, as observed in \cref{sec-quad}, finiteness of the $H_\infty$-norm is necessary for the risk $R(\theta)$ to remain finite in a neighborhood of zero, at least for quadratic objectives. Otherwise, $R(\theta)$ can become infinite for any $\theta > 0$. Therefore, to establish finite upper bounds on the risk $R(\theta)$ for general objectives $f \in \Cml$, we must impose conditions on the algorithm parameters that ensure the $H_\infty$-norm is finite.
							
							Recent work by \cite{gurbuzbalaban2023robustly} establishes a $4 \times 4$ MI which, if satisfied, guarantees both a finite upper bound on the $H_\infty$-norm and a descent property for the Lyapunov function \sa{$V_{P, c_1}(\cdot)$}. This result is tight in the sense that, for certain parameter choices, the bound on the $H_\infty$-norm matches the corresponding quadratic bounds up to \sa{some} constants, and in the noiseless case, this MI recovers the rate results implied by the $3\times 3$ MI \sa{provided in \cite{hu2017dissipativity}.}
							We include this result in the appendix (\cref{thm-hinfty-bound}) for completeness, together with its corollaries (\cref{cor-GD-params-MI,cor-NAG-params-MI}) that provide explicit bounds \mgbis{$\overline{H}_\infty$} on the $H_\infty$-norm 
							\sa{corresponding to GD, NAG-beta-opt and NAG-pop} which are obtained by constructing explicit solutions to the MI. In the next result, we show that when the $4 \times 4$ MI holds, it also enables us to derive bounds on the risk sensitivity $R_K(\theta)$ for any horizon $K$ and for any $\theta > 0$ satisfying $\sqrt{\theta} ~\overline{H}_\infty < 1$. \sa{This type of upper bound on $\theta$ to guarantee finiteness of $H_\infty$} is expected; a similar constraint arises in the quadratic case where computations are exact \mgbis{as in \eqref{eq-rtheta-to-prove}}. The proof of this result mirrors the dynamic programming 
							\sa{technique employed within the} risk-sensitive control \cite{whittle2002risk}. \mgbis{Since the risk metric $R_K(\theta)$ aggregates risk across iterations and the noise sequence $\{w_{k+1}\}_{k\geq0}$ is conditionally independent from $\mathcal{F}_k$, we can isolate the contribution of each step via a backward--induction argument while fixing the iteration budget to $K$. This induction argument \sa{makes the critical parameter condition  explicit, i.e.,}  \(\sqrt{\theta}\,\overline H_{\infty} < 1\) that guarantees finiteness of the risk for all iteration budgets. The proof relies on a novel Lyapunov function, obtained by modifying $V_{P,c_1}(\xi_k)$ into a weighted form that incorporates convex averages with $V_{P,c_1}(\xi_{k-1})$. Full details are provided in the appendix. This modification allows us to derive tighter bounds.\looseness=-1
							}
							\begin{theorem}\label[theorem]{thm-str-cvx-risk-cost-bound}
								Consider the noisy GMM iterations \sa{stated in} \mgb{\eqref{Sys:stoc-RBMM}} for minimizing $f\in \Cml$, starting from $\ys{x_0= x_{-1}} \in \mathbb{R}^d$ with fixed parameters $(\alpha, \beta, \nu)$. Let \cref{assump-exp-moment-det-additive} hold 
								and consider the risk-sensitive cost in \eqref{def-risk-cost} \mgb{for \sa{some given arbitrary iteration budget $K\geq 1$ that is fixed}}.
								
								\mgb{Assume that there exists non-negative scalars $\rho_0, \rho_1, \rho_2, \rho_3 \in [0,1)$, \sa{$a, b, c_1, c_0\geq 0$,} and a positive semi-definite matrix $\tilde{P}\in\mathbb{R}^{2\times 2}$ which satisfies the conditions (i) and (ii) of \cref{thm-hinfty-bound} so that the $H_\infty$-norm of the GMM system admits the finite upper bound 
									\beq \quad\quad 
									\mgb{
										\overline{H}_\infty:= \bigg(
										\tfrac{1 }{1-\big(p+q\big)} {  \tfrac{{
													r } }{{ c_1 + \frac{2}{L} \tilde{r}(\tilde{P}) }} 
										} 
									}
									\bigg)^{1/2} <\infty,
									\eeq
									where 
									$$p = \rho_\ys{0}^2 + c_1\rho_1^2 + \rho_2^2
									+ \frac{4b(1 + \nu)^2L^2}{\mu}c_1, \quad
									q = \rho_3^2 + \frac{4b\nu^2L^2}{\mu}c_1, \quad 
									\ys{r := a + \alpha^2\left(\frac{c_1L}{2} + \tilde P_{11}\right)},$$
									and 
									$$\tilde{r}(\tilde P) = \begin{cases}
										\tilde{P}_{11} - {\tilde{P}_{12}^2}/{\tilde{P}_{22}} & \mbox{if} \quad \tilde{P}_{22}\neq 0, \\
										\tilde{P}_{11} & \mbox{if} \quad \tilde{P}_{22}=  0.
									\end{cases}
									$$  
									Then,}
								given $\theta>0$, we have the upper bound $R_{K}(\theta) 
								\leq \overline{R}_{K}(\theta)
								$ with
								\beq
								\overline{R}_{K}(\theta) :=
								\begin{cases}
									\displaystyle
									\frac{1}{\left(c_1 + \frac{2}{L} \tilde{r}(\tilde{P})\right)(K+1)}
									\min \left\{
									\frac{1}{J_{p,q}} \bar{c}\left(K, \frac{\theta}{J_{p,q}} \right),
									\bar{c}(K+1, \theta)
									\right\}, & \text{if } \sqrt{\theta} \,\overline{H}_\infty < 1; \\[10pt]
									+\infty, & \text{otherwise},
								\end{cases}
								\label{def-risk-bound-bar-RK}
								\eeq
								where 
								$
								\bar{c}(K, \theta) :=
								\left( \frac{1 - \lambda_+^{K+1}}{1 - \lambda_+} - 
								1 + J_{p,q}\right) V_{P,c_1}(\xi_0)
								+ \ys{\hat c(K,\theta)}
								$,

								$$ \lambda_+ = \frac{p+\sqrt{p^2+4q}}2 \in [0,1), \quad J_{p,q} = \frac{1}{1+\lambda_+-p} \in (0,1], 
								$$
								
									$$
									\ys{\hat c(K,\theta)}
									\coloneqq 
									\begin{cases}
										\frac{\mgb{2\sigma^2}\left(c_1 + \frac2L \tilde{r}(\tilde P)\right)}{\theta}
										\sa{\displaystyle\sum_{j=1}^{K}}\log\left(\frac{
											c_1 + \frac2L \tilde{r}(\tilde P)
											+ r~\theta 
											\sa{a_{j,K}} }{
											c_1 + \frac2L \tilde{r}(\tilde P)
											- r~\theta 
											\sa{a_{j,K}} }\right), & \text{if }\sqrt{\theta}\,\overline{H}_\infty < 1,\\
										+\infty & \mbox{else},
									\end{cases}
									$$
									\beq  
									\quad 
									a_{j,K} := \frac{1-\lambda_+^{K-j+1}}{1-(p+q)}  \quad \mbox{for} 
									\quad \mgbis{j =1, 2, \dots, K},
									\label{def-a-j-k}
									\eeq 
									\sa{and $V_{P,c_1}$ is as in \eqref{eq-lyap-gmm-general} with $P = \tilde P \otimes \mathsf{I}_d$. }
								\end{theorem}
								\begin{proof}\mgbis{The proof is 
										\sa{provided in} Appendix~\ref{proof-risk-sensitive-cost-bound-GMM}}.
								\end{proof}
								\vspace{0.2in}
								As a corollary of this theorem, in the next result, we obtain bounds on the expected suboptimality for the averaged iterates.
								\begin{corollary} 
									\sa{Under the premise} of \cref{thm-str-cvx-risk-cost-bound}, \sa{for any given $K\geq 0$,} the expected suboptimality of the averaged iterate $\bar{x}_K:=\frac{x_0+x_1+\dots + x_K}{K+1}$ satisfies
									\begin{eqnarray}
										\mathbb{E}[f(\bar{x}_K)-
										\sa{f_*}]
										&\leq& R_K(0)\\
										&\leq& \overline R_{K}(0) := \lim_{\theta \downarrow 0} \overline R_{K}(\theta) \\
										&=&
										\sa{\frac{\left(c_1 + \frac2L \tilde{r}(\tilde P)\right)^{-1}}{K+1}}
										\min\left\{\frac{1}{J_{p,q}} \bar c(K,0), \bar c(K+1,0)\right\}, \qquad \label{ineq-ergodic-subopt}
									\end{eqnarray}
									where 
									\begin{align}\bar{c}(K,0) 
										\sa{:=} \left( \frac{1 - \lambda_+^{K+1}}{1 - \lambda_+} - 1 +J_{p,q} \right) V_{P,c_1}(\xi_0) + 4\sigma^2 r\sa{\sum_{j=1}^{K} a_{j,K}}.\label{def-barc-K-at-zero}
									\end{align}
								\end{corollary}
								\begin{proof} From the Taylor expansion of \sa{$R_{K}(\theta)$ around $\theta=0$,  
										using \eqref{def-risk-param-zero} we have:}
									\[
									\E[f(\bar{x}_K)-
									\sa{f_*}] \leq
									\frac1{K+1}\sum_{k=0}^K \E[f(x_k) - 
									\sa{f_*}]
									=
									R_{K}(0)=
									\lim_{\theta \downarrow 0}
									R_{K}(\theta),
									\]
									\mgbis{where we used the convexity of $f$ in the above inequality.} We note that for $\theta>0$ sufficiently small, $R_K(\theta)$ and $\overline R_{K}(\theta)$ are continuous in $\theta$ and they satisfy $R_{K}(\theta) \leq \overline R_{K}(\theta)$ by \cref{thm-str-cvx-risk-cost-bound}. Furthermore, it is straightforward to check that the limit 
									$\overline{R}_K(0):=\lim_{\theta\downarrow 0} \overline{R}_K(\theta)$
									exists and admits the formula \sa{in} \eqref{ineq-ergodic-subopt}, 
									\sa{which follows from} \eqref{def-risk-bound-bar-RK} and the fact that 
									$$\lim_{\theta \downarrow 0} \hat c(K,\theta) = 4\sigma^2 r\sa{\sum_{j=1}^{K} a_{j,K}}.$$ 
									Furthermore, $\overline{R}_K(\theta)$ is \sa{right-continuous} at $\theta=0$; 
									therefore, it is necessarily the case that $\E[f(\bar{x}_K)-
									\sa{f_*}] = R_K(0) \leq \overline{R}_K(0) =\lim_{\theta \downarrow 0} \overline R_{K}(\theta)$ and we conclude. 
								\end{proof}
								\begin{corollary} For 
									\sa{GD}, \sa{given any $K\geq 0$ and $\alpha\in (0, \frac{2}{L})$,} the risk-sensitive cost 
									admits the following bound:
									$$ R_{K}(\theta) \leq \min\left\{\overline R_{K}^{(1)}(\theta), \overline R_{K}^{(2)}(\theta)\right\} \quad \mbox{for} \quad \theta \in \left(0,\frac{1}{\overline{H}_\infty^2}\right)$$ 
									where
									$$\small \overline{H}_\infty=  \begin{cases} 
										\frac{1}{\sqrt{2\mu}} & \mbox{if} \quad 0 < \alpha \leq \frac{1}{L}, \\
										\frac{1}{\sqrt{2\mu}}\frac{\alpha L}{2-\alpha L} & \mbox{if} \quad \frac{1}{L} <\alpha \leq \frac{2}{L+\sqrt{L\mu}},\\
										\frac{1}{\sqrt{2\mu}}  \sqrt{\kappa} & \mbox{if} \quad  \frac{2}{L+\sqrt{L\mu}} < \alpha \leq \frac{2}{L+\mu}, \\
										\frac{\alpha\sqrt{L}}{\sqrt{2}(2-\alpha L)}  & \mbox{if} \quad  \frac{2}{L+\mu}<\alpha < \frac{2}{L}, \\
									\end{cases}
									$$
									\[
									\overline R_{K}^{(1)}(\theta)
									\coloneqq 
									\frac{L}{2(\ys{K+1})}
									\left(\frac{1-[\rho_{GD}(\alpha)]^{K+1}}{1-\rho_{GD}(\alpha)} \|x_0-x_*\|^2 + \ys{\frac{4\sigma^2}{L\theta}} \sum_{j=1}^{K} \sa{\log\left(\frac{2 + \frac{\theta L\alpha^2}{1-\rho_{GD}(\alpha)}~a_{j,K}^{(1)}}{2 - \frac{\theta L \alpha^2}{1-\rho_{GD}(\alpha)}~a_{j,K}^{(1)}}\right)}\right),
									\]
									with 
									\sa{$a_{j,K}^{(1)}=\frac{1-[\rho_{GD}(\alpha)]^{K-j+1}}{1-\rho_{GD}(\alpha)}$ for $j=1,\ldots,K$},  \sa{$\rho_{GD}(\alpha)=\max\{\|1-\alpha\mu|,|1-\alpha L|\}$} and
									\begin{align*}
										\overline R_{K}^{(2)}(\theta)
										&\coloneqq 
										\frac{1}{\ys{K+1}}\left[
										\frac{1 - \left[1-2\mu\alpha(1-\frac{L\alpha}{2})+ \alpha\mu|1-\alpha L|s_{\scriptscriptstyle \mathrm{GD}}\right]^{K+1}}{2\mu\alpha(1-\frac{L\alpha}{2})-\alpha\mu|1-\alpha L|s_{\scriptscriptstyle \mathrm{GD}}}\left[f(x_0) - 
										\sa{f_*}\right] \right.\\
										&\quad\quad\quad\quad\quad\quad\quad\quad\quad\quad\quad\quad
										\left. 
										+\ys{\frac{2\sigma^2}\theta} \sum_{j=1}^{K} 
										\log\left(
										\frac{1 + \theta a_{j,K}^{(2)}\alpha\left(\frac{L}{2}\alpha + \frac{|1-\alpha L|}{2s_{\scriptscriptstyle \mathrm{GD}}}\right)}{1 - \theta a_{j,K}^{(2)}\alpha\left(\frac{L}{2}\alpha + \frac{|1-\alpha L|}{2s_{\scriptscriptstyle \mathrm{GD}}}\right)}
										\right)
										\right],
									\end{align*}
									\sa{with}  
									$a_{j,K}^{(2)} = \frac{1 - \left[1-2\mu\alpha(1-\frac{L\alpha}{2})+ \alpha\mu|1-\alpha L|s_{\scriptscriptstyle \mathrm{GD}}\right]^{K-j+1}}{2\mu\alpha(1-\frac{L\alpha}{2})-\alpha\mu|1-\alpha L|s_{\scriptscriptstyle \mathrm{GD}}}$ \sa{for $j=1,\ldots,K$},  $s_{\scriptscriptstyle \mathrm{GD}} = 1$ for \sa{$\alpha\in [0, \frac{1}{L}]$} 
									and $s_{\scriptscriptstyle \mathrm{GD}} = \frac{2-\alpha L}{\alpha L}$ for \sa{$\alpha\in (\frac{1}{L}, \frac{2}{L})$.} 
								\end{corollary}
								\begin{proof} \Cref{cor-GD-params-MI} shows that the conditions of \cref{thm-str-cvx-risk-cost-bound} are satisfied for GD with $\alpha\in (0,2/L)$ if we take
									$$p=\rho_{GD}(\alpha)=\max\{|1-\alpha\mu|,|1-\alpha L|\},\quad q=0, \quad r = \frac{\alpha^2}{1-\rho_{GD}(\alpha)}, \quad \tilde P=\mathsf{I}_2, \quad c_1 = 0,$$
									which corresponds to \sa{the Lyapunov function} $V_{P,c_1}(\xi_k)=\|x_k-x_*\|^2$. 
									Then, $\lambda_+ = p,~J_{p,q}=1$, $\tilde{r}(\tilde P)= 1$ and $\min\{\bar c(K,\theta), \bar c(K+1,\theta)\} = \bar c(K,\theta)$. Invoking \cref{thm-str-cvx-risk-cost-bound} implies the desired upper bound $\overline R_{K}^{(1)}(\theta)$.
									Similarly, by \cref{cor-GD-params-MI}, we can also take
									$$p=1-2\mu\alpha\left(1-\frac{L\alpha}{2}\right)+ \alpha\mu|1-\alpha L|s_{\scriptscriptstyle \mathrm{GD}}, ~~ q=0, ~~ r = \alpha\left(\frac{L}{2}\alpha + \frac{|1-\alpha L|}{2s_{\scriptscriptstyle \mathrm{GD}}}\right), ~~ \tilde{P}=0, ~~ c_1 = 1,$$
									within \cref{thm-str-cvx-risk-cost-bound}, which corresponds to the Lyapunov function $V_{P,c_1}(\xi_k)=f(x_k)-
									\sa{f_*}$ \sa{and together with} $\lambda_+=p$, $J_{p,q}=1$ and $\tilde{r}(\tilde{P})=0$ leads to the desired upper bound $\overline R_{K}^{(2)}(\theta)$.
								\end{proof}
								
								\begin{corollary}
									For NAG with $\alpha\in(0,1/L]$ and $\beta = \nu = \frac{1-\sqrt{\alpha\mu}}{1+\sqrt{\alpha \mu}}$, the bound \eqref{def-risk-bound-bar-RK} in \cref{thm-str-cvx-risk-cost-bound} holds with constants 
									{\small
										\begin{eqnarray*}
											p
											=
											\left[1 - \sqrt{\alpha\mu}\right] + 
											\left[\frac{2}{s_1}\left(\frac{\alpha}{4(1+\sqrt{\alpha\mu})^2}\right) \right] +
											\left[ 
											\frac{2}{\mu}
											\left(\frac{\alpha^2\mu^2+2\alpha\mu+\alpha\mu(1-\sqrt{\alpha\mu})}{4(1+\sqrt{\alpha\mu})^2}\right)
											\right] +
											\frac{4L^3\alpha^2}{2\mu s_2}(1+\nu)^2\,, 
										\end{eqnarray*}
										\begin{eqnarray*}
											q 
											= \frac{2}{s_1}\left(\frac{\alpha(1-\sqrt{\alpha\mu})}{4(1+\sqrt{\alpha\mu})^2}\right) + \frac{4L^3\alpha^2\nu^2}{2\mu s_2}\,, 
											\quad
											r 
											=
											a + \alpha^2 \left(\frac{c_1 L}2 +\tilde P_{11}\right)
											=
											s_1 + \frac{L\alpha^2(2+s_2)}{2s_2} + \frac{\alpha}2,
										\end{eqnarray*}
										$$c_1 = 1, \quad \tilde{r}(\tilde P)=0,$$
									}
									where 
									{\small
										\[
										s_1 = \frac{2(5-2\sqrt{\alpha\mu}+\alpha\mu)\ys{\sqrt{\alpha}}}{\sqrt{\mu}(1+\sqrt{\alpha\mu})^2}, \quad s_2 = \frac{8L^3\alpha\ys{\sqrt{\alpha}}}{\mu\sqrt{\mu}(1+\sqrt{\alpha\mu})^2}(4+(1-\sqrt{\alpha\mu})^2).
										\]
									}
								\end{corollary}
								\begin{proof} For NAG with this choice of step size \sa{$\alpha$} and momentum \sa{parameters $\beta=\nu$}, by \cref{cor-NAG-params-MI}, we observe that the conditions of \cref{thm-str-cvx-risk-cost-bound} \sa{hold} for this choice of $p,q,r$ values \sa{with $c_1=1$ and} 
									$$
									\tilde P = \frac1{2\alpha}\begin{bmatrix}1 \\ -(1-\sqrt{\alpha\mu})\end{bmatrix}\begin{bmatrix}1 & -(1-\sqrt{\alpha\mu})\end{bmatrix}.
									$$
									A short computation reveals that for this choice of $\tilde{P}$, \cref{thm-str-cvx-risk-cost-bound} is applicable with $\tilde r(\tilde P)=0$, and we conclude. 
								\end{proof}
								We \sa{next} leverage \cref{thm-str-cvx-risk-cost-bound} \sa{to 
									bound} the risk-sensitive index in the following corollary.
								\begin{corollary}\label[corollary]{cor-asymptotic-risk-index-bound}
									Under the 
									\sa{premise} of \cref{thm-str-cvx-risk-cost-bound}, for $\theta>0$ it holds that
									\begin{equation}  R(\theta) 
										\leq 
										\overline R(\theta)\coloneqq \mgbis{\limsup_{K\to\infty} \overline R_{K}(\theta)=}
										\begin{cases*}
											\frac{2\sigma^2}{\theta} \log\left(\frac{1 + \theta (\overline{H}_\infty)^2}{1 - \theta(\overline{H}_\infty)^2 }\right), & $\sqrt{\theta}~\overline{H}_\infty < 1,$ \\
											+\infty, & otherwise.
										\end{cases*}\label{eq-R-bar-theta}
									\end{equation}
									\mgbis{Furthermore, 
										\beq \limsup_{K\to\infty}\mathbb{E}[f(\bar{x}_K)-
										\sa{f_*}] \leq \lim_{\theta\downarrow 0}\overline{R}(\theta)=4\sigma^2 (\overline{H}_\infty)^2.\label{ineq-lim-suboptimality-with-bias}
										\eeq}
								\end{corollary}
								\begin{proof} By \cref{thm-str-cvx-risk-cost-bound}, \mgbis{for $\theta > 0$ satisfying $\sqrt{\theta} \,\overline{H}_\infty < 1$}, we have
									{\small
										\beq  
										\limsup_{K\to\infty} \overline R_{K}(\theta) 
										&=&
										\ys{\frac{1}{c_1 + \frac2L \tilde{r}(\tilde P)}}~\limsup_{K\to\infty} \frac{1}{K+1}\min\left\{\frac{1}{J_{p,q}} \ys{\bar c\left(K, \frac{\theta}{J_{p,q}}\right), \bar c(K+1, \theta)}\right\} \label{ineq-bar-c-K-largeK-limit-1}\\
										&=&
										\mgbis{\frac{1}{c_1 + \frac2L \tilde{r}(\tilde P)}}~\limsup_{K\to\infty} \frac{1}{K+1}   \bar c(K+1, \theta)\label{ineq-bar-c-K-largeK-limit-2}\\
										&=&\mgbis{\frac{\mgb{2\sigma^2}}{\theta}
											\limsup_{K\to\infty}\sa{\displaystyle\frac{1}{K+1}{}\sum_{j=1}^{K+1}}\log\left(\frac{
												c_1 + \frac2L \tilde{r}(\tilde P)
												+ r~\theta 
												\sa{a_{j,K+1}} }{
												c_1 + \frac2L \tilde{r}(\tilde P)
												- r~\theta 
												\sa{a_{j,K+1}} }\right)}  \label{ineq-bar-c-K-largeK-limit-3}\\
										&=& \frac{\mgb{2\sigma^2}}{\theta}
										\limsup_{K\to\infty}\sa{\displaystyle\frac{1}{K+1}{}\sum_{j=1}^{K+1}}\log\left(\frac{
											c_1 + \frac2L \tilde{r}(\tilde P)
											+ r~\theta 
											\sa{a_{1,j}} }{
											c_1 + \frac2L \tilde{r}(\tilde P)
											- r~\theta 
											\sa{a_{1,j}} }\right)\,,\label{ineq-bar-c-K-largeK-limit}
										\eeq}%
									{where we used the definition of $\bar{c}(K+1,\theta)$ and $a_{j,K+1}$ given in \cref{thm-str-cvx-risk-cost-bound}.  Since the sequence $\{a_{1,j}\}_{j\geq 1}$ is non-decreasing in $j$ admitting $\lim_{j\to\infty}a_{1,j} = \frac{1}{1-(p+q)}$; the sequence $\ell_j:=\log\left(\frac{
											c_1 + \frac2L \tilde{r}(\tilde P)
											+ r~\theta 
											a_{1,j} }{
											c_1 + \frac2L \tilde{r}(\tilde P)
											- r~\theta 
											a_{1,j} }\right)$ is also non-decreasing and admits the limit
										{\small
											$$\ell_{\infty} := \lim_{j\to\infty} \ell_j =  \log\left(\frac{
												c_1 + \frac2L \tilde{r}(\tilde P)
												+ r~\theta 
												1/(1-(p+q)) }{
												c_1 + \frac2L \tilde{r}(\tilde P)
												- r~\theta 
												1/(1-(p+q)) }\right) = \log\left(\frac{1 + \theta (\overline{H}_\infty)^2}{1 - \theta(\overline{H}_\infty)^2 }\right),$$}%
										using the definition of $\overline{H}_\infty$ from \cref{thm-str-cvx-risk-cost-bound}. Hence we conclude from \eqref{ineq-bar-c-K-largeK-limit} that
										{\small
											\beqs
											\overline R(\theta) = \limsup_{K\to\infty} \overline R_{K}(\theta) 
											&=&
											\frac{2\sigma^2}{\theta} \ell_\infty
											=
											\frac{2\sigma^2}{\theta} \log\left(\frac{1 + \theta (\overline{H}_\infty)^2}{1 - \theta(\overline{H}_\infty)^2 }\right).\eeqs}}%
									Since {$R(\theta)=\limsup_{K\to\infty} R_K(\theta) \leq \limsup_{K\to\infty} \overline R_{K}(\theta)$ and $R_K(\theta)=+\infty$ for $\sqrt{\theta} \,\overline{H}_\infty \geq 1$ for $\forall K$, we obtain \eqref{eq-R-bar-theta} as desired. With similar computations, taking limit superior on both sides in \eqref{ineq-ergodic-subopt} with respect to $K$ and using \eqref{def-barc-K-at-zero}, we obtain
										$ \limsup_{K\to\infty}\mathbb{E}[f(\bar{x}_K)-
										\sa{f_*}] \leq 4\sigma^2 (\overline{H}_\infty)^2.
										$
										Finally, the fact that $\lim_{\theta\downarrow 0}\overline{R}(\theta)=4\sigma^2 (\overline{H}_\infty)^2$ follows simply by taking limit as $\theta\to 0$ in \eqref{eq-R-bar-theta}.}
								\end{proof}
								{
									\begin{remark}  
										For momentum methods with constant step size and unbiased stochastic gradient 
										\sa{estimates}, ergodic averages converge to a neighborhood controlled by the variance proxy $\sigma^2$. For example, for NAG-beta-opt (\sa{using parameters as} in \cref{tab:algo-params}) with i.i.d.\ noise $w_{k+1} \sim \mathcal{N}(0,\tfrac{\sigma^2}{d})$, \cite[Cor.~4.8]{siopt-ragm} gives 
										$\limsup_{K\to\infty}\mathbb{E}[f(\bar{x}_K)-
										\sa{f_*}] \leq \sigma^2 \tfrac{\sqrt{\alpha}}{2\sqrt{\mu}}(1+\alpha L)$ for $\alpha \in (0,1/L]$, which vanishes as $\alpha \to 0$. In contrast, under biased noise the neighborhood is generally bounded away from zero, owing to a potentially non-vanishing worst-case drift. From \eqref{ineq-lim-suboptimality-with-bias}, the asymptotic error is governed by the \sa{bound on} $H_\infty$-norm. For NAG-beta-opt, \sa{under the \cref{assump-exp-moment-det-additive}}, by \cref{cor-NAG-params-MI} in the appendix, we have
										$\overline{H}_\infty = \tfrac{2\sqrt{5}}{\sqrt{\mu}} + \mathcal{O}(\sqrt{\alpha})$ as $\alpha \to 0$, and by \eqref{ineq-lim-suboptimality-with-bias} this yields
										$\limsup_{K\to\infty}\mathbb{E}[f(\bar{x}_K)-
										\sa{f_*}] \leq \sigma^2\bigl(\tfrac{80}{\mu} + \mathcal{O}(\sqrt{\alpha})\bigr)$, 
										and the right-hand side remains strictly positive as $\alpha\to 0$. Corollary \ref{cor-asymptotic-risk-index-bound} shows that momentum algorithms with smaller \sa{bound $\overline H_\infty$} 
										exhibit tighter bounds on the asymptotic (suboptimality) bias, i.e., $\limsup_{K\to\infty}\mathbb{E}[f(\bar{x}_K)-
										\sa{f_*}]$. Similar observations can be made for GD, HB or GMM. To our knowledge, this connection to the $H_\infty$-norm is entirely new for momentum methods.
									\end{remark}
								}
								For quadratic functions, we obtained large deviation results which obtains both the upper bound and the lower bound in \eqref{def-LDP} through an exact computation of the risk-sensitive index as $K\to \infty$. Beyond quadratics, we have only bounds on the risk-sensitive index. In the following result, we leverage this bound to obtain a large-deviation upper bound and high probability bounds for finite $K$. 
								\begin{theorem}\label[theorem]{theorem-rate-function-finite-time} Under the premise of Theorem \ref{thm-str-cvx-risk-cost-bound}, 
									consider the averaged iterates $\bar{x}_K = \frac{\sum_{j=0}^{K} x_j}{K+1}$. 
									Then, for any \sa{$K\geq 1$} given and for any \sa{$\theta\in[0,~1/\overline{H}_\infty^2)$} 
									and $t\geq 0$, we have 
									{\small
										\beq \mathbb{P}\left( f(\bar{x}_K) - 
										\sa{f_*}\geq t\right) \leq \mathbb{P}\left(\frac{S_K}{K+1}\geq t\right)\leq
										\exp\left(-(K+1)\frac{\theta}{2\sigma^2}\left(t-\overline{R}_{K}(\theta) \right)\right).
										\label{high-proba-bound} 
										\eeq
									}Furthermore, for any $t\geq 0$,
									{\small
										\beq \mathbb{P}\left( f(\bar{x}_K) - 
										\sa{f_*}\geq t\right) \leq \exp\left(-(K+1) \bar{I}_K(t)\right),
										\label{large-deviation-bound}
										\eeq 
									}where
									{\small
										\beq
										\bar{I}_K(t)\sa{:=}\sup_{0\leq \theta<1/\overline{H}_\infty^2} \frac{\theta}{2\sigma^2}\left(t-\overline{R}_{K}(\theta)\right)\geq 
										\max\left(  \check c_K^{(1)}  \Psi\left(\frac{t}{ \check c_K^{(1)} }, \check a_K^{(1)} ,  \check b_K^{(1)} \right), \Psi\left(t, \check a_K^{(2)},  \check b_K^{(2)}\right)\right),
										\label{my-large-deviation-bound-2}
										\eeq 
									}
									with
									$$ \check a_K^{(1)} = \frac{\bigl(\tfrac{1-\lambda_+^{K+1}}{1-\lambda_+}-1+J_{p,q}\bigr)V(\xi_0)}{\bigg(c_1+\tfrac{2}{L}\tilde{r}(\tilde P)\bigg)K}, \quad \check b_K^{(1)} =\frac{(1-\lambda_+^K)(\overline{H}_\infty)^2}{J_{p,q}},
									\quad \check c_K^{(1)} = \frac{K}{(K+1)J_{p,q}}, 
									$$ 
									$$ \check a_K^{(2)} = \frac{\bigl(\tfrac{1-\lambda_+^{K+2}}{1-\lambda_+}-1+J_{p,q}\bigr)V(\xi_0)}{\bigg(c_1+\tfrac{2}{L}\tilde{r}(\tilde P)\bigg)(K+1)}, \quad \check b_K^{(2)}= (1-\lambda_+^{K+1})(\overline{H}_\infty)^2, 
									$$ 
									and 
									\beq \Psi(t,\check a,\check b)
									=
									\begin{cases}
										\displaystyle \frac{t-\check{a}}{2\sigma^2 \check b} \sqrt{1 - \frac{4\sigma^2 \check b}{t-\check a}} - \log\left( \frac{1 + \sqrt{1 - \frac{4\sigma^2 \check b}{t-\check a}}}{1 - \sqrt{1 - \frac{4\sigma^2 \check b}{t-\check a}}} \right), & \text{if} \quad t -\check a \geq 4\sigma^2 \check b; \\
										0, & \text{if}\quad 0\leq t -\check a < 4\sigma^2 \check b.
									\end{cases}
									\label{str-cvx-dual}
									\eeq
									
								\end{theorem}
								\begin{proof} The first inequality in \eqref{high-proba-bound} is a consequence of the convexity of $f$. {The second inequality is trivial for $\theta=0$, and for $(0,1/(\overline{H}_\infty)^2)$} it is \mgc{simply} a Chernoff-type bound.~More specifically, for $t\geq 0$, and $\theta \in (0,1/(\overline{H}_\infty)^2)$,
									\beqs 
									\mathbb{P}\left(\frac{S_K}{K+1}\geq t\right)
									&=&
									\mathbb{P}\left(\exp\left(\frac{\theta}{2\sigma^2}S_K\right)\geq \exp\left(\frac{\theta}{2\sigma^2}(K+1)t\right)\right)\\
									&\leq& 
									\exp\left(-\frac{\theta}{2\sigma^2}(K+1)t\right)
									\mathbb{E}\left(
									\exp\left(\frac{\theta}{2\sigma^2}S_K\right)
									\right)\\
									&\leq & 
									\exp\left(-(K+1)\frac{\theta}{2\sigma^2}\left(t-\overline{R}_{K}(\theta) \right)\right),
									\eeqs
									where we used $\overline{R}_{K}(\theta)\geq R_{K}(\theta)$. Then, \eqref{large-deviation-bound} is obtained by taking supremum over $\theta$.
									Finally, \sa{the lower bound on $\bar{I}_K(t)$ given in \eqref{my-large-deviation-bound-2}} is a consequence of Lemma \ref{lemma-finite-time-rate-function} in the appendix. 
									
								\end{proof}
								
								\sa{For} the large $K$ limit in the previous result, we can obtain the following large deviation bound. 
								\looseness=-1
								\begin{corollary} 
									\sa{Under the premise} of \cref{thm-str-cvx-risk-cost-bound}, it holds that
									{
										\beq 
										\limsup_{K \to \infty}
										\frac{1}{K} \log \mathbb{P}\left( f(\bar{x}_K) - 
										\sa{f_*} \geq t \right)
										\leq -\bar{I}(t), \quad 
										\label{ineq-str-cvx-rate-function-bd}
										\eeq
										where
										{\small
											\beq \bar{I}(t) \sa{:=} 
											\begin{cases}
												\displaystyle \frac{t}{2\sigma^2 (\overline{H}_\infty)^2} \sqrt{1 - \frac{4\sigma^2 (\overline{H}_\infty)^2}{t}} - \log\left( \frac{1 + \sqrt{1 - \frac{4\sigma^2 (\overline{H}_\infty)^2}{t}}}{1 - \sqrt{1 - \frac{4\sigma^2 (\overline{H}_\infty)^2}{t}}} \right), & \text{if} \quad t \geq 4\sigma^2 (\overline{H}_\infty)^2; \\
												0, & \text{if}\quad t < 4\sigma^2 (\overline{H}_\infty)^2.
											\end{cases}
											\label{eq-rate-function-str-cvx}
											\eeq
										}
									}
								\end{corollary}
								\begin{proof}
									Taking limit superior of both sides in \eqref{high-proba-bound}, for any \sa{$\theta\in [0, 1/\overline{H}_\infty^2)$,} 
									{
										\beqs
										\limsup_{K \to \infty}
										\frac{1}{K} \log \mathbb{P}\left( f(\bar{x}_K) - f(x_*) \geq t \right)
										&=&
										-\frac{\theta}{2\sigma^2} \left( t - \limsup_{K \to \infty} \overline{R}_{K}(\theta) \right).
										\eeqs
									}
									Finally, taking supremum over $\theta$ and using \cref{cor-asymptotic-risk-index-bound}, we see that \eqref{ineq-str-cvx-rate-function-bd} holds with 
									{\small
										$$\bar{I}(t)
										= \sup_{0 < \theta < 
												1/\overline{H}_\infty^2} \mgbis{\overline{F}}(\theta):=\left[ \frac{\theta t}{2\sigma^2} - \log\left( \frac{1 + \theta (\overline{H}_\infty)^2}{1 - \theta (\overline{H}_\infty)^2} \right) \right].
											$$
										}%
										\sa{Moreover, $\overline{F}(\cdot)$ is differentiable for $0 < \theta < 
												1/\overline{H}_\infty^2$ with} the derivative
											{\small
												\beq
												\mgbis{\overline{F}}'(\theta) = \frac{t}{2\sigma^2} - \frac{2(\overline{H}_\infty)^2}{1 - \left(\theta (\overline{H}_\infty)^2\right)^2},
												\eeq}%
											\sa{being} negative for all $\theta \in (0,~1/\overline H_\infty^2)$ when $t<4\sigma^2\overline H_\infty^2$ in which case we get $\bar I(t) = \mgbis{\overline{F}}(0) = 0$. If $t\geq 4\sigma^2\ys{(\overline H_{\infty})^2}$, then $\mgbis{\overline{F}}'(\theta)=0$ when 
											$\theta = \theta_*:= 
											\frac{1}{\overline{H}_\infty^2} \sqrt{1 - \frac{4\sigma^2 \overline{H}_\infty^2}{t}}.
											$
											Therefore, we get $\bar{I}(t) = \mgbis{\overline{F}}(\theta_*)$ which implies \eqref{eq-rate-function-str-cvx} as desired.
										\end{proof}
										{
											\begin{remark}
												By Corollary~\ref{cor-asymptotic-risk-index-bound}, 
												$
												\limsup_{K\to\infty}\,\mathbb{E}\!\left[f(\bar{x}_K)-\sa{f_*}\right] 
												= \overline{m} := 4\sigma^2(\overline H_\infty)^2.
												$
												Since this value represents the asymptotic expected suboptimality, \sa{the events 
													with suboptimality below \(\overline{m}\) are \textit{not}} exponentially unlikely; 
												\sa{hence,} 
												\(\bar I(t)=0\) whenever \(t < \overline{m}\), in accordance with \eqref{eq-rate-function-str-cvx}.  
												In contrast, for \(t>\overline{m}\), 
												\eqref{eq-rate-function-str-cvx} shows that \(\bar I(t)\) is strictly decreasing in \(\overline H_\infty\) over the region where \(\bar I(t) > 0\), i.e., when $
												0 < \overline H_\infty < \tfrac{\sqrt{t}}{2\sigma}
												$.
												Therefore, smaller values of \(\overline H_\infty\) correspond to strictly larger decay rates \(\bar I(t)\) at the threshold \(t\), provided that \(t\) exceeds the critical level \(\overline{m} = 4\sigma^2(\overline H_\infty)^2\). This shows that a \sa{smaller bound on} $H_\infty$-norm—reflecting greater robustness to worst-case deterministic noise—leads to sharper concentration properties and faster decay of the tail probabilities 
												\sa{associated with $f(\bar{x}_K)-\sa{f_*}$}. This is broadly consistent with the numerical behavior observed in \cref{fig:rate-fn-nag-beta-opt-vary-stepsize} for quadratics: reducing the step size increases the rate function $I(t)$ while simultaneously decreasing \sa{the bound on $H_\infty$-norm} for NAG-beta-opt, 
												\sa{implied by} \eqref{ineq-hinf-bound-agd-by-hand} in the appendix.
											\end{remark}
										}
										\section{Numerical Experiments}
										
										We consider the following strongly convex smooth objective:
										\[
										f(x) \coloneqq \sum_{i=1}^p \sa{g_\lambda}(a_i^\top x - b_i) +\frac{\mu}2\|x\|^2,
										\]
										where \sa{$\mu > 0$ is a regularization parameter, $a_i \in \R^d$, \sa{$b_i\in\reals$} for $i = 1,...,p$ for some $p \geq 1$ denotes the problem data,} and \sa{$g_\lambda(\cdot)$ denotes} the Huber loss defined as
										\[
										\sa{g_\lambda(x)}
										\coloneqq 
										\begin{cases*}
											\frac12 x^2, & $|x| \leq \lambda$; \\
											\lambda\left(|x| - \frac12 \lambda\right), & $|x| > \lambda$,
										\end{cases*}
										\]
										\sa{with $\lambda>0$ denoting the smoothing parameter of the Huber-loss.} {Such objectives arise in regularized formulations of robust linear regression \cite{filzmoser2021robust}}. {This objective resembles the heavy-ball counter-example analyzed in \cite{lessard2016analysis} and \sa{its} variations 
											have been studied in the literature \cite{scoy-gmm-ieee,gurbuzbalaban2023robustly} to test momentum methods.} 
										
										In our experiments, 
										\sa{for $d = 10, p = 1000$,} the entries of the data matrix $A = [a_1, a_2, ...,a_p] \in \R^{d \times p}$ and the vector $b \in \R^p$ 
										\sa{are i.i.d. with} a uniform distribution on the interval $\left[-1, 1\right]$. We set $\mu = 0.005$ and $L = 20$, and scale $A$ so that $\|A\| = \sqrt{L - \mu}$ \sa{--indeed, with these choice of parameters, $f$ is $\mu$-strongly convex and $\grad f$ is Lipschitz with constant $L$.} \sa{Finally, the smoothing parameter is set to $\lambda = 0.1$.}
										
										\sa{In this experiment we consider} 
										RS-HB, HB-fastest, GD-pop, NAG-beta-opt, and TMM (
										\sa{using} the parameters provided in \cref{tab:algo-params} for each method). On the left 
										\sa{plot} of \cref{fig: risk‐sensitive‐cost-str-cvx}, we consider unbiased stochastic noise setting, where \emph{batching} \sa{is used with batch size $b<p$} to construct an unbiased estimator of the gradient, i.e., the gradients are estimated 
										\sa{with} 
										$\tilde{\nabla}f(y_k) \coloneqq \frac{p}{b} \sum_{\sa{t}=1}^b g_\lambda(a_{\sa{i_t}}^\top y_k + b_{\sa{i_t}}) + \mu y_k$,
										\sa{using} \sa{$i_1,...,i_b$} \sa{that} are sampled with replacement from $\{1,...,\mgc{p}\}$ independently at step $k$. This type of estimation is common in the context of stochastic gradient descent and momentum methods \cite{aybat2019universally,gurbuzbalaban2021fractional,gurbuzbalaban2021heavy}. For each method, we run $n = 100$ sample paths, i.e., for every run \sa{$\ell$, $\{x_j^{(\ell)}\}_{j=0}^K$ is a simulated trajectory and each $\ell\in\{1,2,\dots,n\}$ corresponds to an independent run where we execute the algorithm in consideration 
											with the iteration budget of $K=1000$ iterations.} 
										We estimate the sensitive cost \sa{sequence $\{R_k(\theta)\}_{k}$} numerically by considering the sample averages, \sa{i.e.,}
										\[
										\hat{R}_{k}(\theta)
										\coloneqq
										\frac{2\hat{\sigma}^2}{\theta \mg{(k + 1)}}
										\log\left(
										\frac{1}{n} \sa{\sum_{\ell=1}^n}
										\exp\left(
										\frac{\theta}{2\hat{\sigma}^2}
										\sum_{j=\mg{0}}^{\mg{k}}
										\left[
										f\left(x_j^{(\sa{\ell})}\right) - 
										\sa{f_*}
										\right]
										\right)
										\right),
										\quad
										0 \leq k \leq K.
										\]
										where $\hat{\sigma}^2$ is an estimate of the variance proxy $\sigma^2$ 
										\sa{using} the empirical variance of 
										\sa{stochastic} gradients.\footnote{At iteration $k$, we compute the empirical variance, $\text{emp var}_k$, using 100 samples of the gradient evaluated at $y_k$, each of batch size 64. The final variance proxy estimate is $\hat \sigma^2 \coloneqq \frac14 \max_{1 \leq k \leq K} \text{emp var}_k$ --the scaling by $\frac14$ is from the fact for a $\sigma^2$-subGaussian random variable, $X$, $\text{Var}[X] \leq 4\sigma^2$. This can be seen from the intermediary moment calculation in \cref{lemma-mgf-bound} by taking $\tilde p = 2$.} 
										{On the left 
											\sa{plot} of \cref{fig: risk‐sensitive‐cost-str-cvx}, we 
											\sa{illustrate} $\hat{R}_k(\theta)$ for $\theta \in \{0.001,1000\}$ \mg{in a logarithmic scale}. We observe that GD-pop begins with a higher risk-sensitive cost than its momentum-based competitors, but after 1500 iterations \mgbis{this cost} begins to fall 
											\sa{below the competitors' costs}. This phenomenon reflects the trade-off between convergence rate and risk, as discussed in \cref{sec-quad}. While GD-pop converges more slowly, it can exhibit greater robustness to noise, since the impact of noise accumulates \sa{more gradually over iterations than the others}. Furthermore, both HB-fastest and its robustly-stable variant (RS-HB) converge to higher risk levels than TMM and NAG-beta-opt, with NAG-beta-opt achieving the lowest risk among the momentum-based methods we compared. This behavior is consistent with the patterns observed in the quadratic case, shown in \cref{fig: pareto-boundary}.}\looseness=-1 
										
										On the right 
										\sa{plot} of \cref{fig: risk‐sensitive‐cost-str-cvx}, we retain the same experimental setup as in the left 
										\sa{plot except for}
										one key difference: we introduce a bias term into the stochastic \sa{gradients.} Specifically, at each iteration $k$, we add an adversarial (bias) component to the previously unbiased stochastic gradients. To construct this bias, \sa{for every iteration $k$}, we sample 50 points from a $\delta$-ball centered at $\nabla f(y_k)$ with $\delta = 2.5$. This value of $\delta$ is chosen to clearly illustrate the contrast between the non-adversarial and adversarial noise scenarios. Among the sampled points, we select the one that maximizes the suboptimality $f(x_{k+1}) - 
										\sa{f_*}$. This form of adversarial perturbation is inspired by techniques from adversarial learning \cite{madry2019deeplearningmodelsresistant}. \ys{
											We plot the risk for both $\theta \in \{0.001, 1000\}$. On the right 
											\sa{plot}, we observe the hierarchy in the performance of the algorithms; \sa{moreover,} GD notably settles below its momentum-based competitors earlier (at around 750 iterations) than in the case with no adversarial noise. 
											
										}
										

										\begin{figure}[h!]
											\centering
											\includegraphics[width=0.98\linewidth]{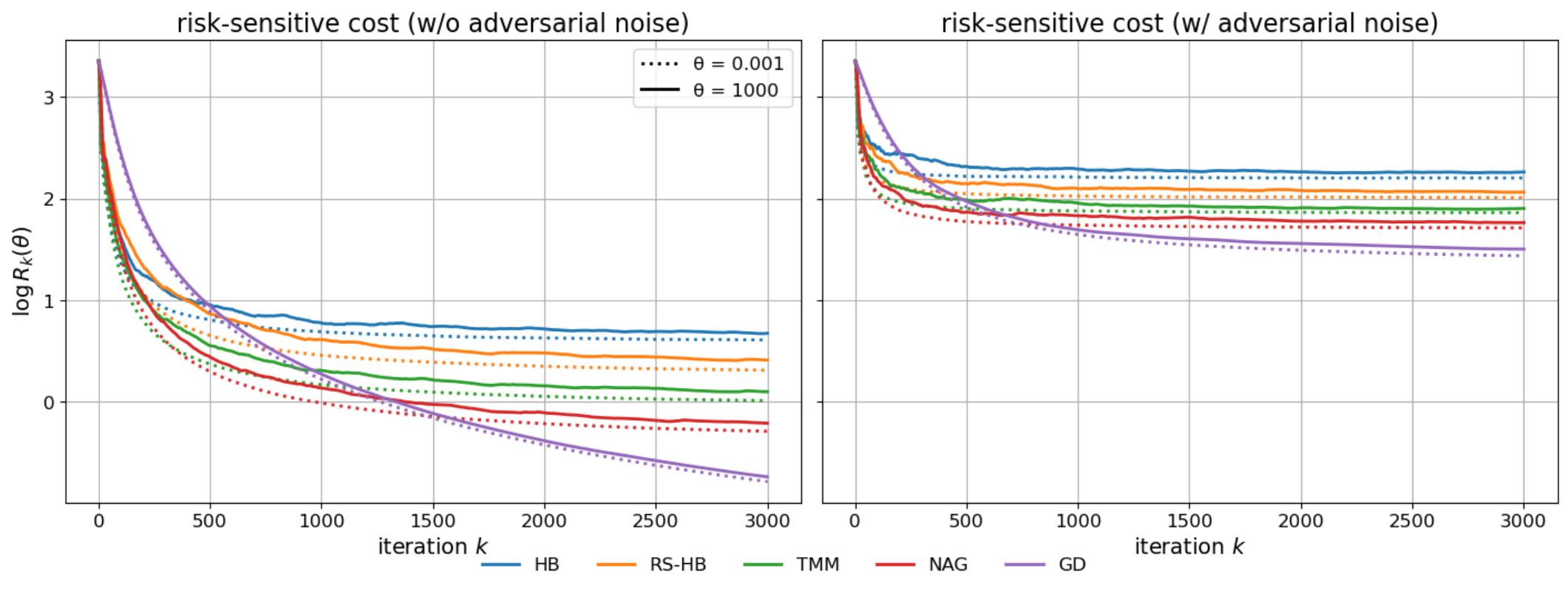}
											\caption{The risk-sensitive cost for GD settles the lowest among its popular competitors, HB, RS-HB, GD, and NAG in the case where there is pure stochastic noise and mixed stochastic and adversarial noise; $d = 10, p = 1000, \mu=0.005, L = 20, \theta = \{0.001,1000\}$.}
											\label{fig: risk‐sensitive‐cost-str-cvx}
										\end{figure}
										
										
										
										\section{Conclusion}
										{We analyzed the trade-off between convergence rate and robustness to gradient errors in generalized momentum methods (GMMs), encompassing \sa{Nesterov's} accelerated gradient, heavy-ball, and gradient descent \sa{methods}. Using the concept of 
											\sa{\textit{risk sensitivity index} (RSI)} from robust control, we characterized robustness under stochastic, \sa{potentially biased, additive gradient errors}. For quadratic objectives with \sa{isotropic} Gaussian noise, we derived closed-form RSI expressions via Riccati equations and identified a Pareto frontier between convergence rate and the robustness. We further established a large-deviation principle that \sa{connects the rate function characterizing the associated large-deviation principle to RSI and the $H_\infty$-norm, thus} bridging robustness to stochastic biased noise with robustness to deterministic worst-case errors. Extending our results to strongly convex smooth objectives, beyond quadratics, we provided finite-time RSI bounds under biased sub-Gaussian errors, yielding high-probability and non-asymptotic guarantees, and highlighted similar robustness--convergence trade-offs. To our knowledge, this is the first risk-sensitive, non-asymptotic analysis of GMMs with biased gradients. We provided numerical experiments on robust regression to illustrate our results.}
										\bibliographystyle{elsarticle-num}
										\bibliography{risk1,papers,robust}
										
										\appendix 
										\label{appendix}
										\section{{Existing bounds on the $H_\infty$-norm for 
												\sa{$f\in \Cml$}}}\label{sec-hinfty-appendix}

										\begin{theorem}[\cite{gurbuzbalaban2023robustly}]\label[theorem]{thm-hinfty-bound} Consider \mgb{the inexact} GMM dynamics \sa{in} \eqref{Sys:stoc-RBMM} for minimizing $f\in \Cml$ with fixed parameters $(\alpha, \beta, \nu)$ where $\alpha>0$, and $\beta,\nu \geq 0$, \mgb{when the gradient at step $k$, \sa{$\grad f(y_k)$,} is subject to additive gradient error $w_{k+1}$ \sa{resulting in $\tilde\grad f(y_k,w_{k+1})$}}. 
											\mgb{Assume that there exists non-negative scalars $\rho_0, \rho_1, \rho_2,\rho_3 \in [0,1),  a, b, c_0, c_1\geq 0$ and a positive semi-definite matrix $\tilde{P}\in\mathbb{R}^{2\times 2}$ such that the following conditions hold:} 
											\begin{itemize}
												\item [$(i)$] \mgb{$
													p+q < 1$, where 
													\beq p \coloneqq \mgb{\rho_0^2} + c_1\rho_1^2 + \rho_2^2
													+ \frac{4b(1 + \nu)^2L^2}{\mu}c_1, \qquad 
													q \coloneqq \rho_3^2 + \frac{4b\nu^2L^2}{\mu}c_1;\label{defn-pq}\eeq
												}
												\item [$(ii)$] $c_1 + \frac{2}{L} \tilde{r}(\tilde{P})>0$ \sa{and} $c_1 + \tilde{P}_{11}>0$ \sa{such that}
												\beq   \tilde{M}_2(\tilde{P},a) +c_1 \tilde{M_1} + c_0 \tilde{M}_0 \succeq 0\,,\label{constraint-sdp}
												\eeq\end{itemize}
											where $\tilde{M}_0, \tilde{M}_1$ and $\tilde{M}_2(\tilde P,a)$ are $4\times 4$ real symmetric matrices defined as
											{
												\beqs
												\mbox{$\tilde{M}_0$} = \begin{bmatrix}   2\mu L\tilde{C}^\top \tilde{C} & -(\mu+L)\tilde{C}^\top & 0_{2 \times 1}\\ -(\mu+L)\tilde{C} & 2 & 0 \\
													0_{1 \times 2} & 0 & 0 
												\end{bmatrix},\quad \tilde{M}_1 = 
												\left[ \begin{array}{@{}c|c@{}}
													\tilde{X} + \tilde{Z}
													& \begin{matrix} \frac{L\alpha (\beta - \nu)}{2} \tilde{M}_3^\top  \\ \frac{\alpha (1-L\alpha)}{2} \end{matrix} \\
													\cmidrule[0.4pt]{1-2}
													\begin{matrix} \frac{L\alpha (\beta - \nu)}{2} \tilde{M}_3 & \frac{\alpha (1-L\alpha)}{2} \end{matrix}  & 0 \\
												\end{array} \right],
												\eeqs
											}
											\beqs
											\mbox{$\tilde{M}_2(\tilde{P},a)$} =  \begin{bmatrix}	 -\tilde{A}^\top  \tilde{P} \tilde{A} + \rho_0^2 \tilde{P}   & - \tilde{A}^\top  \tilde{P} \tilde{B}  &  - \tilde{A}^\top  \tilde{P} \tilde{B}  \\
												- \tilde{B}^\top  \tilde{P}  \tilde{A}   &  - \tilde{B}^\top  \tilde{P} \tilde{B} +bc_1   &   - \tilde{B}^\top  \tilde{P}  \tilde{B}   \\
												- \tilde{B}^\top  \tilde{P} \tilde{A}   & - \tilde{B}^\top  \tilde{P} \tilde{B}  & a 
											\end{bmatrix}, \\
											\eeqs 
											with  
											$$ \tilde{M}_3 := \begin{bmatrix} 1 & - 1\end{bmatrix}, \qquad \tilde{Z} = \begin{bmatrix} \rho_1^2 P_{11} +\frac{\mu}{2} \rho_2^2 &  \rho_1^2 P_{12} & 0\\
												\rho_1^2 P_{12} & \rho_1^2 P_{22} +\frac{\mu}{2}\rho_3^2 & 0 \\
												0   &				0			& 0
											\end{bmatrix},$$
											\beq 
											~\tilde{X} = \tilde{X}_1 + \rho_0^2 \tilde{X}_2 + (1-\rho_0^2) \tilde{X}_3,\qquad  
											\tilde{X}_1 = \frac{1}{2}
											\begin{bmatrix}
												-L\mg{\tilde{\Delta}}^2 & L\mg{\tilde\Delta}^2 & -(1-\alpha L)\mg{\tilde{\Delta}} \\
												L\mg{\tilde\Delta}^2 & -L\mg{\tilde{\Delta}}^2 & (1-L\alpha)\mg{\tilde{\Delta}} \\
												-(1-\alpha L)\mg{\tilde{\Delta}} & (1-L\alpha)\mg{\tilde{\Delta}} & \alpha(2-L\alpha)
											\end{bmatrix},
											\label{def-tilde-X}
											\eeq 
											\begin{align*}
												\tilde{X}_2 = \frac{1}{2}
												\begin{bmatrix} 
													\nu^2\mu& -\nu^2\mu & -\nu \\
													-\nu^2\mu & \nu^2\mu & \nu  \\
													-\nu  & \nu  & 0
												\end{bmatrix},\qquad 
												\tilde{X}_3 = \frac{1}{2}
												\begin{bmatrix} 
													(1+\nu)^2\mu & -\nu(1+\nu)\mu & -(1+\nu) \\
													-\nu(1+\nu)\mu & \nu^2\mu & \nu  \\
													-(1+\nu)  & \nu  & 0 
												\end{bmatrix},
											\end{align*}
											\sa{and} $\mg{\tilde{\Delta}}:= \beta - \nu$; \sa{moreover,} $\tilde{A},\tilde{B},\tilde{C}$ matrices are as in \eqref{def: system mat for TMM} and 
											\beq  \label{def-r-P}
											\tilde{r}(\tilde P) = \begin{cases}
												\tilde{P}_{11} - {\tilde{P}_{12}^2}/{\tilde{P}_{22}} & \mbox{if} \quad \tilde{P}_{22}\neq 0, \\
												\tilde{P}_{11} & \mbox{if} \quad \tilde{P}_{22}=  0.
											\end{cases}
											\eeq 
											Then, \mg{the Lyapunov function $V_{P, c_1}$ defined in \eqref{eq-lyap-gmm-general} with $P =\tilde P\times \mathsf{I}_d$ satisfies the \mgb{decay property}}
											\beq\label{lyapunov-descent}
											V_{P,c_1}(\xi_{k+1})
											\leq 
											pV_{P,c_1}(\xi_k)
											+
											qV_{P,c_1}(\xi_{k-1})
											+
											r\|w_{k+1}\|^2,
											\eeq
											\mgb{for \sa{all} $k\geq 0$ along the GMM iterates where \mgbis{$p,q$ are as in \eqref{defn-pq}}, and
												$r \coloneqq 
												a + \alpha^2\left(\frac{c_1L}{2} + \tilde P_{11}\right)$ with the convention that $\xi_{-1}:=\xi_0$}. \mgb{Furthermore, in the special case when the sequence $\{w_{k+1}\}_{k\geq 0}$ is deterministic and square-summable, the $H_\infty$-norm of the GMM defined in \eqref{def-Hinfty} is finite and admits the bound}
											\beq \quad\quad 
											\mgb{
												H_\infty \leq \overline{H}_\infty:= \bigg(
												\tfrac{1 }{1-\big(p+q\big)} {  \tfrac{{
															r } }{{ c_1 + \frac{2}{L} \tilde{r}(\tilde{P}) }} 
												} 
											}
											\bigg)^{1/2} \label{def-Hinfty-ub}.
											\eeq
											\mgbis{More specifically, it satisfies \eqref{def-l2-gain-general} for $\gamma=\overline{H}_\infty$ with}
											\ys{
												\[
												\mgbis{H_\gamma(\xi_0)}
												=
												\frac{1 + \frac{4b\nu^2L^2}{\mu}c_1 + \rho_3^2}{1 -(p+q)}\left(\frac{4\|\tilde P\|}{\mu} + c_1\right) \frac{1}{c_1 + \frac2L\tilde r(\tilde P)} (f(x_0) - 
												\sa{f_*}).
												\]}
											
										\end{theorem} 
										\begin{proof} \mgb{The result follows directly from Theorem 5.6 and the inequalities (47)-(48) in \cite{gurbuzbalaban2023robustly}.}
										\end{proof}
										The following results are corollaries of the previous theorem, providing explicit bounds on the $H_\infty$-norm for GD and NAG.
										\begin{corollary}[\cite{gurbuzbalaban2023robustly}]\label[corollary]{cor-GD-params-MI}
											For GD with step size with $\alpha\in (0,2/L)$, the descent property \eqref{lyapunov-descent} and the bound \eqref{def-Hinfty-ub} holds with
											$$p=\rho_{GD}(\alpha):=\max\{|1-\alpha\mu|,|1-\alpha L|\},\quad q=0, \quad r = \frac{\alpha^2}{1-\rho_{GD}(\alpha)}, \quad \tilde P=\mathsf{I}_2, \quad c_1 = 0.$$ 
											In addition, 
											\eqref{lyapunov-descent} and 
											\eqref{def-Hinfty-ub} \sa{also hold with an alternative set of parameters:}
											$$p=1-2\mu\alpha(1-\frac{L\alpha}{2})+ \alpha\mu|1-\alpha L|s_{\scriptscriptstyle \mathrm{GD}}, \quad q=0, \quad r = \alpha\left(\frac{L}{2}\alpha + \frac{|1-\alpha L|}{2s_{\scriptscriptstyle \mathrm{GD}}}\right), \quad \tilde{P}=0, \quad c_1 = 1,$$
											where 
											$s_{\scriptscriptstyle \mathrm{GD}}=1$ when $\alpha\leq \frac{1}{L}$ and $s_{\scriptscriptstyle \mathrm{GD}}=\frac{2-\alpha L}{\alpha L}$ for $\alpha\in (1/L,2/L)$. Therefore, \cref{thm-hinfty-bound} implies
											$$\small H_\infty \leq \overline{H}_\infty=  \begin{cases} 
												\frac{1}{\sqrt{2\mu}} & \mbox{if} \quad 0 < \alpha \leq \frac{1}{L}, \\
												\frac{1}{\sqrt{2\mu}}\frac{\alpha L}{2-\alpha L} & \mbox{if} \quad \frac{1}{L} <\alpha \leq \frac{2}{L+\sqrt{L\mu}},\\
												\frac{1}{\sqrt{2\mu}}  \sqrt{\kappa} & \mbox{if} \quad  \frac{2}{L+\sqrt{L\mu}} < \alpha \leq \frac{2}{L+\mu}, \\
												\frac{\alpha\sqrt{L}}{\sqrt{2}(2-\alpha L)}  & \mbox{if} \quad  \frac{2}{L+\mu}<\alpha < \frac{2}{L}. 
											\end{cases}
											$$
										\end{corollary}
										\begin{proof}
											See \cite[Section 5]{gurbuzbalaban2023robustly} for details. 
										\end{proof}
										\begin{corollary}[\cite{gurbuzbalaban2023robustly}]\label[corollary]{cor-NAG-params-MI}
											For NAG with $\alpha\in(0,1/L]$ and $\beta = \nu = \frac{1-\sqrt{\alpha\mu}}{1+\sqrt{\alpha \mu}}$, the descent property \eqref{lyapunov-descent} and the bound \eqref{def-Hinfty-ub} holds for
											$$ p= \rho^2_{\tiny \mbox{NAG}} +  \frac{8\mg{\alpha}\sqrt{\alpha} L^3}{\mu (1+\sqrt{\alpha\mu})^2 \hat{s}_2}  + \frac{\sqrt{\alpha}\mg{(4-\sqrt{\alpha\mu} + \alpha\mu)}}{\mg{2}\hat{s}_1 (1+\sqrt{\alpha\mu})^2 },  $$
											$$ q= \frac{\sqrt{\alpha} (1- \sqrt{\alpha\mu}) }{2\hat{s}_1  (1+\sqrt{\alpha\mu})^2 } + \frac{2\alpha \sqrt{\alpha} L^3 \beta^2}{\hat{s}_2 \mu},\quad r= \frac{\alpha( 1 + \alpha L)}{2} + \sqrt{\alpha}\hat{s}_1 + \frac{\mg{\alpha^2} \sqrt{\alpha} L}{2}\hat{s}_2,$$
											$$c_1 = 1, \quad \tilde P = \frac1{2\alpha}\begin{bmatrix}1 \\ -(1-\sqrt{\alpha\mu})\end{bmatrix}\begin{bmatrix}1 & -(1-\sqrt{\alpha\mu})\end{bmatrix}, $$
											where 
											\[
											\hat s_1 = \frac{2(5-2\sqrt{\alpha\mu}+\alpha\mu)}{\sqrt{\mu}(1+\sqrt{\alpha\mu})^2}, \quad \hat s_2 = \frac{8L^3\alpha}{\mu\sqrt{\mu}(1+\sqrt{\alpha\mu})^2}(4+(1-\sqrt{\alpha\mu})^2),
											\]
											\ys{and $\rho_{NAG}^2 \coloneqq 1 - \sqrt{\alpha\mu}$}. Therefore, \cref{thm-hinfty-bound} provides the bound
											{\small
												\beq  \hspace{-0.13in} H_\infty &\leq& \overline{H}_\infty= 
												\sqrt{ \frac{4 \mg{(5-2\sqrt{\alpha\mu} + \alpha\mu)} }{ \mu (1+\sqrt{\alpha\mu})^2 } +
													\frac{\sqrt{\alpha}( 1 + \alpha L)}{\sqrt{\mu}} 
													+ \frac{8 \mg{\alpha^3} L^4\left(4+(1-\sqrt{\alpha\mu})^2\right) }{ \mu^2 (1+\sqrt{\alpha\mu})^2 }  }.
												\label{ineq-hinf-bound-agd-by-hand} 
												\eeq
											}
										\end{corollary}
										\begin{proof}
											See \cite[Section 5]{gurbuzbalaban2023robustly} for details.
										\end{proof}
										\section{Proof of Theorem \ref{thm-gmm-risk-formula}}\label[appendix]{sec-proof-theorem-risk-quadratic}
										We first present a lemma that characterizes the risk-sensitive index in terms of the solutions to the $d\times d$ DARE \sa{in}~\eqref{dare}.
										\begin{lemma}\label[lemma]{lem-quad-risk-index} In the setting of \cref{thm-gmm-risk-formula},  
											the risk-sensitive index of the GMM algorithm \sa{in~\eqref{RBMM}} for a given risk parameter $\theta>0$ admits the formula,
											$$ R(\theta) = 
											\begin{cases} 
												-\frac{\sigma^2}{\theta} \log \det (\mathsf{I}_\ys{d} - \frac{\theta}{d} B^\top  X B) & \mbox{if} \quad \sqrt{\theta} H_\infty < \sqrt{d}, \\
												+\infty  & \mbox{if} \quad \sqrt{\theta} H_\infty \geq \sqrt{d},
											\end{cases} 
											$$
											where $H_\infty$ is given by \eqref{hinfty-quad}, and $X$ is the ``stabilizing solution" of the DARE \sa{in} \eqref{dare} for $\gamma = \frac{1}{\sqrt{\theta/d}}$, i.e., $X$ is a solution and has the property that
											$ \rho(A_Q + \gamma^{-2} B B^\top  X ) < 1 $ 
											for $\gamma = \frac{1}{\sqrt{\theta/d}}$. 
											
										\end{lemma} 
										\begin{proof} \mgb{Since $f$ is a quadratic, the GMM system \eqref{Sys:quad-deterministic-noise}--\eqref{def-T} is a linear dynamic system described by the system matrices $(A_Q, B,T)$. For linear systems, it is well-known that the $H_\infty$-norm admits the frequency domain characterization}  
												\beq H_\infty = \sa{\norm{G}_\infty}:=\max_{\{z\in\mathbb{C}:\|z\|=1\}}\| G(z)\|\,,\label{def-H-infty-freq}
												\eeq
												(see e.g., \cite{gurbuzbalaban2023robustly}); 
												where $G$ is the transfer matrix of the system $(A_Q, B, T)$ defined as 
												$G(z) \coloneqq T(z\mathsf{I}_{2d} - A_Q)^{-1}B$ for $z\in \mathbb{C} \sa{\setminus}\mbox{\text{Spec}}(A_Q)$ with $\mbox{\text{Spec}}(A_Q):=\bigcup_i \{\lambda_i\}$ denoting the set of eigenvalues of $A_Q$. \mgb{Thus, we can conclude that} $\|G\|_\infty=H_\infty$ and $H_\infty$ is given by \eqref{hinfty-quad}. 
												
												On the other hand, by scaling the gradient noise, we can rewrite the \mg{GMM system \eqref{Sys:quad-deterministic-noise}--\eqref{def-T}}:
												\beq 
												{\xi}_{k+1}^c &= A_Q {\xi}^c_{k} +  B_\sigma \zeta_{\ys{k+1}}, \quad z_k &= T {\xi}_k^c, \label{Sys:stoc-RBMM-scaled}
												\eeq
												where $B_\sigma = \mgb{\frac{\sigma}{\sqrt{d}}} B$ and $\zeta_k \sim \mathcal{N}(0, \sa{\mathsf{I}_d})$ is \sa{a 
													normal random vector with zero mean and unit variance}. Glover and Doyle showed that the risk-sensitive index of this system with the system matrices $(A_Q, B_\sigma, T)$ can be represented by the following integral 
												\beq R(\theta) = -\frac{\sigma^2}{2\pi \theta} \int_{-\pi}^\pi \log \det \left(\mathsf{I}_\ys{d} - \frac{\theta}{\sigma^2} G_\sigma(e^{i\omega}) G_\sigma^*(e^{i\omega})\right) d\omega\,,
												\label{eq-centered-iter}
												\eeq
												\sa{for $\theta\in\mgbis{\mathbb{R}}$ such that $\theta \|G_\sigma\|_\infty^2 < \sigma^2$; otherwise, $R(\theta)=+\infty$} (see \cite[Section 3]{glover1988state})\footnote{\mgc{While \cite{glover1988state} primarily focuses on the case $\theta > 0$, 
														the discussion and the formula for the risk-sensitive index in 
														\cite[Section~3]{glover1988state} also apply when $\theta \leq 0$, 
														since they are derived from \cite[Lemmas~3.1 and~3.2]{glover1988state}, whose validity extends to all $\theta \leq 0$.}} where $G_\sigma$ is the transfer \mgbis{matrix} of the system $(A_Q, B_\sigma, T)$ \footnote{Here, $A_Q \coloneqq A + BQC$ where $B$ is not scaled by $\sigma$, as $A_Q$ is our iteration matrix from \eqref{def-AQ}.} and $B_\sigma = \mgb{\frac{\sigma}{\sqrt{d}}} B$ --with the convention that the integral \eqref{eq-centered-iter} for $\theta=0$ is to be interpreted by continuity according to
												\small
												\begin{equation*} 
													R(0)\sa{:=} \lim_{\theta\downarrow 0} -\frac{\sigma^2}{2\pi \theta} \int_{-\pi}^\pi \log \det \left(\mathsf{I}_\ys{d} - \frac{\theta}{\sigma^2} G_\sigma(e^{i\omega}) G_\sigma^*(e^{i\omega})\right) d\omega =\frac{1}{2\pi}\int_{-\pi}^\pi \operatorname{trace} \left( G_\sigma(e^{i\omega}) G_\sigma^*(e^{i\omega}) \right)d\omega.
												\end{equation*} 
												\sa{Since} $G_\sigma = \mgb{\frac{\sigma}{\sqrt{d}}} G$, 
												\mgbis{for $\theta\|G\|_\infty^2 < d$}, it holds that 
												\beq  R(\theta) &=& -\frac{\sigma^2}{2\pi \theta} \int_{-\pi}^\pi \log \det \left(\mathsf{I}_\ys{d} - \mgb{\frac{\theta}{d}} G(e^{i\omega}) G^*(e^{i\omega})\right) d\omega,
												\label{entropy-integral}
												\eeq
												\mgbis{with the convention that $R(0) = \frac{\sigma^2}{2\pi d}\int_{-\pi}^\pi \operatorname{trace} \left( G(e^{i\omega}) G^*(e^{i\omega}) \right)d\omega$}
												and \sa{$R(\theta)=+\infty$ for $\theta>0$ such that $\sqrt{\theta} \|G\|_\infty \geq \sqrt{d}$. Indeed, more compactly we can write} 
												\beq
												R(\theta)=\frac{\sigma^2}{d} \mathrm{Ent} \left(G, \frac{1}{\mgb{\sqrt{\theta/d}}}\right)\,, \label{eq-entropy-vs-risk}
												\eeq
												where \sa{$\mathrm{Ent}(G,\gamma)$ denotes the entropy integral of the transfer matrix $G$ with respect to the parameter $\gamma>0$ \cite{stoorvogel1993discrete}, i.e.,}
												\beq \mathrm{Ent}(G,\gamma) \sa{:=} -\frac{\gamma^2 }{2\pi} \int_{-\pi}^\pi \log \det (\mathsf{I}_\ys{d} - \gamma^{-2}G(e^{i\omega})G^*(e^{i\omega}))d\omega,
												\eeq 
												\sa{which can be computed as follows:}
												\beq\label{entropy} \mathrm{Ent}(G,\gamma) = \begin{cases} 
													-\gamma^2 \log \det (\mathsf{I}_\ys{d} - \gamma^{-2} B^\top  X  B) & \mbox{if} \quad \|G\|_\infty < \gamma,  \\
													\infty				&		\mbox{if} \quad \|G\|_\infty \geq  \gamma, 
												\end{cases}	
												\eeq
												where $X$ is the stabilizing solution to the discrete-time Riccati equation \eqref{dare}, i.e., $X$ solves \eqref{dare} and satisfies
												$ \rho(A_Q + \gamma^{-2} B B^\top  X ) < 1 $
												(see e.g., \cite{iglesias2000entropy}, \cite[Remark 4.11]{peters2012minimum}, \cite[Lemma 3.4]{stoorvogel1993discrete}). 
												Applying this formula to \eqref{eq-entropy-vs-risk} \mgb{and using $\|G\|_\infty = H_\infty$}, we conclude.  
											\end{proof}
											We next present a lemma that establishes a connection between the $2 \times 2$ solutions of the DARE given in \eqref{eq-small-riccati-bis} and the $2d \times 2d$ solutions of the larger DARE provided in \eqref{dare}.
											\begin{lemma}\label[lemma]{prop-riccati-dim-reduction} In the setting of \cref{thm-gmm-risk-formula}, consider the $2d\times 2d$ DARE \sa{in}
												\eqref{dare}\sa{, i.e.,} 
												\begin{equation*} X = A_Q^\top  X A_Q + A_Q^\top  X B (\gamma^2 \mathsf{I}_d - B^\top  X B)^{-1} B^\top  X A_Q + T^\top T,
												\end{equation*}
												for $\gamma= \frac{1}{\sqrt{\theta/d}}$ where the matrices $A_Q, B$ and $T$ are defined in~\eqref{def-AQ}, \eqref{def-ABC} and \eqref{def-T}, respectively. \mg{Assume that the parameter $\theta$ satisfies $\sqrt{\theta}H_\infty < \mgb{\sqrt{d}}$ where $H_\infty$ is as in \eqref{hinfty-quad}}. \mg{Then, the stabilizing solution $\overline X\in\mathbb{R}^{2d\times 2d}$ to this $2d\times 2d$ DARE is unique and is given by}
												$$\overline X := Y \underset{i=1,...,d}{\Diag}( \tilde{X}^{(\lambda_i})) Y^\top  \in \mathbb{R}^{2d\times 2d}\,,$$
												where $\tilde{X}^{(\lambda_i)}\in\mathbb{R}^{2\times 2}$ is the stabilizing solution to \eqref{eq-small-riccati-bis} for $\gamma = \frac{1}{\sqrt{\theta/d}}$ and $\lambda=\lambda_i$, and 
														$Y = V\Sigma \in \mathbb{R}^{2d\times 2d}$ is an orthonormal matrix with \mg{$V:=\mbox{Diag}(U,U)$ and $U$ is the orthonormal matrix defined in \eqref{def-U-matrix}}, 
														and $\Sigma\in\mathbb{R}^{2d\times 2d}$ is the permutation matrix corresponding to the permutation $\pi$ over the set $\{1,2,\dots,2d\}$ that satisfies $\pi(i)= 2i-1$ for $1\leq i \leq d$ and $\pi(i)= 2(i-d)$ for $d+1\leq i \leq 2d$, i.e., $\Sigma$ is such that $\Sigma_{ij} = 1$ if $j=\pi(i)$; otherwise, $\Sigma_{ij} = 0$.
													\end{lemma} 
													\begin{proof} 
														We first consider the $2\times 2$
														DARE \sa{in} \eqref{eq-small-riccati-bis}, \sa{i.e.,}
														\begin{equation}\tilde{X} := (\tilde{A}^{(\lambda_i)})^\top  \tilde{X} \tilde{A}^{(\lambda_i)} +   (\tilde{A}^{\lambda_i)})^\top  \tilde{X}
															\tilde{B}(\gamma^2  - \tilde{B}^\top  \tilde{X}\tilde{B})^{-1}\tilde{B}^\top 
															\tilde{X}
															\tilde{A}^{(\lambda_i)}
															+ {\tilde Q}^{(\lambda_i)}\,, 
															\label{eq-small-riccati-bis-copy}
														\end{equation}
														for $i = 1,...,d$, where
														$${\tilde{Q}}^{(\lambda_i)} = \begin{bmatrix} \frac{\lambda_i}{2}& 0 \\ 0 & 0 \end{bmatrix} = {\big({\tilde{T}}^{(\lambda_i)}\big)}^\top {{\tilde{T}}^{(\lambda_i)}} \quad \mbox{with} \quad {{\tilde{T}}^{(\lambda_i)}}:= 
														\begin{bmatrix} \frac{\sqrt{\lambda_i}}{\sqrt{2}}& 0 \\ 0 & 0 \end{bmatrix}.
														$$
														\mg{We will first argue that a stabilizing solution $\tilde{X}^{(\lambda_i)}$ exists and is unique. For any $i\in\{1,2,\dots,d\}$, the existence of \sa{a solution to} the DARE in~\eqref{eq-small-riccati-bis-copy} can be inferred from the transfer matrix of the system $(\tilde{A}^{(\lambda_i)}, \tilde{B}, T^{(\lambda_i)})$, defined as $G^{(\lambda_i)}(z) \sa{:=} T^{(\lambda_i)}(z \mathsf{I}-\tilde{A}^{(\lambda_i)})^{-1}\tilde B$. More specifically, it is known that a stabilizing solution exists \sa{and it is unique} if and only if 
															\beq\gamma>\|G^{(\lambda_i)}\|_\infty=\max_{\{z\in\sa{\mathbb{C}}:~\|z\|=1\}}\| G^{(\lambda_i)}(z)\|,
															\label{cond-existence-soln-small-riccati}
															\eeq 
															(see \cite[Corollary 2.1, Prop. 3.7]{hinrichsen1991stability}, \cite[Lemma 3.4]{stoorvogel1993discrete}) where we also know $\|G^{(\lambda_i)}\|_\infty$ equals the $H_\infty$-norm of the corresponding system $(\tilde{A}^{(\lambda_i)}, \tilde{B}, T^{(\lambda_i)})$ (see \eqref{def-H-infty-freq}). 
															This system subject to a 
															\mgb{given \ys{deterministic} noise input \sa{$\{\tilde{\delta}_k\}_{k\geq 0}$} admits the state-space representation
															}
															$$ \tilde{\xi}_{k+1} = \tilde{A}^{(\lambda_i)} \tilde{\xi}_{k} + \tilde B \tilde{\delta}_k \quad \mbox{where}\quad
															\tilde{\xi}_{k} = \begin{bmatrix} \tilde{x}_k \\ \tilde{x}_{k-1} \end{bmatrix} \in \mathbb{R}^2, 
															$$
															\sa{with the} output $\tilde{z}_{k} = T^{(\lambda_i)} \tilde{\xi}_{k}$ \sa{for $k\geq 0$,} 
															which is equivalent to the recursion: 
															\begin{eqnarray}
																\tilde{x}_{k+1}&=& \tilde{x}_{k}-\alpha (\sa{\tilde{f}_i'}( \tilde{y}_k ) + \tilde{\delta}_k) + \beta(\tilde{x}_{k}- \tilde{x}_{k-1}), \quad \tilde{y}_{k}=\tilde{x}_k+\nu (\tilde{x}_k- \tilde{x}_{k-1}),\\
																\tilde{z}_k &=& \frac{\sqrt{\lambda_i}}{\sqrt{2}}\tilde{x}_{k},\quad |\tilde{z}_k|^2 = \frac{\lambda_i}{2} (\tilde{x}_k)^2,
															\end{eqnarray}
															\sa{where} the auxillary function $\sa{\tilde{f}_i}:\mathbb{R}\to\mathbb{R}$ defined as $\tilde{f}(\tilde{x}):= \frac{1}{2}\lambda_i \tilde{x}^2$ \sa{and $\tilde{f}_i'$ represents the derivative of $\tilde f_i$}.
														This recursion is equivalent to GMM with inexact gradients for minimizing \sa{$\tilde{f}_i$} in dimension one; therefore, applying Theorem \ref{thm-h-inf} to \sa{$\tilde{f}_i$} with $d=1$, $\mu=L=\lambda_i$, it follows that $$\|G^{(\lambda_i)}\|_\infty = \frac{\alpha}{\sqrt{2}} \frac{ \sqrt{\lambda_i}}{\tilde{s}_{\lambda_i}}.$$ 
														Then, by our choice of $\theta$ and $\gamma=\frac{1}{\sqrt{\theta/d}}$, we have $\gamma>\mg{H_\infty} =\max_{1\leq i \leq d} \frac{\alpha}{\sqrt{2}} \frac{ \sqrt{\lambda_i}}{\tilde{s}_{\lambda_i}}\geq \|G^{(\lambda_i)}\|_\infty$ for every $i=1,2,\dots,d$, where in the last equality we used \cref{thm-h-inf}. Therefore, the condition \eqref{cond-existence-soln-small-riccati} holds, and we conclude that the stabilizing solution $\tilde{X}^{(\lambda_i)}$ exists and is unique for every $i=1,2,\dots,d$.
													}
													\mg{Next, we introduce the square matrix}
													\beq \hat{X} := \underset{i=1,...,n}{\Diag}( \tilde{X}^{\ys{(\lambda_i)}}) \in \mathbb{R}^{2d\times 2d},
													\label{def-hat-X}
													\eeq 
													as well as the block diagonal matrices $\hat{B} \in \mathbb{R}^{2d\times d}$ with $\hat{B}=\Diag(\tilde{B}, \tilde{B}, \dots, \tilde{B})$ with $\tilde{B}=\begin{bmatrix} 
														-\alpha &
														0
													\end{bmatrix}^\top$, $\hat{Q} = \underset{i=1,...,d}{\Diag}({\tilde Q}^{(\lambda_i)}) \in \mathbb{R}^{2d\times 2d}$, $\hat{A} =   \underset{i=1,...,d}{\Diag}( \tilde{A}^{(\lambda_i)})\in\mathbb{R}^{2d\times 2d}$. 
															Then, $\hat{X}$ satisfies
															\beq \hat{X} = \hat{A}^\top  \hat{X} \hat{A} +   \hat{A}^\top  \hat{X}
															\hat{B}(\gamma^2 \mathsf{I}_d - \hat{B}^\top \hat{X}\hat{B})^{-1}\hat{B}^\top 
															\hat{X}\hat{A}
															+\hat{Q}\,,
															\label{eq-riccati-hat}
															\eeq
															where we used the fact that $\tilde{X}^{(\lambda_i)} \in \mathbb{R}^{2\times 2}$ is a solution to \eqref{eq-small-riccati-bis}. Since $\tilde{X}^{(\lambda_i)}$ is a stabilizing solution, it also satisfies $\rho(\tilde{A}^{(\lambda_i)} + \gamma^{-2}\tilde{B}\tilde{B}^\top  \tilde{X}^{(\lambda_i)})<1$ for every $i$. Therefore, $\hat{X}$ is a stabilizing solution of \eqref{eq-riccati-hat} satisfying 
															\beq \rho(\hat{A} + \gamma^{-2}\hat{B}\hat{B}^\top  \hat{X})= \max_{1\leq i\leq d} \rho(\tilde{A}^{(\lambda_i)} + \gamma^{-2}\tilde{B}\tilde{B}^\top  \tilde{X}^{(\lambda_i)}) <1.
															\label{hat-X-stabilizing-sol}
															\eeq 
														\mg{Note that $Q = U \Lambda U^\top  = U \big( \underset{i=1,...,d}{\Diag}(\lambda_i)\big) U^\top$ and we have
															$$ V^\top A_Q V  =  \begin{bmatrix} 
																(1+\beta)\mathsf{I}_d -\alpha(1+\nu)\Lambda & -\beta \mathsf{I}_d +\alpha\nu \Lambda \\
																\mathsf{I}_d & 0_d 
															\end{bmatrix} = \Sigma \hat{A} \Sigma^{\top},$$ 
															where $A_Q$ is as in \eqref{def-AQ}}. 
														By similar computations, it also follows that 
														\beq Y\hat{A}Y^\top  = A_Q, \quad Y\hat{Q}Y^\top  = T^\top  T. 
														\label{eq-AQ-transformed}
														\eeq
														If we multiply the equation \eqref{eq-riccati-hat} with $Y$ from left and with $Y^\top $ from right, using the fact that $Y^\top Y = I$, we obtain from \eqref{eq-AQ-transformed} that $\overline X := Y\hat{X} Y^\top $ solves the matrix equation
														\beq X = A_Q^\top  X A_Q + A_Q^\top  X M A_Q^\top  X + T^\top  T 
														\label{eq-riccati-satisfied}
														\eeq  
														with
														$$ M := Y\hat{B}(\gamma^2 \mathsf{I}_d - \hat{B}^\top \hat{X}\hat{B})^{-1}\hat{B}^\top  Y^\top . $$
														Note that, by the definition of $\hat{B}$, it follows that 
														$$ \hat{B}(\gamma^2 \mathsf{I}_d - \hat{B}^\top \hat{X}\hat{B})^{-1}\hat{B}^\top  = \underset{i=1,...,d}{\Diag}(\hat{Z}^{(i)} ), \quad 
														\hat{Z}^{(i)} :=   \begin{bmatrix}
															(\frac{\gamma^2}{\alpha^2} - \hat{X}_{2i-1,2i-1})^{-1} &  0 \\
															0  &  0 
														\end{bmatrix}. $$
														Multiplying this equality from left by $\Sigma$ and from right by $\Sigma^\top $ further, we obtain
														\beqs
														\Sigma\hat{B}(\gamma^2 \mathsf{I}_{2d} - \hat{B}^\top \hat{X}\hat{B})^{-1}\hat{B}^\top \Sigma^\top  
														&=& \begin{bmatrix}  Z  & 0_d \\
															0_d &  0_d 
														\end{bmatrix}, \quad Z = \underset{i=1,...,d}{\Diag}\left((\frac{\gamma^2}{\alpha^2} - \hat{X}_{2i-1,2i-1})^{-1}\right).
														\eeqs
														Multiplying this \sa{equality} with $V$ from left and with $V^\top $ from right, we obtain 
														$$ M = \begin{bmatrix}  UZU^\top   & 0_d \\
															0_d &  0_d 
														\end{bmatrix}. $$
														On the other hand, after a similar computation, we have
														\beq \overline X = Y\hat{X} Y^\top  =   \begin{bmatrix}  U \underset{i=1,...,d}{\Diag}\left( \hat{X}_{2i-1,2i-1}\right) U^\top   & 0_d \\
															0_d &  0_d 
														\end{bmatrix}.\label{eq-X-from-tildeX}
														\eeq
														Then, using the definition of the matrices $B$ and $Z$, 
														\beq B (\gamma^2 \mathsf{I}_d - B^\top  \sa{\overline X} B)^{-1}B^\top   &=& B \left(\gamma^2 \mathsf{I}_d -\alpha^2  U \underset{i=1,...,d}{\Diag}\left( \hat{X}_{2i-1,2i-1}\right) U^\top  \right)^{-1}B^\top  \nonumber 
														\\
														&=& \begin{bmatrix} UZU^\top   & 0_d \\
															0_d &  0_d 
														\end{bmatrix} = M.
														\eeq
														Plugging this expression into \eqref{eq-riccati-satisfied}, we conclude that $\overline X$ satisfies the Riccati equation \eqref{dare} as claimed. 
														Furthermore, using $Y\hat{A}Y^\top  = A_Q$ and $Y Y^\top =\mathsf{I},$ 
														\beq 
														\rho(A_Q + \gamma^{-2} B B^\top  \overline X) &=& \rho \left(Y(\hat{A} +\gamma^{-2} Y^\top  B B^\top  \overline X Y) Y^\top \right) \\
														&=& \rho \left(\hat{A} +\gamma^{-2} Y^\top  B B^\top  \overline X Y\right)\\
														&=& \rho \left(\hat{A} +\gamma^{-2} Y^\top  B B^\top  Y \hat{X} \right) \\ 
														&=& \rho \left(\hat{A} +\gamma^{-2} \hat{B}{\hat{B}}^\top  \hat{X} \right) < 1, \label{ineq-stabilizing-sol-big-riccati}
														\eeq
														where in the second equality we used the fact that $Y$ is orthogonal, in the third \sa{equality} we used $\overline X=Y\hat X Y^\top$ and in the last \sa{equality} we used the fact that $Y^\top  B B^\top  Y = \hat{B}\hat{B}^\top $ and \eqref{hat-X-stabilizing-sol}. We conclude from \eqref{ineq-stabilizing-sol-big-riccati} that $\overline X$ is a stabilizing solution to \eqref{dare}. Finally, if the stabilizing solution exists then it is unique, see \cite{Ionescu}, \mg{\cite[Lemma 3.4]{stoorvogel1993discrete}}. This completes the proof.
																		\end{proof}
																		Now, the stage is set to complete the proof of \cref{thm-gmm-risk-formula}, where we will build on \cref{prop-riccati-dim-reduction} and \cref{lem-quad-risk-index}.
																		Let  $\gamma = \frac{1}{\sqrt{\theta/d}}$. According to \cref{prop-riccati-dim-reduction}, $\overline X = Y \underset{i=1,...,n}{\Diag}( \sa{\tilde{X}^{(\lambda_i)}}) Y^\top  \in \mathbb{R}^{2d\times 2d}$ is a solution to \eqref{dare} satisfying $\rho(A_Q + \gamma^{-2}BB^\top \sa{\overline X} ) < 1$. Therefore, from \cref{lem-quad-risk-index},  
																		\beq\label{eq-risk in terms of hat x} R(\theta) = 
																		\begin{cases} 
																			-\frac{\sigma^2}{\theta} \log \det (\mathsf{I}_d - \frac{\theta}{d} B^\top  Y \underset{i=1,...,d}{\Diag}( \sa{\tilde{X}^{(\lambda_i)}}) Y^\top  B) & \mbox{if} \quad \sqrt{\theta} H_\infty < \sqrt{d}, \\
																			\infty  & \mbox{if} \quad \sqrt{\theta} H_\infty \geq \sqrt{d},
																		\end{cases} 
																		\eeq
																		From the \sa{definitions of $B:=\tilde B\otimes \mathsf{I}_d$ and 
																			$Y:=V\Sigma$}, we obtain
																		$$ B^\top  Y \underset{i=1,...,d}{\Diag}( \sa{\tilde{X}^{(\lambda_i)}}) Y^\top  B   
																		= \alpha^2 U {\Diag}( \sa{\tilde{X}_{11}^{(\lambda_i)}})  U^\top. 
																		$$
																		Consequently, 
																		\beqs  \log \det \left(\mathsf{I}_\ys{d} - \frac{\theta}{d} B^\top  Y \underset{i=1,...,d}{\Diag}( \sa{\tilde{X}^{(\lambda_i)}} ) Y^\top  B \right) 
																		&=& \log \det \left( U \left( \mathsf{I}_\ys{d} - \frac{\theta}{d} \alpha^2  \underset{i=1,...,d}{\Diag}( \sa{\tilde{X}_{11}^{(\lambda_i)}}) \right)  U^\top    
																		\right) \\
																		&=& \sum_{i=1}^d \log \left(1-\frac{\theta}{d}\alpha^2  \tilde{X}_{11}^{\mgbis{(\lambda_i)}}\right).
																		\eeqs
																		Inserting this expression into \eqref{eq-risk in terms of hat x}, we conclude.
																		\section{Proofs of \cref{proposition-rate-function} and \cref{coro-risk-index-as-onedim-integral}}
																		\subsection{Proof of \cref{proposition-rate-function}}\label[appendix]{sec-prop-rate-function}
																		From \eqref{entropy-integral}, we recall that for $\theta/d < \ys{1/}H_\infty^2$, we have
																		\beq  R(\theta) &=& -\frac{\sigma^2}{2\pi \theta} \int_{-\pi}^\pi \log \det \left(\mathsf{I}_\ys{d} - \mgb{\frac{\theta}{d}} G(e^{i\omega}) G^*(e^{i\omega})\right) d\omega <\infty\,,
																		\label{eq-r-theta-representation}
																		\eeq
																		\mgbis{with the convention that $R(0) = \frac{\sigma^2}{2\pi d}\int_{-\pi}^\pi \operatorname{trace} \left( G(e^{i\omega}) G^*(e^{i\omega}) \right)d\omega$} where 
																		$H_\infty= \max_\omega \|G(e^{i\omega})\|$; and  $R(\theta)= +\infty$ if $\theta/d \geq \ys{1/} H_\infty^2$. We conclude that\footnote{{
																				\sa{$R(\theta)<\infty$} for $\theta \in (0,\frac{d}{H_\infty^2})$ also follows from \cref{thm-gmm-risk-formula}; however, that result does not apply to the $\theta<0$ regime.}}
																		\beq R(\theta)<\infty \mbox{ when } \theta \in (-\infty, \frac{d}{H_\infty^2}),  \quad \mbox{and} \quad R(\theta)=+\infty \mbox{ otherwise}.
																		\label{def-domain-risk-index}
																		\eeq
																		Next, we consider the following logarithmic moment generating function: 
																		$$\tilde{\Lambda}_K(\lambda) \coloneqq \log \E\left[e^{\lambda \tilde{S}_K}\right],\quad \mbox{where} \quad \tilde S_K = \frac{S_K}{K+1}.$$ 
																		It is known that the function $\tilde{\Lambda}_K(\lambda)$ is convex in $\lambda$ in its domain for every $K$.\footnote{Indeed, for any \( \bar\lambda_1, \bar\lambda_2 \in \mbox{dom}(\Lambda_k):=\{\lambda : \Lambda_k(\lambda)<\infty\} \), and \( \omega \in [0, 1] \), we have $
																			\tilde{\Lambda}_K(\omega \bar\lambda_1 + (1 - \omega)\bar\lambda_2)
																			= \log \mathbb{E}[e^{\omega \bar\lambda_1 \tilde{S}_K} e^{(1 - \omega)\bar\lambda_2 \tilde{S}_K}] \leq
																			\log \big( \mathbb{E}[e^{\bar\lambda_1 \tilde{S}_K}]^{\omega} \mathbb{E}[e^{\bar\lambda_2 \tilde{S}_K}]^{1-\omega}\big)
																			= 
																			\omega \tilde{\Lambda}_K(\bar\lambda_1) + (1 - \omega) \tilde{\Lambda}_K(\bar\lambda_2)$
																			where the inequality is due to H\"older's inequality. Thus, \( \tilde{\Lambda}_K(\lambda) \) is convex in $\lambda$.
																		} Recall that the risk-sensitive index is given by: $R(\theta) = \lim_{K\to\infty} \frac{2\sigma^2}{\theta (K+1)} \log \E\left[e^{\frac{\theta}{2\sigma^2} S_K}\right]$. Therefore, taking $\tilde S_K \coloneqq \frac1{K+1} S_K$ and $\lambda = \frac{\theta}{2\sigma^2}$, we observe that the limit,  
																		\beq
																		\tilde{\Lambda}(\lambda):=
																		\lim_{K\to\infty} \frac1{K+1}\log \E\left[e^{\lambda (K+1) \tilde S_K}\right] = 
																		\lambda R(2\sigma^2\lambda)\,, \label{def-Lambda-function}
																		\eeq
																		exists as an extended number for any $\lambda\in \ys{\mathbb{R}}$, {i.e., the limit is a real number or $+\infty$}. In addition, 
																		$\tilde\Lambda$ admits the domain 
																		\beq \mbox{dom}(
																		\tilde\Lambda):=\{ \lambda : \tilde{\Lambda}(\lambda)<\infty\} = \left(-\infty, \frac{d}{2\sigma^2 H_\infty^2}\right),
																		\label{def-dom-Lambda}
																		\eeq
																		where we used \eqref{def-domain-risk-index} and \eqref{def-Lambda-function}. 
																		Furthermore, $\tilde\Lambda$ is a convex function of $\lambda$ as it is the pointwise limit of convex functions $\frac{1}{K+1}\tilde{\Lambda}_K(\lambda)$.
																		G\"artner-Ellis theorem \cite[Section 2]{dembo2009large} says that if 
																		\begin{itemize}
																			\item [(I)] $0$ in the interior of the domain of $\tilde\Lambda$, i.e., $0 \in (\mbox{dom}(\tilde\Lambda))^{o}$,
																			\item [(II)]$\tilde\Lambda$ is lower semi-continuous,
																			\item [(III)] $\tilde\Lambda$ is \emph{essentially smooth}, i.e., it satisfies the following properties:
																			\begin{itemize}
																				\item  [(i)] The interior of the domain of $\tilde{\Lambda}(\lambda)$ is non-empty,
																				\item [(ii)] $\tilde{\Lambda}(\lambda)$ is differentiable in the interior of its domain,
																				\item [(iii)] $|{\tilde\Lambda}'(\lambda_n)| \to +\infty$ for all $\lambda_n \to \lambda^*$ where $\lambda^*$ is in the boundary of the interior of the domain,
																			\end{itemize}
																		\end{itemize}
																		then the LDP holds with the good rate function 
																		$$I(s):=\sup_{\lambda \in \text{dom}(\tilde\Lambda)} [\lambda s - \tilde{\Lambda}(\lambda)] =  \sup_{-\infty < \theta < \frac{d}{H_\infty^2}}\left[\frac{\theta}{2\sigma^2} \big(\ys{s} - R(\theta)\big)
																		\right],$$
																		which implies \eqref{eq-rate-function-quad} as desired. In the rest of the proof, we will argue that these conditions hold. By the assumption $\rho(A_Q)<1$, \cref{thm-h-inf} applies and $H_\infty<\infty$. Using this fact and \eqref{def-dom-Lambda}, we conclude that condition (I) and condition (i) holds. From the representation \eqref{eq-r-theta-representation}, we observe that $R(\theta)$ is continuous and differentiable in its domain; therefore, using \eqref{def-Lambda-function}, we conclude that $\tilde\Lambda$ is continuous and differentiable on its domain. \sa{Hence,} conditions (II) and (ii) both also hold. It suffices to argue that condition (iii) holds. Note that $\lambda_* = \frac{d}{2\sigma^2H_\infty^2}$ is the only boundary point for the domain of $\tilde\Lambda$ and
																		\beq \tilde\Lambda'(\lambda) = R(2\sigma^2\lambda) + 2\sigma^2\lambda \frac{dR_{\sigma^2}(\theta)}{d\theta} \bigg|_{\theta=2\sigma^2\lambda}.
																		\label{eq-Lambda-derivative}
																		\eeq
																		Let $\lambda_n$ be a sequence with $\lambda_n \to \lambda_*$. First of all, for $ \theta \in [0,d/H_\infty^2)$,  differentiating \eqref{eq-r-theta-representation} under the integral sign, we obtain
																		\beq   \frac{d R}{d\theta}(\theta) &=& \frac{\sigma^2}{2\pi} 
																		\sum_{k=1}^{\infty} \int_{-\pi}^\pi 
																		\frac{\theta^{\,k-1}}{k}\,
																		\operatorname{trace}\!\bigl(\Theta(\omega)^{k}\bigr)
																		d\omega \geq 0\,,
																		\label{ineq-non-negative-integral}
																		\eeq
																		with $\Theta(\omega) =  G(e^{i\omega})  G^*(e^{i\omega})/d\succeq 0$, where we used the Taylor series
																		\[
																		-\frac{\log\det\!\bigl(I-\theta X\bigr)}{\theta}
																		\;=\;
																		\sum_{k=1}^{\infty}\frac{\theta^{\,k-1}}{k}\,
																		\operatorname{trace}\!\bigl(X^{k}\bigr) \quad \mbox{when} \quad \|\theta X\|<1\,,
																		\]
																		for $X=\Theta(\omega)$, noting that $H_\infty=\max_\omega \|G(e^{i\omega})\|$ and $\|\theta \,\Theta(\omega)\|\leq \theta H_\infty^2 /d <1$.  
																		Similarly, differentiating \eqref{ineq-non-negative-integral} again, for $ \theta \in [0,d/H_\infty^2)$,
																		\beq  R''(\theta) = \frac{d^2 }
																		{d^2\theta}R(\theta) &=& 
																		\frac{\sigma^2}{2\pi} 
																		\sum_{k=2}^{\infty} \int_{-\pi}^\pi 
																		\frac{(k-1)\theta^{\,k-2}}{k}\,
																		\operatorname{trace}\!\bigl(\Theta(\omega)^{k}\bigr)
																		d\omega \geq 0\,.
																		\label{ineq-second-derivative-R-theta}
																		\eeq
																		With the non-negativity of the second derivative, we conclude that $R(\theta)$ is a non-decreasing convex function of $\theta$ when $\theta \in [0, \theta_*)$ where $\theta_* := d /H_\infty^2>0$. If we consider $\theta_n:= 2\sigma^2 \lambda_n$, then $\theta_n \uparrow \theta_*$. Consider $R_{max}:=\sup_n R(\theta_n)$. If $R_{\max}$ is finite, then by continuity of the function $R(\theta)$ in $\theta$, we would have $R(\theta_*) = R_{\max}$ and this would conflict with $R(\theta_*)=+\infty$. Therefore, we have necessarily $R_{\max}=+\infty$ and 
																		\beq 
																		R(\theta_n) = \tilde\Lambda(\lambda_n) \to +\infty.
																		\label{ineq-R-blowup}
																		\eeq
																		Similarly, for $\theta\in [0,\theta_*)$, by differentiating 
																		\eqref{ineq-second-derivative-R-theta} again, we can observe the third derivative $R'''(\theta)\geq 0$ and we conclude $R'(\theta)$ is a non-decreasing convex function of $\theta$. Similarly, if $R'_{\max}:=\sup_n R'(\theta_n)<\infty$, then by the continuity and monotonicity of $R'(\theta)$, we have $R'(\theta_*):=R'_{\max}$ and this would imply $R(\theta_*)\leq R(0) + R'_{\max} \theta_*$ which would contradict with the fact $R(\theta_*)=+\infty$. Hence, we \sa{necessarily have} 
																		\beq R'(\theta_n) \to +\infty.
																		\label{ineq-Rprime-theta}
																		\eeq
																		Therefore, combining \eqref{eq-Lambda-derivative}, \eqref{ineq-Rprime-theta} and \eqref{ineq-R-blowup}, we obtain 
																		$\tilde\Lambda'(\lambda_n) = R(\theta_n) + 2\sigma^2\lambda R'(\theta_n)\to \infty,
																		$
																		and this proves (iii) as desired and the LDP follows from the G\"artner-Ellis Theorem. To prove \eqref{eq-rate-function-interval}, note that by the Gaussian noise assumption, the density of $S_K/(K+1)$ is positive and continuous at any point $s \in (0,\infty)$; therefore, $I(s)<+\infty$ for any $s>0$ (see \cite[Ex. 2.2.39]{dembo2009large}. \sa{Thus,} the rate function $I$, as the dual of $\tilde\Lambda$, is continuous on its domain which includes $(0,\infty)$, and we conclude that for the set $E=[t,\infty]$ with $t>0$, $$\inf_{s\in E^o} I(s) = \inf_{s\in \bar{E}} I(s) = \inf_{s\in[t,\infty]} I(s), $$ and based on the LDP for $S_K/(K+1)$, this implies \ys{the equality in}
																		\eqref{eq-rate-function-interval}. \ys{The inequality in \eqref{eq-rate-function-interval} is a consequence of \eqref{convexity-averaging-ldp}.}
																		\subsection{Proof of \cref{coro-risk-index-as-onedim-integral}}\label[appendix]{appendix-proof-coro-risk-index-as-onedim-integral}
																		For momentum methods, \cite[eqn. (59)]{gurbuzbalaban2023robustly} shows that the transfer matrix $G(e^{i\omega})$ admits the explicit representation
																		\begin{equation} 
																			G(e^{i\omega})=\frac{1}{\sqrt{2}} \underset{j=1,...,d}{\Diag}\left[\frac{-\alpha \sqrt{\lambda_j}}{e^{i\omega}-(1+\beta-\beta e^{-i\omega})+\alpha\lambda_j (1+\nu-\nu e^{-i\omega})} \right] U^T,
																			\label{eq-transfer-fun-factor}
																		\end{equation}
																		as a diagonal matrix where $\{\lambda_j\}_{j=1}^d$ are the eigenvalues of $Q$, where we have the eigenvalue decomposition $Q = U \Lambda U^\top$ with $\Lambda$ diagonal satisfying $\Lambda_{jj}=\lambda_j$ and $U$ is a real orthogonal matrix satisfying $U^\top U=\mathsf{I}_d$. Thus, with some further straightforward computations, it can be seen that
																		{\small$$G(e^{i\omega}) G^*(e^{i\omega}) =  \underset{j=1,...,d}{\Diag}\left[\frac{\alpha^2 {\lambda_j}}{2\|e^{i\omega}-(1+\beta-\beta e^{-i\omega})+\alpha\lambda_j (1+\nu-\nu e^{-i\omega}) \|^2} \right] =  \underset{j=1,...,d}{\Diag} \left[ h_{\omega}(\lambda_j) \right].$$}
																		Then, by exploiting the diagonal structure of $G(e^{i\omega}) G^*(e^{i\omega})$, we obtain the equalities
																		$$ \displaystyle \int_{-\pi}^\pi \log \det \left(\mathsf{I}_d - \mgb{\frac{\theta}{d}} G(e^{i\omega}) G^*(e^{i\omega})\right) d\omega =  \displaystyle \sum_{j=1}^d \displaystyle \int_{-\pi}^\pi \log\left(1 -\frac{\theta}{d}  h_\omega(\lambda_j)\right) d\omega,
																		$$
																		$$\displaystyle \int_{-\pi}^\pi \operatorname{trace} \left( G(e^{i\omega}) G^*(e^{i\omega}) \right)d\omega
																		= \displaystyle\sum_{j=1}^d\displaystyle \int_{-\pi}^\pi   h_{\omega}(\lambda_i)
																		d\omega,
																		$$
																		showing that \eqref{eq-entropy-integral} is equivalent to \eqref{eq-entropy-integral-bis}. This completes the proof.
																		
																		\section{Proof of \cref{thm-str-cvx-risk-cost-bound}}\label{proof-risk-sensitive-cost-bound-GMM}
																		\mgb{Let $\theta>0$ be given such that
																			$\sqrt{\theta}\overline{H}_\infty < \ys{1}$. 
																			\sa{
																				For the GMM iterate sequence, consider the partial sum sequence defined as}
																			\begin{equation}
																				\bar S_{k,K} = \sum_{j=k}^K \left[J_{p,q}\mgb{V}(\xi_j) + (1-J_{p,q})\mgb{V}(\xi_{j-1})\right] \quad \mbox{for} \quad 0\leq k\leq K, \label{ineq-bound-SK}
																			\end{equation}
																			where here and in the rest of the proof, we drop the subscripts $c_1, P$ in the Lyapunov function for notational simplicity, and we use the convention that $\xi_{-1}:=\xi_0$.
																		} \mgbis{The sum \eqref{ineq-bound-SK} is based on a modified Lyapunov function that takes convex averages with $V(\xi_{j-1})$ at time $k$.} \mg{Note that when $\tilde{P}_{22} >0$, by a Schur complement argument, we can write  
																			$$\tilde P
																			\sa{=}
																			\begin{bmatrix} \tilde{P}_{11}-\frac{\tilde{P}_{12}^2} {\tilde{P}_{22}} & 0 \\
																				0 & 0 
																			\end{bmatrix} +
																			\begin{bmatrix}  \tilde{P}_{12}^2/\tilde{P}_{22} & \tilde{P}_{12} \\
																				\tilde{P}_{12} & \tilde{P}_{22} 
																			\end{bmatrix} 
																			\succeq 
																			\begin{bmatrix} \tilde{P}_{11}-\frac{\tilde{P}_{12}^2} {\tilde{P}_{22}} & 0 \\
																				0 & 0 
																			\end{bmatrix} 
																			=\begin{bmatrix} \tilde{r}(\tilde{P}) & 0 \\
																				0 & 0 
																			\end{bmatrix},
																			$$
																			and when $\tilde{P}_{22}=0$, we have $\tilde P_{12}=0$ since $\tilde{P}\succeq 0$. Therefore, using $P=\tilde{P}\otimes \mathsf{I}_d$, for any $j\geq 0$,
																			{\small
																				\begin{eqnarray*}V(\xi_j) = c_1 [f(x_j)-f(x_*)] + \begin{bmatrix} {x}_j - x_* \\ {x}_{j-1} - x_*  \end{bmatrix}^\top  P\begin{bmatrix} {x}_j - x_* \\ {x}_{j-1} - x_*  \end{bmatrix} &\geq& c_1[f(x_j)-f(x_*)] +\tilde{r}(\tilde{P}) \|x_j-x_*\|^2 \\
																					&\geq& (c_1 + \frac{2}{L}\tilde{r}(\tilde{P}))[f(x_j)-f(x_*)],
																				\end{eqnarray*}
																			}%
																			where the last inequality follows from the $L$-Lipschitzness of $\nabla f$.
																		}
																		\mg{Hence,} \mgbis{by the definition of $S_K$ given in \eqref{def-risk-cost},} \sa{we get}
																		\beq \frac{\bar S_{0,K}\mg{-(1-J_{p,q})V(\xi_0)}}{c_1 + \frac2L \tilde{r}(\tilde{P})} \geq 
																		\left(\mgb{S_{K-1}+ J_{p,q}[f(x_\ys{K}) - f(x_*)]}
																		\right)
																		\mgb{\geq J_{p,q} S_K}.
																		\label{ineq-SK-upperbound}
																		\eeq
																		\mgb{Let $x^{+}, x, x^{-} \in \mathbb{R}^d$ be given arbitrary vectors.} \mgb{We consider the $2d$-dimensional vectors,
																			\beq  \xi^+:=
																			\begin{bmatrix} x^+ \\ 
																				x\end{bmatrix}\,, \quad \xi := \begin{bmatrix} x \\ 
																				x^{-}\end{bmatrix}\,,
																			\label{def-xi-xiplus}
																			\eeq
																		}
																		\mgb{and for $0\leq k\leq K$, we introduce the deterministic functions$\mathcal{V}_{k,K}$ defined as}
																		
																		\[ 
																		\mathcal V_{k,K}(\mgb{\xi^+}, \xi;\bar{\theta})
																		\coloneqq
																		\frac1{\bar\theta} \log \E\left[e^{\bar\theta \bar S_{k,K}} \mid \xi_k = \mgb{\xi^+}, \xi_{k-1} = \xi \right]
																		\quad \mbox{where} \quad
																		\bar \theta \coloneqq \frac{\theta}{\mgb{2\sigma^2}\left(c_1 + \frac2L \tilde{r}(\tilde P)\right)}.
																		\]
																		Since $\theta \in \left(0,\frac{1}{(\overline{H}_\infty)^2}\right)$, we notice that \beq \bar{\theta} \in (0,\bar{\theta}^{ub}) \quad\mbox{with}\quad \bar{\theta}^{ub}:=\frac{1}{(\overline{H}_\infty)^2 \ys{2\sigma^2} \left(c_1 + \frac2L \tilde{r}(\tilde P)\right)}.
																		\label{ineq-bar-theta}
																		\eeq
																		\mgb{From \eqref{ineq-SK-upperbound} and the definition of the risk-sensitive cost, we can infer that\footnote{We use the convention that the right-hand side of \eqref{ineq-risk-bound-q-zero-1} is $+\infty$, when $\frac{\bar \theta}{J_{p,q}}\geq \bar{\theta}^{ub}$; in which case the bound \eqref{ineq-risk-bound-q-zero-1} is trivial.}
																			\beq 
																			R_{K}(\theta) &\leq& \frac{1}{ J_{p,q}} \frac{1}{\left(c_1 + \frac2L \tilde{r}(\tilde P)\right)\mg{(K+1)}} \left[\mathcal{V}_{0,K}\left(\xi_0, \xi_{-1};\frac{\bar{\theta}}{J_{p,q}}\right)\mg{-(1-J_{p,q})V(\xi_0)}\right]\,, 
																			\label{ineq-risk-bound-q-zero-1}\\
																			R_{K-1}(\theta) &\leq& \frac{1}{\left(c_1 + \frac2L \tilde{r}(\tilde P)\right)\mg{K}} \left[ \mathcal{V}_{0,K}\left(\xi_0, \xi_{-1};\bar{\theta}\right)
																			\mg{-(1-J_{p,q})V(\xi_0)}
																			\right]\,. 
																			\label{ineq-risk-bound-q-zero-2}
																			\eeq
																		}\mgb{In the rest of the proof, we will obtain bounds for the function $\mathcal{V}_{k,K}$ with a backwards induction over $k\in\{0,1,\dots,K\}$ and build on the last two inequalities. More specifically, we will show that for every $0\leq k\leq K$, and for any choice of the vectors $x^+,x, x^-\in\mathbb{R}^d$, we have 
																			\beq  \mathcal{V}_{k,K}(\xi^+,\xi;\bar{\theta}) \leq a_{k,K} V(\xi^+) + b_{k,K} V(\xi) + c_{k,K}(\bar\theta)\,,
																			\label{ineq-to-prove-induction}
																			\eeq for suitably chosen scalars $a_{k,K},b_{k,K}$ and $c_{k,K}$ (that may depend on the choice of $\bar \theta$) where $\xi^+$ and $\xi$ are as in \eqref{def-xi-xiplus}. Clearly, such a bound holds for $k=K$ with $a_{K,K} = J_{p,q}, b_{K,K} = 1-J_{p,q}$ and $c_{K,K}(\bar \theta) = 0$. Next, we consider $k < K$, and introduce the next GMM iterate starting from $\xi_k = \xi^+$: 
																			$$ \xi^{++} := A \xi^+ +B\nabla f(C\xi^+) +  B w_{\ys{k+1}}.$$
																		}
																		Note that by the definition of $\mathcal{V}_{k,K}$, 
																		$$
																		e^{\bar{\theta} \mathcal{V}_{k,K}(\xi^+, \xi;\bar\theta)}
																		= 
																		\E \left[ e^{ \bar{\theta} \left[\left\{J_{p,q} V(\xi^+) + (1-J_{p,q}) V(\xi)\right\} + \mathcal{V}_{k+1}(\xi^{++}, \xi^+\ys{; \bar\theta}) \right]} 
																		\mid
																		\xi_k = \xi^+, \xi_{k-1} = \xi
																		\right]\,,
																		$$
																		\mgb{where the expectation is taken with respect to randomness introduced by the noise $w_{k+1}$}, and, by Theorem \ref{thm-hinfty-bound}, we have 
																		\beq
																		V(\xi_{k+1})
																		\leq 
																		p V(\xi_k)
																		+
																		q V(\xi_{k-1}) 
																		+
																		r\|w_{\ys{k+1}}\|^2,
																		\label{ineq-almost-sure}
																		\eeq
																		\mgb{holding almost surely} where $p,q,r$ are as in \eqref{lyapunov-descent}. We then invoke the inductive hypothesis, where we assume \eqref{ineq-to-prove-induction} holds at step $k+1$, and use the previous inequality:
																		\begin{align*}
																			\mathcal V_{k+1,K}(\xi^{++}, \xi^+;\bar{\theta}) 
																			&\leq 
																			a_{k+1,K} V(\xi^{++}) + b_{k+1,K}V(\xi^{+}) + c_{k+1,K}(\bar\theta) \\
																			&\leq 
																			a_{k+1,K}[pV(\xi^+) + q V(\xi) + r\|w_{\ys{k+1}}\|^2] + b_{k+1,K}V(\xi^+) + c_{k+1,K}(\bar\theta) \\
																			&\leq
																			(a_{k+1,K}p + b_{k+1,K})V(\xi^{\ys{+}})
																			+
																			a_{k+1,K}q V(\xi)
																			+
																			a_{k+1,K}r\|w_{\ys{k+1}}\|^2 + c_{k+1,K}(\bar\theta).
																		\end{align*}
																		This leads to
																		{\small
																			\begin{align*}
																				e^{\bar\theta \mathcal{V}_{k,K}(\xi^+, \xi;\bar\theta)}
																				&=
																				\E\Bigg[
																				\exp\Bigg\{
																				\bar\theta \Big[
																				\big\{J_{p,q} V(\xi^+) + (1 - J_{p,q}) V(\xi)\big\} 
																				+ \mathcal{V}_{k+1,K}(\xi^{++}, \xi^+\ys{; \bar\theta})
																				\Big]
																				\Bigg\}
																				\,\Big|\, \xi_k = \xi^+,\, \xi_{k-1} = \xi
																				\Bigg] \\[1em]
																				&\leq 
																				\E\Bigg[
																				\exp\Bigg\{
																				\bar\theta \Big[
																				\big\{J_{p,q} V(\xi^+) + (1 - J_{p,q}) V(\xi)\big\}
																				+ \big(
																				(a_{k+1,K} p + b_{k+1,K}) V(\xi^+)
																				+ a_{k+1,K} q V(\xi) \\
																				&\hspace{10em}
																				+ a_{k+1,K} r \|\ys{w_{\ys{k+1}}}\|^2
																				+ c_{k+1,K}(\bar\theta)
																				\big)
																				\Big]
																				\Bigg\}
																				\,\Big|\, \xi_k = \xi^+,\, \xi_{k-1} = \xi
																				\Bigg] \\[1em]
																				&=
																				\E\Bigg[
																				\exp\Bigg\{
																				\bar\theta \Big[
																				\big(J_{p,q} + a_{k+1,K} p + b_{k+1,K}\big) V(\xi^+)
																				+ \big(1 - J_{p,q} + a_{k+1,K} q\big) V(\xi) \\
																				&\hspace{10em}
																				+ a_{k+1,K} r \|\ys{w_{\ys{k+1}}}\|^2
																				+ c_{k+1,K}(\bar\theta)
																				\Big]
																				\Bigg\}
																				\,\Big|\, \xi_k = \xi^+,\, \xi_{k-1} = \xi
																				\Bigg] \\[1em]
																				&=
																				\exp\Bigg\{
																				\bar\theta \Big[
																				\big(J_{p,q} + a_{k+1,K} p + b_{k+1,K}\big) V(\xi^+)
																				+ \big(1 - J_{p,q} + a_{k+1,K} q\big) V(\xi)
																				\Big]
																				\Bigg\} \\
																				&\quad\times 
																				\exp\Bigg\{
																				\bar\theta \Bigg[
																				c_{k+1,K}(\bar\theta)
																				+ \frac{1}{\bar\theta} \log \E\left[
																				e^{\frac{\tilde{t}_{k+1,K} \|\ys{w_{\ys{k+1}}}\|^2}{2\sigma^2}}
																				\,\Big|\, \xi_k = \xi^+, \ys{\xi_{k-1}=\xi}
																				\right]
																				\Bigg]
																				\Bigg\} \\[1em]
																				&\leq
																				\exp\Bigg\{
																				\bar\theta \Big[
																				\big(J_{p,q} + a_{k+1,K} p + b_{k+1,K}\big) V(\xi^+)
																				+ \big(1 - J_{p,q} + a_{k+1,K} q\big) V(\xi)
																				\Big]
																				\Bigg\} \\
																				&\quad\times
																				\ys{
																					\exp\Bigg\{
																					\bar\theta \Bigg[
																					c_{k+1,K}(\bar\theta)
																					+ \frac{1}{\bar\theta} \log\left(
																					\frac{1 + \tilde{t}_{k+1,K}}{1 - \tilde{t}_{k+1,K}}
																					\right)
																					\Bigg]
																					\Bigg\}
																				}\,,\end{align*}}
																		where $\mgb{\tilde t_{k+1,K}} \coloneqq \mgb{2\sigma^2} \bar\theta a_{k+1,K} r$, \mgb{provided that $0\leq \tilde t_{k+1,K}< \ys{1}$. In the last inequality, we used \eqref{ineq-cumulant-bound}.} 
																		This implies that \eqref{ineq-to-prove-induction} indeed holds \mg{for \ys{$a_{k,K}, b_{k,K}, c_{k,K}$} values that satisfy the following backwards recursion} 
																		{
																			\beq
																			a_{k,K} &=& J_{p,q} + a_{k+1,K}p + b_{k+1,K}, \quad
																			b_{k,K} = 1-J_{p,q} + a_{k+1,K}q, \label{defn-recurrence-relation-1}\\
																			c_{k,K}(\bar\theta) &=& c_{k+1,K}(\bar\theta) \ys{+ \frac1{\bar\theta}\log\left(\frac{1+\tilde{t}_{k+1,K}}{1-\tilde{t}_{k+1,K}}\right)}\,,\quad 
																			\label{defn-recurrence-relation-2}
																			\eeq}for $0\leq k\leq K-1$ with the boundary condition $a_{K,K} = J_{p,q}, 
																		b_{K,K} = 1-J_{p,q}, c_{K,K}(\bar\theta)= 0$; as long as the condition
																		\beq 0 \leq \tilde{t}_{k+1,K} < \ys{1} \quad \mbox{for every} \quad 0\leq k\leq K-1,
																		\label{cond-on-finite-mgf-tk-bound}
																		\eeq
																		holds.
																		Indeed, by \cref{lem-recurrence-soln}, if we choose 
																		$J_{p,q} = \frac{1}{1+(\lambda_+-p)}$, then the solution $a_{k,K}$ to the recurrence \eqref{defn-recurrence-relation-1} satisfies 
																		$a_{k,K} \leq \frac{1}{1 - (p+q)}$ for all $0\leq k\leq K$. Using this inequality, the bound \eqref{ineq-bar-theta} and the formula \eqref{def-Hinfty-ub}, for $0\leq k\leq K-1$, we have 
																		\begin{align}
																			\tilde{t}_{k+1,K}
																			=
																			2\sigma^2 \bar\theta a_{k+1,K} r
																			< \ys{\frac{1}{\left(c_1 + \frac2L \tilde{r}(\tilde P) \right) (\overline{H}_\infty)^2 [1 - (p+q)]}r = 1}.
																			\label{mgf-control-q-nonzero-case}
																		\end{align}
																		Therefore, the condition \eqref{cond-on-finite-mgf-tk-bound} indeed holds. 
																		From \eqref{formula-ak-bk} \mgbis{which implies that $a_{0,K}+b_{0,K} =\frac{1-\lambda_+^{K+1}}{1-\lambda_+}$}, the $c_{k,K}$ update rule \eqref{defn-recurrence-relation-2} and the fact that $\xi_{-1}=\xi_{0}$, we can conclude that
																		\beq 
																		\mathcal{V}_{0,K}(\xi_0,\xi_{-1};\bar{\theta}) &\leq& a_{0,K} V(\xi_0) + b_{0,K} V(\xi_0) +  c_{0,K}(\bar \theta) \\
																		&=& \ys{\left(\frac{1-\lambda_+^{K+1}}{1-\lambda_+}\right)V(\xi_0)} +\hat{c}(K, \theta),\eeq
																		with
																		\beq 
																		\ys{\hat{c}(K, \theta) 
																			\coloneqq 
																			c_{0,K}(\bar\theta)
																			=
																			\frac{\mgb{2\sigma^2}\left(c_1 + \frac2L \tilde{r}(\tilde P)\right)}{\theta}
																			\sum_{j=0}^{K-1} \log\left(\frac{\left(c_1 + \frac2L \tilde{r}(\tilde P)\right) + \theta a_{j+1,K}r}{\left(c_1 + \frac2L \tilde{r}(\tilde P)\right) - \theta a_{j+1,K}r}\right).}  
																		\label{eq-ck}
																		\eeq
																		Combining this inequality with \eqref{ineq-risk-bound-q-zero-1}-\eqref{ineq-risk-bound-q-zero-2}, we finally obtain
																		the bound \eqref{def-risk-bound-bar-RK}. This completes the proof. 

																		\section{Supporting Lemmas}
																		\begin{lemma}\label[lemma]{lemma-mgf-bound} Under \cref{assump-exp-moment-det-additive}, 
																			for every $k\geq 0$, we have the almost sure bound
																			\beq 
																			\sa{
																				\log \mathbb{E}\left[e^{\frac{t \|\ys{w}_{k+1}\|^2}{2\sigma^2}} \mid \mathcal F_k \right] \leq
																				\log\left(\frac{1+t}{1-t}\right) \quad \mbox{for} \quad 0\leq t< t_*:=1.}
																			\label{ineq-cumulant-bound}
																			\eeq
																		\end{lemma}
																		\begin{proof}
																			\ys{
																				Under \cref{assump-exp-moment-det-additive}, \mgb{for any $\tilde{p}\geq 1$},
																				\begin{align*}
																					\E \left[\|\ys{w}_{k+1}\|^{\tilde{p}} \mid \mathcal F_{k} \right]
																					&=
																					\int_0^\infty \mathbb P(\|\ys{w}_{k+1}\|^{\tilde{p}} \geq t \mid \mathcal F_k)\sa{dt}
																					= \sa{\int_0^\infty \mathbb P(\|\ys{w}_{k+1}\| \geq t^{1/{\tilde{p}}} \mid \mathcal F_k) \,dt} \\
																					&\leq \sa{\int_0^\infty 2 \exp\Big(\frac{-t^{2/{\tilde{p}}}}{2\sigma^2}\Big) \, dt = (2\sigma^2)^{{\tilde{p}}/2} {\tilde{p}}\int_0^\infty e^{-s}s^{{\tilde{p}}/2 - 1} \,ds\, =(2\sigma^2)^{{\tilde{p}}/2} {\tilde{p}} \Gamma({\tilde{p}}/2),} 
																				\end{align*}
																				\sa{where in the second equality we used the change of variables $s:=\frac{t^{2/{\tilde{p}}}}{2\sigma^2}$, and the third equality follows from the definition of $\Gamma(\cdot)$ function.}
																				Thus, \sa{by the dominated convergence theorem, 
																					\begin{align*}
																						\E\left[e^{\frac{t\|\ys{w}_{k+1}\|^2}{2\sigma^2}} \mid \mathcal F_k\right]
																						&=
																						\E\left[
																						1 + \sum_{{\tilde{p}}=1}^\infty \frac{(t\|\ys{w}_{k+1}\|^2)^{\tilde{p}}}{(2\sigma^2)^{\tilde{p}}\tilde{p}!} \mid \mathcal F_k 
																						\right] 
																						=
																						1 + \sum_{{\tilde{p}}=1}^\infty \frac{t^{\tilde{p}}\E\left[\|\ys{w}_{k+1}\|^{2\tilde{p}} \mid \mathcal F_k\right]}{(2\sigma^2)^{\tilde{p}}\tilde{p}!} \\
																						&\leq 
																						1 + \sum_{{\tilde{p}}=1}^\infty \frac{2(2t\sigma^2)^{\tilde{p}} \tilde{p}!}{(2\sigma^2)^{\tilde{p}} \tilde{p}!}
																						= 
																						2\sum_{{\tilde{p}}=0}^\infty t^{\tilde{p}}-1 
																						=
																						\frac{2}{1 - t}-1=\frac{1+t}{1-t}\,,
																					\end{align*}
																					provided $|t| < 1$, this implies \eqref{ineq-cumulant-bound} and we conclude.} 
																			}
																		\end{proof}
																		
																		\begin{lemma}\label[lemma]{lem-recurrence-soln}
																			\mgb{Let $K\geq 1$ be a given positive integer that is fixed.} Consider \mgb{the backwards recurrence \eqref{defn-recurrence-relation-1} where the scalars $a_{k,K}\in\mathbb{R}$ and $b_{k,K}\in\mathbb{R}$ satisfy}
																			\beq
																			a_{k,K} &=& J_{p,q} + a_{k+1,K}p + b_{k+1,K}, \quad
																			b_{k,K} = 1-J_{p,q} + a_{k+1,K}q,
																			\eeq
																			for $0\leq k \leq \mgb{K}$ with $a_{K,K} =  J_{p,q}$, $b_{K,K} =  1-J_{p,q}$ where $p,q$ are \mgb{non-negative scalars as in \eqref{defn-pq} satisfying $p+q<1$}. Then, for any $0\leq k\leq K$,
																			$$\begin{bmatrix}
																				a_{k,K}\\
																				b_{k,K}
																			\end{bmatrix} = \left(\sum_{j=0}^{K-k} \overline{M}^j\right)  v_{p,q} \quad \mbox{where} \quad
																			\quad v_{p,q}= \begin{bmatrix}
																				J_{p,q}\\
																				1-J_{p,q}
																			\end{bmatrix},
																			\quad\mbox{and}
																			\quad
																			\overline{M} = \begin{bmatrix}
																				p & 1\\ q & 0
																			\end{bmatrix}.
																			$$
																			The eigenvalues of $\overline{M}$ are $\lambda_{+,-} = \frac{p \pm \sqrt{p^2+4q}}{2}$ with the largest eigenvalue $\lambda_+ \in [0,1)$. Furthermore, if \ys{we} set $J_{p,q} = \frac{1}{1+(\lambda_+-p)}
																			$, then $v_{p,q}$ is an eigenvector of $\overline{M}$ satisfying $\overline{M} v_{p,q}= \lambda_+ v_{p,q}$ and,
																			\beq \begin{bmatrix}
																				a_{k,K}\\
																				b_{k,K}
																			\end{bmatrix} = \frac{1-\lambda_+^{K-k+1}}{1-\lambda_+}
																			\begin{bmatrix}
																				\frac{1}{1+(\lambda_+-p)}\\
																				\frac{\lambda_+ - p}{1+(\lambda_+-p)}
																			\end{bmatrix}
																			=
																			\frac{1}{1-(p+q)} (1-\lambda_+^{K-k+1})
																			\begin{bmatrix}
																				1 \\ \lambda_+
																				-p
																			\end{bmatrix},
																			\label{formula-ak-bk}
																			\eeq
																			\mgb{for any $0\leq k\leq K$. Furthermore, we have \beq a_
																				{K,K} \leq a_{K-1,K}\leq \cdots \leq a_{1,K} \leq a_{0,K} \leq  \frac{1}{1-(p+q)}.\label{ineq-monotonicity} \eeq}
																		\end{lemma}
																		\begin{proof}For $K$ fixed, the backwards recursion over $k$ for $a_{k,K}$ and $b_{k,K}$ is equivalent to
																			$$
																			\begin{bmatrix}
																				a_{k,K}\\
																				b_{k,K}
																			\end{bmatrix}
																			=\overline{M} \begin{bmatrix}
																				a_{k+1,K}\\
																				b_{k+1,K}
																			\end{bmatrix} + v_{p,q}
																			\quad \mbox{for} \quad 
																			0\leq k\leq K-1,
																			$$
																			subject to boundary conditions  
																			$\begin{bmatrix}
																				a_{K,K} &
																				b_{\ys{K},K}
																			\end{bmatrix}^T = v_{p,q}$. It can be seen that the solution is simply
																			\beq \begin{bmatrix}
																				a_{k,K}\\
																				b_{k,K}
																			\end{bmatrix} = \left(\sum_{j=0}^{K-k} \overline{M}^j\right)  v_{p,q} \quad \mbox{for} \quad 
																			0\leq k\leq K-1.
																			\label{soln-M-powers}
																			\eeq
																			The characteristic equation $\det (\overline{M}-\lambda I)=\lambda^2 - p\lambda - q$ admits the roots $\lambda_{\pm}$ as the eigenvalues of $\overline{M}$. The fact that $\lambda_+\in [0,1)$ is a consequence of the inequalities $p,q\geq 0$ and $p+q<1$. For the choice of $J_{p,q} = \frac{1}{1+(\lambda_+-p)}
																			$, it is easy to check that $\overline{M}v_{p,q}= \lambda_+ v_{p,q}$. In this case, using \eqref{soln-M-powers}, 
																			$$\begin{bmatrix}
																				a_{k,K}\\
																				b_{k,K}
																			\end{bmatrix} = \left(\sum_{j=0}^{K-k} \lambda_+^j\right)  v_{p,q} = \frac{1-\lambda_+^{K-k+1}}{1-\lambda_+} v_{p,q}\,,$$
																			and the formula \eqref{formula-ak-bk} follows noting $(1-\lambda_+)(1+(\lambda_+-p))=1-(p+q)$. Finally, note that by \eqref{formula-ak-bk} and the fact that $\lambda_+ \in [0,1)$, for any $k<K$, we have
																			$$a_{k+1,K} - a_{k,K} =
																			\frac{1}{1-(p+q)} (\lambda_+^{K-k+1}- \lambda_+^{K-k})\leq 0.$$
																			This implies \eqref{ineq-monotonicity}, together with the fact that $a_{0,K} = \frac{1-\lambda_+^{K+1}}{1-(p+q)} \leq \frac{1}{1-(p+q)}$.
																			\end{proof}
																			
																			\begin{lemma}\label[lemma]{lemma-dual-of-base-function}
																				Consider the function $\phi:\mathbb{R} \times \mathbb{R}_+ \times \mathbb{R}_+ \to \mathbb{R}\cup \{+\infty\}$ defined as 
																				$$\phi(\theta,\check a, \check b) = \begin{cases}
																					\check{a} + 4\sigma^2 \check{b} & \mbox{if} \quad \theta = 0,\\
																					\check{a} \mg{+} \frac{2\sigma^2}{\theta}\log\left(\frac{1+\theta \check b}{1-\theta\check b}\right) & \mbox{if}\quad 0< \theta < {1}/{\check b},\\
																					+\infty & \mbox{if} \quad \theta \geq {1}/{\check b} \mbox{ or } \theta<0,
																				\end{cases}
																				$$
																				for $\sigma>0$ given and fixed. Then, for any $t\geq 0$,
																				\beq \Psi(t,\check a,\check b)&:=&
																				\sup_{\theta \in \mathbb{R}}  \frac{\theta}{2\sigma^2} (t- \phi(\theta, \check a, \check b)) 
																				=\sup_{0\leq \theta <\frac{1}{\check b}}  \frac{\theta}{2\sigma^2} (t- \phi(\theta, \check a, \check b)) \\
																				&=& 
																				\begin{cases}
																					\displaystyle \frac{t-\check{a}}{2\sigma^2 \check b} \sqrt{1 - \frac{4\sigma^2 \check b}{t-\check a}} - \log\left( \frac{1 + \sqrt{1 - \frac{4\sigma^2 \check b}{t-\check a}}}{1 - \sqrt{1 - \frac{4\sigma^2 \check b}{t-\check a}}} \right), & \text{if} \quad t -\check a \geq 4\sigma^2 \check b, \\
																					0, & \text{if}\quad 0\leq t -\check a < 4\sigma^2 \check b.
																				\end{cases}
																				\label{eq-rate-function-general-params}
																				\eeq
																				Also, for any $\check c>0$,
																				\beq 
																				\sup_{\theta \in \mathbb{R}}  \frac{\theta}{2\sigma^2} (t- \check c \phi(\theta, \check a, \check b)) 
																				= \check c \, \Psi\left(\frac{t}{\check c}, \check a, \check b\right).
																				\label{eq-scaling-property}
																				\eeq
																			\end{lemma}
																			\begin{proof} By the definition of $\phi(\theta,\check a,\check b)$, we have
																				$$\Psi(t,\check a,\check b)= \sup_{ 0 < \theta < \frac{1}{\check b}}
																				F (\theta,t,\check a,\check b):=\frac{\theta}{2\sigma^2} 
																				\left(
																				t -\check{a}
																				\mg{-} \frac{2\sigma^2}{\theta}\log\left(\frac{1+\theta \check b}{1-\theta\check b}\right)
																				\right).
																				$$
																				The derivative
																				\[
																				\frac{d}{d\theta}F (\theta,t,\check a,\check b) = \frac{t-\check a}{2\sigma^2} - \frac{2\check b }{1 - \left(\theta \check b\right)^2},
																				\]
																				\mgbis{is negative for $0\leq t-a< 4\sigma^2 \check b$ and $0<\theta<1/\check{b}$ in which case we obtain $$\Psi(t,\check a,\check b)= F (0,t,\check a,\check b):= \lim_{\theta\downarrow 0}F (\theta,t,\check a,\check b)=0;$$ whereas for $t-a> 4\sigma^2 \check b$, the derivative is zero at}
																				$\theta = {\theta}_*:= 
																				\sqrt{1 - \frac{4\sigma^2 \check b}{t-\check a}}/{\check b}>0,
																				$
																				in which case, we get $\Psi(t,\check a,\check b)= F (\theta_*,t,\check a,\check b)$ which implies \eqref{eq-rate-function-general-params} as desired. The equality \eqref{eq-scaling-property} is obtained by a simple scaling argument. Indeed, by the positivity of $\check c$, 
																				$$
																				\sup_{\theta \in \mathbb{R}}  \frac{\theta}{2\sigma^2} (t- \check c\phi(\theta, \check a, \check b)) = \check c\sup_{\theta \in \mathbb{R}}  \frac{\theta}{2\sigma^2} \left(\frac{t}{\check c}-  \phi(\theta, \check a, \check b)\right) =  \check c \,\Psi\left(\frac{t}{\check c}, \check a, \check b\right).
																				$$
																			\end{proof}
																			\begin{lemma}\label[lemma]{lemma-finite-time-rate-function} In the setting of \cref{theorem-rate-function-finite-time}, the lower bound \eqref{my-large-deviation-bound-2} holds.
																			\end{lemma}
																			\begin{proof}
																				Because $a_{j+1,K}$ is non-increasing for $j\in \{0,\dots,K-1\}$, note that the function sequence
																				$$
																				\hat\phi_j(\theta):=
																				\begin{cases}
																					\frac{\mgb{2\sigma^2}}{\theta}
																					\log\left(\frac{\left(c_1 + \frac2L \tilde{r}(\tilde P)\right) + \theta a_{j+1,K}r}{\left(c_1 + \frac2L \tilde{r}(\tilde P)\right) - \theta a_{j+1,K}r}\right), & \text{if } 0\leq \theta \,(\overline{H}_\infty)^2 < 1,\\
																					+\infty & \mbox{else},
																				\end{cases} 
																				$$
																				with the convention 
																				$$\hat\phi_j(0) = \lim_{\theta\downarrow 0} \frac{\mgb{2\sigma^2}}{\theta}
																				\log\left(\frac{\left(c_1 + \frac2L \tilde{r}(\tilde P)\right) + \theta a_{j+1,K}r}{\left(c_1 + \frac2L \tilde{r}(\tilde P)\right) - \theta a_{j+1,K}r}\right)= 4 \sigma^2 a_{j+1,K}r\,,$$
																				is non-increasing, i.e., $\hat\phi_{j_1}(\theta) \geq \hat\phi_{j_2}(\theta)$ for every $\theta$ whenever $0\leq j_1\leq j_2\leq K-1$. Therefore, in the setting of \cref{thm-str-cvx-risk-cost-bound}, we have 
																				$$ \hat{c}(K,\theta)\leq K \left(c_1 + \frac2L \tilde{r}(\tilde P)\right) \hat \phi_0(\theta), $$
																				and thus from \cref{thm-str-cvx-risk-cost-bound}, we obtain
																				$$ \overline{R}_{K}(\theta) \leq \min \bigl(\check c_K^{(1)} \phi(\theta,\check a_K^{(1)}, \check  b_K^{(1)}),  \phi(\theta,\check a_K^{(2)}, \check b_K^{(2)})\bigr),
																				$$
																				where we used $a_{1,K} = \frac{1-\lambda_+^{K}}{1-(p+q)}$ implied by \eqref{formula-ak-bk} and the formula \eqref{def-Hinfty-ub} \ys{with $\phi(\theta,\cdot,\cdot)$ defined as in \cref{lemma-dual-of-base-function}}. Then, 
																				using \cref{lemma-dual-of-base-function}, we conclude that
																				{\small
																					\beqs
																					\bar{I}_K(t)&=&\sup_{0\leq \theta<\frac{1}{(\overline{H}_\infty)^2}} \frac{\theta}{2\sigma^2}\left(t-\overline{R}_{K}(\theta)\right) \\
																					&\geq&
																					\max\bigg\{
																					\sup_{0\leq \theta<\frac{1}{(\overline{H}_\infty)^2}} \frac{\theta}{2\sigma^2}\left(t-\check c_K^{(1)} \phi(\theta,\check a_K^{(1)}, \check  b_K^{(1)})\right), \\
																					&& \qquad \qquad  \qquad \qquad\qquad \qquad
																					\sup_{0\leq \theta<\frac{1}{(\overline{H}_\infty)^2}} \frac{\theta}{2\sigma^2}\left(t- \phi(\theta,\check a_K^{(2)}, \check  b_K^{(2)})\right)\bigg\} 
																					\\
																					&=& 
																					\max\left\{ \check c_K^{(1)}  \Psi\left(\frac{t}{\check c_K^{(1)}},\check a_K^{(1)} , \check b_K^{(1)} \right), \Psi\left(t,\check a_K^{(2)} , \check b_K^{(2)} \right)
																					\right\}.
																					\eeqs 
																				}
																			\end{proof}

																			\vskip 6mm
																			\noindent{\bf Acknowledgments} This research was supported in part by the Office of Naval Research under award number N00014-24-1-2628 and N00014-24-1-2666.
																			
																		\end{document}

\end{document}